\documentclass{amsart}
\usepackage{amsfonts, amssymb, latexsym, amscd}
\usepackage{amsmath, amsthm}
\usepackage[mathscr]{eucal}
\usepackage{enumerate}
\usepackage{fullpage}
\usepackage[pdftex]{graphicx, color}
\usepackage[usenames,dvipsnames]{xcolor}
\usepackage{dsfont}
\usepackage{bbm}

\theoremstyle{plain}
\newtheorem{theorem}{Theorem}[section]

\newtheorem{lemma}[theorem]{Lemma}
\newtheorem{proposition}[theorem]{Proposition}

\newtheoremstyle{named}{}{}{\itshape}{}{\bfseries}{.}{.5em}{\thmnote{#3's }#1}
\theoremstyle{named}

\theoremstyle{definition}
\newtheorem{definition}[theorem]{Definition}

\newtheorem{remark}[theorem]{Remark}

\newtheorem{claim}[theorem]{Claim}

\newcommand{\pref}[1]{(\ref{#1})}

\newcommand{\PP}{\mathbb{P}}
\newcommand{\ZZ}{\mathbb{Z}}
\newcommand{\RR}{\mathbb{R}}

\newcommand{\MM}{\mathcal{M}}
\newcommand{\EE}{\mathbb{E}}

\newcommand{\ve}{\varepsilon}

\newcommand{\vp}{\varphi}

\newcommand{\al}{\alpha}

\newcommand{\ga}{\gamma}

\newcommand{\ka}{\kappa}
\newcommand{\la}{\lambda}

\newcommand{\sig}{\sigma}

\newcommand{\om}{\omega}

\newcommand{\thet}{\theta}

\newcommand{\Ga}{\Gamma}

\newcommand{\La}{\Lambda}

\newcommand{\Om}{\Omega}
\newcommand{\norm}[1]{\left|\left|#1\right|\right|}
 
\newcommand{\rst}[1]{\ensuremath{{\mathbin\upharpoonright}%
\raise-.5ex\hbox{$#1$}}}

\newcounter{gscan}
\newcounter{btscan}
\newcounter{cscan}
\newcounter{hscan}
\newcounter{fscan}
\newcounter{pscan}
\newcounter{sscan}

\renewcommand{\thefscan}{F\arabic{fscan}}

\renewcommand{\thesscan}{S\arabic{sscan}}

\numberwithin{equation}{section}

\DeclareMathOperator*{\esssup}{ess\,sup}

\begin{document}

\title{On the Stochastic Homogenization of Fully Nonlinear Uniformly Parabolic Equations in Stationary Ergodic Spatio-Temporal Media}

\author{Jessica Lin}
\address{University of Chicago\\ 
Department of Mathematics\\
Chicago, IL 60637}
\email[Jessica Lin]{jessica@math.uchicago.edu}



\date{\today}

\thanks{This work was completed as a part of the author's doctoral thesis. The author would like to thank her thesis advisor, Takis Souganidis, for his patient guidance and many helpful discussions. The author was supported as a graduate student on NSF grant DGE-1144082.}

\begin{abstract}
We study homogenization for fully nonlinear uniformly parabolic equations in stationary ergodic spatio-temporal media from the qualitative and quantitative perspective. We show that under suitable hypotheses, solutions to fully nonlinear uniformly parabolic equations in spatio-temporal media homogenize almost surely. In addition, we obtain a rate of convergence for this homogenization in measure, assuming that the environment is strongly mixing with a prescribed rate. A general methodology to study the stochastic homogenization and rates of convergence for stochastic homogenization of uniformly elliptic equations was introduced in \cite{csw} and \cite{cs}. We extend their approach to fully nonlinear uniformly parabolic equations, and we develop a number of new arguments to handle the parabolic structure of the problem. 
\end{abstract}

\maketitle

%\tableofcontents

\section{Introduction}
We study the stochastic homogenization for fully nonlinear uniformly parabolic equations with oscillations in space and time. Specifically, we study the limiting behavior of solutions $u^{\ve}=u^{\ve}(x, t, \om)$ of the initial value problem
 \begin{equation} \label{hom}
\begin{cases}
u^{\ve}_{t}-F(D^{2}u^{\ve}, x/\ve, t/{\ve}^{2}, \om)=0 \quad\text{in}\quad D_{T},\\
u^{\ve}=g \quad\text{on}\quad\partial_{p} D_{T},
\end{cases}
\end{equation}
\noindent where $D\subset \RR^{d}$ is an open, bounded domain, $D_{T}=D\times (0, T]$, $\partial_{p}D_{T}=(\overline{D}\times\left\{0\right\})\cup(\partial D\times[0,T))$ denotes the parabolic boundary which satisfies a uniform exterior cone condition, and $g\in C(\partial_{p}D_{T})$. The argument $\om$ describes an environment which is an element of the underlying probability space $(\Om, \mathcal{F}, \PP)$. We postpone the precise assumptions on \pref{hom} until later (see \pref{f1}-\pref{f5}), but we point out the key hypotheses at this time. For each $\om\in\Om$, $F(\cdot, \cdot, \cdot, \om)$ is uniformly elliptic and satisfies the necessary assumptions for well-posedness in the theory of viscosity solutions. With regards to the random setting, we assume that $F$ is stationary in space-time, and that translations in space-time are ergodic with respect to the law of $F(\cdot, \cdot, \cdot, \om)$. 

The first result pertains to a qualitative observation about the limiting behavior of $u^{\ve}$. We show that almost surely, $u^{\ve}$ converge uniformly to a limiting function $u=u(x,t)$, which solves a deterministic limiting equation.  The limiting equation takes on the form 
 \begin{equation}\label{limit}
\begin{cases}
u_{t}-\overline{F}(D^{2}u)=0 \quad\text{in}\quad D_{T}, \\
u=g \quad\text{on}\quad\partial_{p} D_{T}, 
\end{cases}
\end{equation}
where we construct $\overline{F}$, a real-valued function on $\mathbb{S}^{d}$, the space of $d\times d$ symmetric matrices with real entries. More precisely, we prove
\begin{theorem}\label{stochhom}
Assume \pref{f1}-\pref{f4}. There exists a uniformly elliptic operator $\overline{F}: \mathbb{S}^{d}\rightarrow \RR$, such that if $u^{\ve}(\cdot, \cdot, \om)$ is the solution of \pref{hom}, then as $\ve\rightarrow 0$, a.s. in $\om$, $u^{\ve}\rightarrow u$ uniformly, where $u$ is the unique solution of \pref{limit}. 
\end{theorem}

Furthermore, we establish a rate of convergence for Theorem \ref{stochhom} ``in measure." In order to quantify the rate of convergence of Theorem \ref{stochhom}, it is necessary to quantify the ergodicity hypothesis by assuming that $F$ is strongly mixing with a prescribed rate. In broad terms, the rate of mixing measures the ``decorrelation" of different events of the probability space. Under our hypotheses, we find $f(\ve), h(\ve)$ such that 
\begin{equation*}
\PP\left[\left\{\om: \norm{u^{\ve}(\cdot, \cdot, \om)-u(\cdot, \cdot)}_{L^{\infty}(D_{T})}\geq f(\ve)\right\}\right]\leq h(\ve)
\end{equation*}
where $\lim_{\ve\rightarrow 0} f(\ve)=\lim_{\ve\rightarrow 0} h(\ve)=0$. This gives us a rate of convergence for homogenization in measure. We assume a logarithmic rate of mixing in order to obtain a logarithmic rate for the homogenization of \pref{hom}. The main result we prove regarding the rate of convergence for homogenization is

\begin{theorem}\label{thmfinalrate}
Assume \pref{f1}-\pref{f5}. Let $u^{\ve}, u$ solve \pref{hom} and \pref{limit} respectively. There exist constants $\tilde{C}, \tilde{c}, \hat{C}, \hat{c}$, and $\ve_{0}$ which depend on the ellipticity constants and dimension so that for all $\ve<\ve_{0}$, 
\begin{equation}
\PP[\norm{u^{\ve}-u}_{L^{\infty}(D_{T})}\geq \tilde{C}\ve^{\tilde{c}|\ln\ve|^{-2/3}}]\leq \hat{C}\ve^{\hat{c}|\ln\ve|^{-2/3}}.
\end{equation}
\end{theorem}

The qualitative and quantitative stochastic homogenization for fully nonlinear uniformly elliptic equations was addressed in \cite{csw} and \cite{cs} respectively, and our paper is an extension of their work to the class of fully nonlinear uniformly parabolic equations with oscillations in space and time. Theorem \ref{stochhom} is the parabolic analogue of the main result of \cite{csw}, and the rate in Theorem \ref{thmfinalrate} is comparable to the rate obtained in \cite{cs}. We point out that the strategies to prove Theorems \ref{stochhom} and \ref{thmfinalrate} are similar to those of \cite{csw} and \cite{cs}, however there are several difficulties which arise given the parabolic structure of the problem. The proofs in \cite{csw} and \cite{cs} rely upon delicate applications of the regularity theory for fully nonlinear elliptic equations.  Therefore, extending these results to a different problem requires verifying that these arguments still hold using the regularity estimates available for the problem at hand. For fully nonlinear uniformly parabolic equations, these properties have been developed in a number of works, including (but not limited to) those of Krylov \cite{krylov, krylovabp}, Tso \cite{tso}, Wang \cite{wangreg1, wangreg2}, Lin \cite{linreg1}, and Turanova \cite{olga1}.  A consistent issue in our analysis is having to deal with the natural causality inherent in parabolic equations. Parabolic regularity estimates depend sensitively on time; they are not always merely $\RR^{d+1}$ generalizations of elliptic results in $\RR^{d}$. This dependency on time causes several of the arguments of \cite{csw} and \cite{cs} to break down without additional justification. This paper addresses these issues which arise in the parabolic setting. 

While we were writing this, it came to our attention that Armstrong and Smart \cite{asellip} have very recently developed a different approach to the stochastic homogenization of fully nonlinear uniformly elliptic equations. Their method yields an algebraic rate of convergence for homogenization under the assumption that the environment is independent, and identically distributed, which is a stronger mixing hypothesis than ours. Moreover, it is unclear how to adapt their approach to the parabolic setting, making our results independent of their work. 

We next aim to motivate and outline the proofs of Theorems \ref{stochhom} and \ref{thmfinalrate}. 
\medskip

\noindent \textbf{Homogenization Strategy.} Most of the arguments in the theory of homogenization are based on the ansatz that $u^{\ve}(\cdot, \cdot,\om)$ takes on a multiscale expansion
\begin{equation*}
u^{\ve}(x,t,\om)=u(x,t, \om)+\ve^{2}w^{\ve}(x/\ve,t/\ve^{2},\om)+\ldots
\end{equation*}
and letting $y=x/\ve$, $s=t/\ve^{2}$, this formally implies that 
\begin{equation*}
u_{t}+w^{\ve}_{s}-F(D^{2}_{x}u+D^{2}_{y}w^{\ve}, y,s, \om)=0 \quad\text{in}\quad D_{T}.
\end{equation*}
If homogenization is to occur, the former expression must be independent of $y$ and $s$. This naturally leads to the following compatibility condition: for each $M\in \mathbb{S}^{d}$, there exists a unique constant $\overline{F}(M)$ and a continuous function $w^{\ve}$ satisfying
\begin{equation}\label{badcell}
\begin{cases}
w^{\ve}_{s}-F(M+D^{2}_{y}w^{\ve}, y, s, \om)=-\overline{F}(M)\quad\text{in}\quad \RR^{d+1},\\
\frac{1}{(|y|^{2}+|s|)}w^{\ve}(y,s,\om)\rightarrow 0\quad\text{as}\quad |y|, |s|\rightarrow\infty,\quad\text{a.s. in}\quad \om,
\end{cases}
\end{equation}
 where the second condition is the necessary and sufficient condition for $\overline{F}(M)$ to be uniquely determined. 
 
However, it is unknown whether solutions to \pref{badcell} exist. The fundamental observations of \cite{csw} showed it is possible to homogenize by reformulating the compatibility condition imposed on $w^{\ve}$ and $\overline{F}$. Indeed, it is enough if we require that for each $M\in \mathbb{S}^{d}$, there exists a unique constant $\overline{F}(M)$ and a continuous solution $w^{\ve}$ which solves
\begin{equation}\label{correct}
\begin{cases}
w^{\ve}_{s}-F(D^{2}_{y}w^{\ve}+M,y,s,\om)=-\overline{F}(M) \quad\text{in}\quad Q_{1/{\ve}},\\
w^{\ve}(y,s)=0\quad\text{on}\quad \partial_{p}Q_{1/{\ve}},\\
\end{cases}
\end{equation}
with the condition that $\norm{\ve^{2}w^{\ve}(\cdot, \cdot, \om)}_{L^{\infty}(Q_{1/\ve})}\rightarrow 0$ as $\ve\rightarrow 0$, a.s. in $\om$, where $Q_{r}=B_{r}(0)\times (-r^{2}, 0]$.  We refer to $w^{\ve}$ as the ``approximate corrector" of \pref{hom}. 

The above problem can also be reformulated as follows. Letting
\begin{equation*}
w_{\ve}(y, s, \om)=\ve^{2}w^{\ve}(y/\ve, s/\ve^{2}, \om),
\end{equation*}
and for any $M\in\mathbb{S}^{d}$, setting
\begin{equation*}
F_{M}(\cdot, y,s, \om)=F(\cdot+M, y,s, \om),
\end{equation*}
then \pref{correct} is equivalent to $w_{\ve}$ solving
\begin{equation}\label{fbar}
\begin{cases}
w_{\ve,s}-F_{M}(D^{2}_{y}w_{\ve}, y/\ve, s/\ve^{2},\om)=-\overline{F}(M) \quad\text{in}\quad Q_{1},\\
w_{\ve}=0\quad\text{on}\quad \partial_{p}Q_{1},
%w^{\ve}=0 \hspace{1cm}\text{in}\hspace{1 cm}\partial C_1(\hat{x},\hat{t})
\end{cases}
\end{equation}
with $\norm{w_{\ve}(\cdot, \cdot, \om)}_{L^{\infty}(Q_{1})} \rightarrow 0$ as $\ve\rightarrow 0$, a.s. in $\om$.  

We will frequently change our perspective, by working with $w^{\ve}$ or $w_{\ve}$, and we will try to continually remind the reader of their definitions. Some of the arguments in this paper are much more convenient to explain using one of the two perspectives, so we simply choose the most convenient representation to explain our arguments. We emphasize that all results which hold true for solutions $w^{\ve}$ also hold for $w_{\ve}$ with the appropriate adjustments for scaling. 

To identify the correct choice of $\overline{F}(M)$, we vary the values of the right hand side of \pref{fbar} until the corresponding solution has the desired behavior in the limit as $\ve\rightarrow 0$. More precisely, for each $\ell\in \RR$, we study the limiting behavior of $w_{\ve}^{\ell}$, which solves
\begin{equation}\label{correctell}
\begin{cases}
w^{\ell}_{\ve, s}-F_{M}(D^{2}w^{\ell}_{\ve}, y/\ve, s/\ve^{2},\om)=\ell \quad\text{in}\quad Q_1,\\
w^{\ell}_{\ve}=0\quad\text{on}\quad \partial_{p}Q_1.
%w^{\ve}=0 \hspace{1cm}\text{in}\hspace{1 cm}\partial C_1(\hat{x},\hat{t})
\end{cases}
\end{equation}
The correct choice of $\overline{F}(M)$ has the property that $\displaystyle \lim_{\ve\rightarrow 0}w^{-\overline{F}(M)}_{\ve}(y,s, \om)=0$ a.s. in $\om$, uniformly in $Q_{1}$. For $\ell\in \RR$, we expect that 
\begin{eqnarray}\label{dichot}
\begin{cases}
\displaystyle \lim_{\ve\rightarrow 0} w^{\ell}_{\ve}(y, s, \om)\geq   0\quad\text{if}\quad \ell\geq -\overline{F}(M),\\
\displaystyle \lim_{\ve\rightarrow 0} w^{\ell}_{\ve}(y, s, \om)\leq0\quad\text{if}\quad \ell\leq -\overline{F}(M).
\end{cases}
\end{eqnarray}

In order to establish an almost-sure limit for $\left\{w^{\ell}_{\ve}\right\}$, we compare $w^{\ell}_{\ve}$ to the solutions of the corresponding obstacle problems from above and below. It was shown in \cite{csw} that $w^{\ell}_{\ve}$ is controlled by the measure of the contact sets of the obstacle problems, which are subadditive and stationary quantities. We then appeal to the subadditive ergodic theorem (Theorem \ref{genset}) to obtain an almost-sure limit, which determines the limiting behavior of $w^{\ell}_{\ve}$ for any $\ell\in\RR$. One difficulty we encounter is that parabolic cubes do not satisfy the original assumptions of Akcoglu and Krengel's subadditive ergodic theorem \cite{akerg}, and we could not find a result in the literature which applied to the parabolic setting. We overcome this by extending the results of \cite{akerg} for nested families of cubes with sides of arbitrary length. The precise result and its proof are presented in Appendix A. 

Using the information from the limiting behavior of solutions of the obstacle problem, we are then able to build a well-defined, uniformly elliptic operator $\overline{F}$. By invoking the perturbed test function method of Evans \cite{evanshom}, we are able to conclude. 

\medskip 

\noindent\textbf{Rates Strategy}
We follow the general approach outlined by Caffarelli and Souganidis in \cite{cs}, by computing the rate of homogenization for \pref{hom} in two main steps. 

The first step is to understand, for each $M\in \mathbb{S}^{d}$, the rate of decay of the associated approximate corrector $\ve^{2}w^{\ve}(\cdot, \cdot, \om)$ as $\ve$ tends to 0. The behavior of $w^{\ve}$ is controlled by the solutions to the obstacle problem from above and below, and this amounts to understanding the rate of decay of the measure of the contact sets of the obstacle problem. Roughly speaking, we show that for any choice of $\ell\in \RR$, the product of the second moments of the measure of the contact sets corresponding to \pref{correctell} from above and below satisfy a certain rate of decay. Obtaining a rate for this product is quite technical, and relies heavily on regularity estimates for fully nonlinear uniformly parabolic equations. Even with the appropriate parabolic versions of the elliptic regularity results, the argument of \cite{cs} does not immediately work due to the causality of parabolic equations. We overcome this by utilizing estimates developed in a companion paper \cite{linreg1} to adapt the argument of \cite{cs} to this setting. Using the rate on this product, we show that excluding an exceptional set $A^{M}_{\ve}$, which depends only on the size of $M$, $\ve^{2}w^{\ve}(\cdot, \cdot, \om)$ decays with a given rate. This amounts to obtaining a rate of convergence for the homogenization $\ve^{2}w^{\ve}$ (or equivalently $w_{\ve}$) in measure. 

%In fact, \pref{fbar} may be reformulated once more as a special case of \pref{hom} with data which is quadratic in space. Therefore, the rate of decay in the limiting behavior of $w_{\ve}$ may be interpreted as a rate for the homogenization of \pref{hom} in the special case of data which is quadratic in space. This rate can be easily extended to problems which are linear in time. 

The second step is to quantify the perturbed test function method. In \cite{cs}, this step is formulated in the language of $\delta$-solutions, which have been generalized to the parabolic setting by Turanova in \cite{olga1}. We use the results of \cite{olga1} in order to show that for configurations $\om\in \Om\setminus A^{M}_{\ve}$, with $|M|$ specified, the decay on $w_{\ve}(\cdot, \cdot, \om)$ yields a rate of decay on $\norm{u^{\ve}(\cdot, \cdot, \om)-u(\cdot, \cdot)}_{L^{\infty}(D_{T})}$. This allows us to obtain a rate for the homogenization in measure.

\medskip 

\noindent\textbf{General Assumptions:}
We work in the stationary ergodic, spatio-temporal setting. We assume there exists an underlying probability space $(\Om, \mathcal{F}, \PP)$ equipped with a measure-preserving transformation $\left\{\tau_ {(y', s')}\right\}_{(y', s')\in \RR^{d+1}}$. We also assume that $\partial_{p}D_{T}$ satisfies a uniform exterior cone condition, which allows us to construct global barriers (see \cite{parbar} for the precise assumption). Our hypotheses can be summarized as follows:

\begin{list}{ (\thefscan)}
{
\usecounter{fscan}
\setlength{\topsep}{1.5ex plus 0.2ex minus 0.2ex}
\setlength{\labelwidth}{1.2cm}
\setlength{\leftmargin}{1.5cm}
\setlength{\labelsep}{0.3cm}
\setlength{\rightmargin}{0.5cm}
\setlength{\parsep}{0.5ex plus 0.2ex minus 0.1ex}
\setlength{\itemsep}{0ex plus 0.2ex}
}

\item \label{f1} \textit{Stationarity}: For every $(M, \om)\in \mathbb{S}^{d}\times\Om$, and for all $(y',s')\in\RR^{d+1}$, 
\begin{equation*}
F(M, y+y', s+s', \om)=F(M, y, s, \tau_{(y',s')}\om).
\end{equation*}

\item\label{f2} \textit{Ergodicity}:  If  $A\in \mathcal{F}$, and $\tau_{(y',s')}A=A$ for all $(y', s')\in\RR^{d+1}$, then $\PP(A)=0$ or $\PP(A)=1$. 

\item \label{f3} \textit{Uniform Ellipticity}. There exists $\la, \La>0$ such that for all $N\in \mathbb{S}^{d}$, $N\geq 0$, and for all $\om\in \Om$,
\begin{equation*}
\la\norm{N}\leq F(M+N, y, s, \om)-F(M, y, s, \om)\leq \La\norm{N}.
\end{equation*}
where $\norm{N}$ denotes the largest eigenvalue of $N$. This immediately implies that uniformly in $(y,s, \om)$, $F(\cdot, y, s, \om)$ is Lipschitz continuous on $\mathbb{S}^{d}$. 
\item \label{f4} \textit{Boundedness and Regularity of $F$}:  For every $R>0, \om\in \Om, M\in \mathbb{S}^{d}$ with $\norm{M}\leq R$, 
\begin{equation*}
\left\{F(M, \cdot, \cdot, \om)\right\}\text{is uniformly bounded and uniformly equicontinuous on}~\RR^{d+1}
\end{equation*}
and there exists $C$ so that 
\begin{equation*}
\esssup_{\om\in \Om}|F(0, 0, 0, \om)|<C, 
\end{equation*}
which by \pref{f1}, implies 
\begin{equation*}
\sup_{(y,s)\in \RR^{d+1}}\esssup_{\om\in\Om} |F(0, y, s, \om)|<C.
\end{equation*}
In light of the continuity of $F$, this implies that 
\begin{equation*}
\esssup_{w\in\Om} \sup_{(y,s)\in \RR^{d+1}} |F(0, y, s, \om)|<C.
\end{equation*}
We also require that there exists a modulus of continuity $\rho(\cdot)$, and a constant $\sig>\frac{1}{2}$ such that for all $(M, y, s, \om)\in \mathbb{S}^{d}\times \RR^{d+1}\times \Om$, 
\begin{equation*}
|F(M,y_{1},s_{1},\om)-F(M, y_{2}, s_{2}, \om)|\leq \rho[(1+|M|)(|y_{1}-y_{2}|_{\RR^{d}}+|s_{1}-s_{2}|)^{\sig}]
\end{equation*}
where $\norm{\cdot}_{\RR^{d}}$ denotes the standard Euclidean norm on $\RR^{d}$. 

By applying \pref{f3}, we have that
\begin{equation*}
\esssup_{\om\in\Om} \sup_{(y,s)\in \RR^{d+1}}|F(M, y, s, \om)|\leq C+\La |M|\leq C(1+|M|).
\end{equation*}

\item \label{f5} \textit{Mixing Hypothesis.} $\Omega$ is strongly mixing with rate given by \pref{gmixing}. 
\end{list}

Conditions \pref{f3} and \pref{f4} imply that $F(\cdot, \cdot, \cdot, \om)$ admits a comparison principle (see \cite{users}). 

\medskip 

\noindent\textbf{Historical/Literature Review}
The first results regarding homogenization of fully nonlinear second-order equations were for the periodic homogenization of uniformly elliptic equations in the work of Evans \cite{evanshom}. The problem was later revisited by Caffarelli using a different approach in \cite{cafhom}. In the random setting, stochastic homogenization for elliptic equations originated in the study of linear, uniformly elliptic operators in divergence and nondivergence form in the works of Papanicolau and Varadhan  \cite{papvar2, papvar1} and Kozlov \cite{kozlov1}. For fully nonlinear uniformly elliptic equations in random media, the first work was done by Caffarelli, Souganidis, and Wang \cite{csw}, which was the main inspiration for this paper. The topic was revisited by Armstrong and Smart \cite{ashom} for equations with gradient dependence, and furthermore the method was extended by the same authors to a class of degenerate elliptic equations in the random setting \cite{asdeg}.  

Rates of convergence for homogenization of second-order equations is a far less well-developed topic, with a smaller body of literature. In the linear setting, under more restrictive hypotheses than ours here, Yurinskii \cite{yurest1, yurest2} obtained an algebraic rate of convergence for homogenization of uniformly elliptic equations in dimensions $d\geq 3$. For fully nonlinear elliptic equations in the periodic setting, which have an additional convexity assumption on $F$, algebraic rates of convergence have been proven using several different approaches. Camilli and Marchi \cite{camarchi} proved an algebraic rate of convergence for the periodic homogenization of fully nonlinear uniformly elliptic equations, using the regularity of the correctors and PDE methods. Ichihara \cite{ichihara} showed that for degenerate elliptic equations in the periodic setting, one may obtain algebraic rates of convergence using methods from stochastic control theory. The paper of Caffarelli and Souganidis \cite{cs} was the first paper for fully nonlinear equations which did not require a convexity assumption on $F$, and it addressed rates of convergence for homogenization in the periodic, almost-periodic, and stationary ergodic setting. Assuming a logarithmic mixing hypothesis, they obtain a logarithmic rate of convergence for homogenization in measure. As previously mentioned, a new approach to stochastic homogenization has been put forth by Armstrong and Smart \cite{asellip} which yields an algebraic rate of convergence in measure for stochastic homogenization of fully nonlinear uniformly elliptic equations, assuming that the environment is independent and identically distributed. 

\smallskip

\noindent\textbf{Outline of the Paper}
We begin by addressing some preliminary background and notation necessary for the proofs of Theorem \ref{stochhom} and Theorem \ref{thmfinalrate}. The proof of Theorem \ref{stochhom} follows in Section 3. The rest of the paper is devoted to proving Theorem \ref{thmfinalrate}. In Section \ref{decaymass}, we prove a rate on the decay of the product of the second moments of the measure of the contact sets of the obstacle problems from above and below. This yields a rate of decay on the approximate correctors $w_{\ve}$ in measure, which is explained in Section \ref{decaycorrect}. Finally, in Section \ref{errorcomp}, the argument is completed by employing the regularity estimates of \cite{olga1}. 

The Appendix to this paper contains content which is crucial for the proofs of Theorem \ref{stochhom} and Theorem \ref{thmfinalrate}. In Appendix A, a generalization of Akcoglu and Krengel's multiparameter subadditive ergodic theorem is proven for cubes with sides of arbitrary length. In Appendix B, we include some of the proofs of the technical lemmas in Section \ref{decaymass}, which take on a similar flavor to their elliptic analogues in \cite{cs}. Appendix C reviews the quantitative regularity estimate established in \cite{linreg1}, which is the most delicate step in the proof of Theorem \ref{thmrate1}. In Appendix D, we review the regularity estimates established in \cite{olga1}, which are necessary to complete the proof of Theorem \ref{thmfinalrate}. 
\section{Preliminaries}
\subsection{Notation}
We work in the Euclidean space $\RR^{d+1}$. We reserve $|\cdot|$ to denote the Lebesgue measure of sets or the norm of vectors, while $\norm{\cdot}$ will denote either the maximum eigenvalue of a nonnegative matrix or the appropriate norm for a Banach space. We define the parabolic distance between $(x_1, t_1)$ and $(x_2, t_2)$ in $\RR^{d+1}$ by
\[
d((x_{1}, t_{1}), (x_{2}, t_{2}))=
\left(|x_1-x_2|^{2}+|t_1-t_2|\right)^{1/2}.
\]
For $\al\in (0,1)$, $u\in C^{0, \al}(D_{T})$ if 
\begin{equation*}
\norm{u}_{C^{0, \al}(D_{T})}=\norm{u}_{L^{\infty}(D_{T})}+\sup_{(x_{i}, t_{i})\in D_{T}, \atop (x_{1}, t_{1})\neq (x_{2}, t_{2})}\frac{|u(x_{1}, t_{1})-u(x_{2}, t_{2})|}{[d((x_{1}, t_{1}), (x_{2}, t_{2}))]^{\al}}.
\end{equation*}
Moreover, $u\in C^k$ if for all $\al, \beta$ such that $|\al|+2\beta\leq k$,  $D^{\al}_{x}D^{\beta}_{t}u$ is continuous. 

Unless otherwise stated, $u^{\ve}$ and $u$ will always refer to the respective solutions of \pref{hom} and \pref{limit}. We denote $\mathbb{S}^{d}$ to be the space of $d\times d$ symmetric matrices with real coefficients. For $P, N\in \mathbb{S}^{d}$, we say that $P\geq N$ if tr$(P-N)\geq 0$. We reserve $\la, \La$ to be the ellipticity constants corresponding to \pref{f3}. 

We work in parabolic cylinders $Q_{r}=B_{r}(0)\times (-r^{2}, 0]$ or $Q_{r}(x,t)=B_{r}(x)\times (t-r^{2}, t]$ and parabolic cubes of the form $C_{r}=(-r, r)^{d}\times (-r^{2}, 0]$ or $C_{r}(x,t)=(-r+x, r+x)^{d}\times (t-r^{2}, t]$. We also work with cylinders which are forward in time, which we denote by $Q^{+}_{r}=B_{r}(0)\times (0, r^{2}]$ or $Q^{+}_{r}(x,t)=B_{r}(x)\times (t,t+ r^{2}]$. 

For every $M\in \mathbb{S}^{d}$, we define Pucci's extremal operators for ellipticity constants $\la, \La$ by
\begin{equation*}
\MM^{-}(M)=\la\left(\sum_{e_{i}>0} e_{i}\right)+\La\left(\sum_{e_{i}<0}e_{i}\right)\quad\text{and}\quad\MM^{+}(M)=\La\left(\sum_{e_{i}>0} e_{i}\right)+\la\left(\sum_{e_{i}<0}e_{i}\right),
\end{equation*}
where $e_{i}$ are the eigenvalues of $M$. It follows immediately from \pref{f3} that for all $N\geq 0$, 
\begin{equation*}
\MM^{-}(N)\leq F(M+N, y, s, \om)-F(M, y, s, \om)\leq \MM^{+}(N).
\end{equation*}
For a detailed analysis of Pucci's extremal operators, and a list of their properties, see \cite{cc}. For any domain $D_{T}\subset \RR^{d+1}$, we define $\overline{S}(g, D_{T})$ to be the set of viscosity supersolutions to
\begin{equation*}
u_{t}-\MM^{-}(D^{2}u)\geq g\quad\text{in}\quad D_{T},
\end{equation*}
and respectively, $\underline{S}(g, D_{T}))$ to be the set of viscosity subsolutions to
\begin{equation*}
u_{t}-\MM^{+}(D^{2}u)\leq g\quad\text{in}\quad D_{T}.
\end{equation*}
We define $S(g, D_{T})=\overline{S}(g, D_{T})\cap \underline{S}(g, D_{T})$ and $S^{*}(g, D_{T})=\overline{S}(-\norm{g}_{\infty}, D_{T})\cap \underline{S}(\norm{g}_{\infty}, D_{T})$. 

We work in the probability space $(\Om, \mathcal{F}, \PP)$. We denote $\EE[\cdot]$ to be the expectation of a random variable, and $\mathbb{V}[\cdot]=\EE[(\cdot)^{2}]-(\EE[\cdot])^{2}$ to denote the variance of a random variable. 

We will frequently fix the matrix $M\in \mathbb{S}^{d}$ in our analysis, and therefore when it is clear from context, we use the notation $F(\cdot, y,s,\om)=F_{M}(\cdot, y, s, \om)=F(\cdot+M, y,s, \om)$ in order not to clutter the notation. 

\subsection{Parabolic Regularity}
We will make use of the following regularity results for parabolic equations. These estimates are stated for viscosity solutions, and we refer to \cite{wangreg1} for proofs. \\
\noindent\textbf{Parabolic ABP-Estimate}~(Wang, \cite{wangreg1}).
\textit{
Let $u\in\overline{S}(g,Q_{1})$, and assume that $u\geq 0$ on $\partial_{p}(Q_1)$. There exists $C=C(\la, \La, d)>0$ such that 
\begin{equation*}
\sup_{Q_{1}} u^{-}\leq C\left(\int_{\left\{u=\Gamma(u)\right\}}(g^{-})^{d+1} dx dt\right)^{1/(d+1)}
\end{equation*}
where $\Ga(u)$ is the convex envelope of $-u^{-}$. }
\\
\noindent\textbf{Parabolic Krylov-Safonov Estimate} (Wang, \cite{wangreg1}).
\textit{There exists $C=C(\la, \La, d)>0$, $\sig=\sig(\la, \La, d)\in (0,1)$ such that if $u\in S^{*}(g, Q_{1})$, then $u\in C^{0, \sig}(Q_{1})$ and 
\begin{equation*}
\norm{u}_{C^{0, \sig}(Q_{1/2})}\leq C(\norm{u}_{L^{\infty}(Q_{1})}+\norm{g}_{L^{d+1}(Q_{1})}).
\end{equation*}
}
\subsection{Properties of the Obstacle Problem}
For each $\ell\in \RR$, the limiting behavior of $w^{\ve, \ell}$ is controlled by the solutions to the obstacle problem with obstacle 0. In this section, we drop the dependencies of $w=w^{\ve, \ell}$ since it plays no role in our analysis. To avoid confusion, we will also present everything by solving problems in general domains $K\subset \RR^{d+1}$. We recall that the solution to the obstacle problem from above $\overline{v}(y,s,\om)$ is defined by %Recall that all of the results for viscosity solutions are based on understanding what is going on on the parabolic boundary $\partial_{p}D=\left(\mathcal{D}\times\left\{t=t_{1}\right\}\right)\cup\left(\partial\mathcal{D}\times(t_{1}, t_{2})\right)$. In other words
\begin{equation*}
\overline{v}=\inf \left\{v\in C(\overline{K}): v_{s}-F(D^{2}v, y, s, \om)\geq \ell\quad\text{in}\quad K\quad\text{and}\quad v\geq 0\quad\text{on}\quad \overline{K}\right\}
\end{equation*}
and respectively,
\begin{equation*}
\underline{v}=\sup \left\{v\in C(\overline{K}): v_{s}-F(D^{2}v, y, s, \om)\leq \ell\quad\text{in}\quad K\quad\text{and}\quad v\leq 0\quad\text{on}\quad\overline{K}\right\},
\end{equation*}
The key quantity to study is the measure of the contact sets of the obstacle problems. We define
\begin{equation*}
\overline{m}(K,\ell, \om)=\left| \left\{\overline{v}=0\right\}\right|\quad\text{and}\quad \underline{m}(K, \ell, \om)=\left|\left\{ \underline{v}=0\right\}\right|.
\end{equation*}

We also define
\begin{equation*}
\overline{m}_{\ve}(K,\ell, \om)=\left| \left\{\overline{v}_{\ve}=0\right\}\right|\quad\text{and}\quad \underline{m}_{\ve}(K, \ell, \om)=\left|\left\{ \underline{v}_{\ve}=0\right\}\right|
\end{equation*}
to denote when we are solving the obstacle problem in $K$ with arguments $\left(y/\ve, s/\ve^{2}\right)$. 

The next theorem summarizes the properties of the obstacle problem, which are relevant to this paper. The proofs have been omitted, because they follow as direct parabolic versions of the proofs found in \cite{csw}.
\begin{theorem}\label{obsthm}
Let $\overline{v}$(resp. $\underline{v})$ be defined as above. 
\begin{enumerate}[(i)]
\item \label{obst1} There exists $\sig=\sig(\la, \La, d)\in (0,1)$ such that $\overline{v}, \underline{v}\in C^{0, \al}(\overline{K})$ and $\overline{v}=\underline{v}=0$ in $\partial_{p}K$.
\item \label{obst2}
$\overline{v}$ solves
\begin{equation}\label{obsteqa}
\begin{cases}
\overline{v}_{s}-F(D^{2}\overline{v}, y, s, \om)=\ell+(\ell+F(0,y,s, \om))_{-}\chi_{\left\{\overline{v}=0\right\}}\quad\text{in}\quad K,\\
\overline{v}=0\quad\text{on}\quad \partial_{p}K,
\end{cases}
\end{equation}
and $\underline{v}$ solves
\begin{equation}\label{obsteqb}
\begin{cases}
\underline{v}_{s}-F(D^{2}\underline{v}, y, s, \om)=\ell-(\ell+F(0,y,s, \om))_{+}\chi_{\left\{\underline{v}=0\right\}}\quad\text{in}\quad K,\\
\underline{v}=0\quad\text{on}\quad \partial_{p}K.
\end{cases}
\end{equation}

\noindent  For $(y_{0},s_{0})\in \left\{\overline{v}\neq 0\right\}$ (resp., $\left\{\underline{v}\neq 0\right\}$), $\overline{v}_{s}(y_{0}, s_{0})-F(D^{2}\overline{v}(y_{0}, s_{0}), y_{0}, s_{0}, \om)=\ell$ (resp., $\underline{v}_{s}(y_{0}, s_{0})-F(D^{2}\underline{v}(y_{0}, s_{0}), y_{0}, s_{0}, \om)=\ell)$ . 

\item \label{obst3} If $K_{1}\subset K_{2}$, and $\overline{v}_{1}$ and $\overline{v}_{2}$ solve the obstacle problem from above in those respective domains, then $\overline{v}_{1}\leq \overline{v}_{2}$. Therefore, 
\begin{equation*}
\overline{m}(K_{2}, \ell, \om)\leq \overline{m}(K_{1}, \ell,\om).
\end{equation*}
Similarly,
\begin{equation*}
\underline{m}(K_{2}, \ell, \om)\leq \underline{m}(K_{1}, \ell, \om).
\end{equation*}
%\item Let $D_{1}\subset D_{2}$. Suppose that $\partial_{p}D_{1}\subset \partial_{p} D_{2}$. Then we have that $\overline{\Ga}_{2}\cap D_{1}=\overline{\Ga}_{1}$. Similarly, $\underline{\Ga}_{2}\cap D_{1}=\underline{\Ga}_{1}$\label{obst4} 
\item  \label{obst5} If $F_{1}$ and $F_{2}$ satisfy \pref{f3}, and $F_{1}\leq F_{2}$, then  $\overline{m}(K, \ell, F_{1})\leq \overline{m}(K, \ell, F_{2})$. Similarly, $\underline{m}(K, \ell, F_{1})\leq \underline{m}(K, \ell, F_{2})$.
\end{enumerate}
\end{theorem}

Finally, we highlight an important observation which is unique to parabolic obstacle problems. We use this later to address some interesting challenges which arise in the proof of Theorem \ref{thmrate1}.

\begin{lemma}\label{monobst}
Let  $\overline{v}_{1}$ solve the obstacle problem from above in $K_{1}$, and $\overline{v}_{2}$ solve the obstacle problem from above in $K_{2}$, with $\partial_{p}K_{1}\subset \partial_{p} K_{2}$. Then $(\left\{\overline{v}_{2}=0\right\}\cap K_{1})=\left\{\overline{v}_{1}=0\right\}$. Respectively, for $\underline{v}_{1}$ and $\underline{v}_{2}$ defined analogously, $(\left\{\underline{v}_{2}=0\right\}\cap K_{1})=\left\{\underline{v}_{1}=0\right\}$.
\end{lemma}

\begin{proof}
We only show the proof for $\overline{v}$. By monotonicity of the obstacle problem, it is immediate that $\left\{\overline{v}_{2}=0\right\}\cap K_{1}\subset \left\{\overline{v}_{1}=0\right\}$. To show equality, we suppose for the purpose of contradiction that $\left\{\overline{v}_{2}=0\right\}\cap K_{1}\subsetneq \left\{\overline{v}_{1}=0\right\}$. This means that there exists a point $(\hat{y}, \hat{s})$, where $\overline{v}_{1}(\hat{y}, \hat{s})$ vanishes, but $\overline{v}_{2}(\hat{y}, \hat{s})$ remains strictly positive. By the continuity of $\overline{v}_{1}$ and $\overline{v}_{2}$, there exists an open set $B$ where $\overline{v}_{2}>\overline{v}_{1}$, and hence the maximum of $\overline{v}_{2}-\overline{v}_{1}$ occurs in $B$. Moreover, in $B$, $\overline{v}_{2}$ is a solution to the obstacle problem, while $\overline{v}_{1}$ is only a supersolution. As in the proof of the comparison principle on $K_{1}$, this contradicts the solution properties of $\overline{v}_{2}$ and $\overline{v}_{1}$, which implies $\left\{\overline{v}_{2}=0\right\}\cap K_{1}=\left\{\overline{v}_{1}=0\right\}$. The same argument shows the statement for $\underline{v}_{1}$ and $\underline{v}_{2}$. 
\end{proof}

\subsection{The Mixing Hypothesis}
We assume a quantified mixing hypothesis on the probability space $(\Omega, \mathcal{F}, \PP)$. We assign a rate to the``decorrelation function," which describes how quickly events become independent. We denote $\mathcal{B}$ and $\mathcal{B}(r)$ to denote the smallest $\sigma$-algebras generated by the measurable subsets $\left\{F(\cdot, y, s, \cdot): (y,s)\in C_{1}\right\}$ and $\left\{F(\cdot, y, s, \cdot): d((y,s), C_{1})\geq r\right\}$ of $\Om$. We impose a logarithmic mixing rate, in order to obtain a logarithmic rate of convergence for homogenization:

There exists $c>0$ such that for all $\delta>0$, 
\begin{equation}\label{gmixing}
\sup_{A\in \mathcal{B}, B\in \mathcal{B}(r)} |\PP[A\cap B]-\PP[A]\PP[B]|\leq \delta\quad\text{if}\quad r>\delta^{c\ln \delta}.
\end{equation}

We note that the decay rate in the mixing hypothesis yields a similar rate of decorrelation for random variables on $\mathcal{B}$ and $\mathcal{B}(r)$. Indeed, the following proposition follows from \pref{gmixing}, 
\begin{proposition}[Caffarelli, Souganidis \cite{cs}]
Assume \pref{gmixing}. Let $f\in L^{\infty}_{+}(\mathcal{B})$, the space of bounded, nonnegative random variables on $\mathcal{B}$, and $g\in L^{\infty}_{+}(\mathcal{B}(r))$, the space of bounded, nonnegative random variables on $\mathcal{B}(r)$. Then 
\begin{equation}
\sup_{f\in L^{\infty}_{+}(\mathcal{B})\atop g\in L^{\infty}_{+}(\mathcal{B}(r))}\norm{\EE[fg]-\EE[f]\EE[g]}\leq 3^{-k^{3/2}}\norm{f}_{\infty}\norm{g}_{\infty}\quad\text{if}\quad r> 3^{k^{3}}.
\end{equation}
\end{proposition}
Moreover, the decorrelation of random variables leads to a decay in variances. 
\begin{lemma}[Caffarelli, Souganidis,\cite{cs}] 
Let $\pi_{1}, \pi_{2}\ldots \pi_{M}$ be a family of random variables such that for $i,j=1, \ldots M$ and some $\sigma_{ij}>0$, $\EE[\pi_{i}\pi_{j}]-\EE[\pi_{i}]\EE[\pi_{j}]\leq \sigma_{ij}$. Then
\begin{align*}
&\mathbb{V}\left[\frac{1}{M}\sum_{i=1}^{M}\pi_{i}\right]\leq \frac{1}{M^{2}}\sum_{i=1}^{M}\mathbb{V}[\pi_{i}]+\frac{1}{M^{2}}\sum_{i,j=1}^{M}\sigma_{ij},\\
&\EE\left[\left(\frac{1}{M}\sum_{i=1}^{M}\pi_{i}\right)^{2}\right]\leq \left(\frac{1}{M}\sum_{i=1}^{M}\EE[\pi_{i}]\right)^{2}+\frac{1}{M^{2}}\sum_{i,j=1}^{M}\sigma_{ij}+\frac{1}{M^{2}}\sum_{i=1}^{M}\mathbb{V}[\pi_{i}],
\end{align*}
and if, for all $i=1, \ldots, M$, $\EE[\pi_{i}]=E$, $\mathbb{V}[\pi_{i}]=V$, then 
\begin{equation}\label{probfact}
\mathbb{V}\left[\frac{1}{M}\sum_{i=1}^{M}\pi_{i}\right]\leq \frac{1}{M}V+\frac{1}{M^{2}}\sum_{i,j=1}^{M}\sigma_{ij}.
\end{equation}
\end{lemma}

\section{The Homogenization of \pref{hom}}
\subsection{The Construction of $\overline{F}$}\label{construct}
We use the subadditive ergodic theorem (Theorem \ref{genset}) to identify $\overline{F}$ through the ergodic properties of the measure of the contact sets of the obstacle problem. Recall that for every $M\in \mathbb{S}^{d}$, our choice of $\overline{F}(M)$ relies upon the limiting behavior of $w^{-\overline{F}(M)}_{\ve}$. We will first understand the limiting of behavior of $w^{\ell}_{\ve}$ for each choice of $\ell\in\RR$, in order to find the correct choice $\ell=-\overline{F}(M)$. Since $M$ is fixed throughout this section, we adopt the notation that $F(\cdot, y, s, \om)=F_{M}(\cdot, y, s, \om)$ for simplicity. 

We claim that $\overline{m}(\cdot, \ell, \cdot)$ (resp. $\underline{m}(\cdot, \ell, \om)) : \mathcal{V}\times \Om\rightarrow \RR$, where $\mathcal{V}$ represents the set of subdomains of $\RR^{d+1},$ satisfies \pref{sethypot1}-\pref{bnd} for Theorem \ref{genset}.

\begin{proposition}
$\overline{m}(K, \ell,\om)$ (resp. $\underline{m}(K, \ell, \om))$ satisfies \pref{sethypot1}-\pref{bnd} for Theorem \ref{genset}.
\end{proposition}

\begin{proof}
We only show the proof for $\overline{m}(K, \ell, \om)$ since the other case follows similarly. 
\begin{enumerate}
\item We note that \pref{bnd} is immediate since $\overline{m}(K, \ell,\om)\leq |K|$. 

\item We next verify the subadditivity, Let $K=\bigcup_{j} K_{j}$, where $K_{j}$ are mutually disjoint subsets of $\RR^{d+1}$, and $|\partial_{p}K_{i}\cap\partial_{p}K_{j}|=0$ if $i\neq j$. Let $\overline{v}$ denote the solution to the obstacle problem from above on $K$. The monotonicity of the obstacle problem and the subadditivity yield
\begin{equation*}
\overline{m}(K, \ell, \om)=\sum_{j}\left|\left\{\overline{v}=0\right\}\cap K_{j}\right|\leq \sum_{j} \overline{m}(K_{j}, \ell, \om).
\end{equation*}

\item We show that $\overline{m}(K, \ell, \om)$ satisfies the stationarity assumption \pref{sethypot1} . For $(z,r)\in \RR^{d+1}$, we define
\begin{equation*}
T_{(z,r)}\overline{m}(K, \ell, \om)=\overline{m}((z,r)+K,\ell, \om) \quad\text{and}\quad T_{(z,r)}\underline{m}(K, \ell, \ \om)=\underline{m}((z,r)+K, \ell, \om).
\end{equation*}
We claim that for all $(z,r)\in\RR^{d+1}$, $T_{(z,r)}\overline{m}(K, \ell, \om)=\overline{m}(K, \ell, \tau_{(z,r)}\om)$, and respectively, $T_{(z,r)}\underline{m}(K, \ell, \om)=\underline{m}(K, \ell, \tau_{(z,r)}\om)$.

Fix $(z, r)\in \RR^{d+1}$. We only show the proof for $\overline{m}(K, \ell, \om)$, as the case for $\underline{m}(K, \ell, \om)$ follows similarly. For $(y,s)\in (z,r)+K$, let
\begin{equation*}
\tilde{v}(y, s, \om)=\overline{v}_{(z,r)+K}(y,s, \om). 
\end{equation*}
If we represent $(y,s)=(z,r)+(y',s')$ for $(y', s')\in K$, then $\tilde{v}$ solves
\begin{equation*}
\tilde{v}_s-F(D^{2}\tilde{v},z+y', r+s', \om)=\tilde{v}_{s}-F(D^{2}\tilde{v}, y', s', \tau_{(z,r)}\om) \geq \ell \quad\text{in}\quad K. 
\end{equation*}
Moreover, since $\tilde{v}\geq 0$ on $(z,r)+K$, it follows that $f(y',s',\om)=\tilde{v}(z+y', r+s', \om)\geq 0$ on $K$. 
%\begin{equation*}
%\tilde{v}(y,s, \om)\geq 2\overline{M}(y)=\langle M(z+y'),z+y'\rangle =\langle My',y'\rangle+\langle Mz,z\rangle+2\langle My',z\rangle, 
%\end{equation*}
Combining these two observations, we observe that $f(y', s', \om)$ is an admissible supersolution to the obstacle problem from above in $K$, and by applying the minimality of $\overline{v}$, we have
\begin{equation*}
\left\{f(y',s')=0\right\}\subset\left\{\overline{v}_{K}(y',s', \tau_{(z,r)}\om)=0\right\}.
\end{equation*}
This implies that 
\begin{equation*}
\overline{m}(K,\ell, \tau_{(z,r)}\om)\geq \left|\left\{f(y',s')=0\right\}\right|=\overline{m}((z,r)+K,\ell, \om)=T_{(z,r)}\overline{m}(K, \ell, \om).
\end{equation*}
To obtain the reverse inequality, we can make an analogous argument to show that perturbing $\overline{v}(y,s, \tau_{(z,r)}\om)$ will solve the obstacle problem from above in $K+(z,r)$, and by applying minimality of the obstacle problem, we obtain the desired result. 
\end{enumerate}
\end{proof}

Since $\overline{m}(K, \ell, \om)$ and $\underline{m}(K, \ell, \om)$ satisfy \pref{sethypot1}-\pref{bnd}, Theorem \ref{genset} yields that for each $\ell\in \RR$, there exists an event $\Om^{\ell}\subset \Om$, of full probability, such that for all $\om\in \Om^{\ell}$, 
\begin{equation}\label{pdef}
\overline{p}(\ell)=\lim_{\ve\rightarrow0}\frac{\overline{m}(Q_{1/\ve}, \ell, \om)}{|Q_{1/\ve}|}=\inf_{r>0} \frac{\overline{m}(Q_{r}, \ell, \om)}{|Q_{r}|} \quad\text{and}\quad \underline{p}(\ell)=\lim_{\ve\rightarrow0}\frac{\underline{m}(Q_{1/\ve}, \ell, \om)}{|Q_{1/\ve}|}=\inf_{r>0} \frac{\underline{m}(Q_{r}, \ell, \om)}{|Q_{r}|},
\end{equation}
where $\overline{p}(\ell)$, $\underline{p}(\ell)$ are constant by the ergodicity assumption \pref{f2}. We note that the values of $\overline{p}(\ell)$ and $\underline{p}(\ell)$ do not depend on whether we use cylinders or cubes. More importantly, if we change scales and work in $Q_{1}$, we may represent
\begin{equation}
\overline{p}(\ell)=\lim_{\ve\rightarrow 0}\frac{\overline{m}_{\ve}(Q_{1}, \ell, \om)}{|Q_{1}|}\quad\text{and}\quad \underline{p}(\ell)=\lim_{\ve\rightarrow0}\frac{\underline{m}_{\ve}(Q_{1}, \ell, \om)}{|Q_{1}|}.
\end{equation}
We show that the limiting behavior of $w^{\ell}_{\ve}$ is controlled by whether $\overline{p}(\ell)>0$ or $\overline{p}(\ell)=0$. We claim that if $\overline{p}(\ell)=\al>0$, then the contact set must spread all over, and hence $\overline{v}_{\ve}$ must converge to the obstacle uniformly. 

\begin{lemma}\label{pos}
If $\overline{p}(\ell)=\al>0$ (respectively $\underline{p}(\ell)>0$), then for all $\om\in \Om_{\ell}$, $\displaystyle \lim_{\ve\rightarrow 0} \overline{v}_{\ve}=0$ (resp.,  $\displaystyle \lim_{\ve\rightarrow 0} \underline{v}_{\ve}=0$) uniformly on $\overline{Q}_{1}$. 
\end{lemma}
\begin{proof}
For convenience, we use the unit cube $C_{1}$ instead of the unit cylinder.  We decompose  $C_{1}$ into a disjoint union of parabolic cubes $C^{i}$ of diameter at most $\delta$, and show that $\overline{v}_{\ve}$ must vanish in each $C^{i}$. Indeed, by Theorem \ref{obsthm}\pref{obst3} and Theorem \ref{genset}, for $\ve$ sufficiently small and for all $i$, $|\left\{\overline{v}_{\ve}=0\right\}\cap C^{i}|\leq |\overline{m}_{\ve}(C^{i}, \ell, \om)|=\al |C^{i}|$, for all $\om\in \Om_{\ell}$. This implies that the ratio of contact set of $\overline{v}_{\ve}$ in each $C^{i}$ can never be greater than $\al$. Moreover, the ratio of contact set can never be strictly less than $\al$ in any subcube. Otherwise, in order to preserve the ratio $\al$ over $C_{1}$, a different subcube would have to have ratio strictly greater than $\al$ which is impossible. Therefore, for all $\om\in \Om_{\ell}$, every subcube $C^{i}$ has ratio exactly $\al$ of contact set, which means there exists at least one point in each $C^{i}$ so that $\overline{v}_{\ve}$ vanishes. 

We note that with this scaling, in view of Theorem \ref{obsthm}\pref{obst1}, $\left\{\overline{v}_{\ve}\right\}$ have uniform Holder estimates. Therefore, 
\begin{equation*}
|\overline{v}_{\ve}|\leq C_{h}\delta^{\sig}\quad\text{in}\quad C_{1}.
\end{equation*}

Letting first $\ve\rightarrow 0$, then $\delta\rightarrow 0$, yields $\lim_{\ve\rightarrow 0} \overline{v}_{\ve}=0$ uniformly in $Q_{1}$, for all $\om\in \Om_{\ell}$. 
\end{proof}

Since we are interested in $\lim_{\ve \rightarrow 0} w^{\ell}_{\ve}$ where $w^{\ell}_{\ve}$ solves \pref{correctell}, we consider the behavior of the relaxed half limits
\begin{equation*}
w^{\ell, *}(y,s,\om)=\limsup_{\ve\rightarrow 0 \atop y'\rightarrow y, s'\rightarrow s} w^{\ell}_{\ve}(y',s', \om)\quad\text{and}\quad w^{\ell}_{*}(y,s,\om)=\liminf_{\ve\rightarrow 0 \atop y'\rightarrow y, s'\rightarrow s} w^{\ell}_{\ve}(y',s', \om).
\end{equation*}

We show how the values of $\overline{p}(\ell)$ determine the limiting behavior of $w^{\ell,*}$ and $w^{\ell}_{*}$. 
\begin{lemma}\label{alpos}
If $ \overline{p}(\ell)>0$, then $w^{\ell,*}(\cdot, \cdot, \om)\leq 0$. 
\end{lemma}

\begin{proof}
We fix $\ell$ and drop the dependencies of $\ell$, since it plays no role in our analysis. By the comparison principle, $\overline{v}_{\ve}(\cdot, \om)\geq w_{\ve}(\cdot, \om)$, so by letting $\ve\rightarrow 0$ and applying Lemma \ref{pos}, we have the desired result.
\end{proof}

\begin{lemma}\label{al0}
If $\overline{p}(\ell)=0$, then $w^{\ell}_{*}(\cdot, \cdot, \om)\geq 0$. 
\end{lemma}
\begin{proof}
We drop the dependencies on $\ell$ since they play no role in our analysis. It is immediate from Theorem \ref{obsthm}\pref{obst2} that
\begin{align*}
(w_{\ve}-\overline{v}_{\ve})_{s} -\MM^{-}(D^{2}w_{\ve}-D^{2}\overline{v}_{\ve}) &\geq(w_{\ve}-\overline{v}_{\ve})_{s}+ F(D^{2}\overline{v}_{\ve}, y, s, \om)-F(D^{2}w_{\ve}, y, s, \om)\\
&\geq-\norm{F(0, \cdot, \cdot, \om)}_{\infty}\chi_{\left\{\overline{v}_{\ve}=0\right\}}
\end{align*}
and $w_{\ve}-\overline{v}_{\ve}\rst{\partial Q_{1}}=0$. The ABP-estimate yields
\begin{align*}
\sup_{Q_{1}}(-w_{\ve}(\cdot, \cdot,\om))\leq \sup_{Q_{1}}(\overline{v}_{\ve}(\cdot, \cdot, \om)-w_{\ve}(\cdot, \cdot, \om))&=\sup_{Q_{1}}(w_{\ve}(\cdot, \cdot, \om)-\overline{v}_{\ve}(\cdot, \cdot, \om))^{-}\\
&\leq C[\overline{m}_{\ve}(Q_{1}, \ell, \om)]^{1/(d+1)}.
\end{align*}

Sending $\ve\rightarrow0$, and using that $\overline{p}(\ell)=0$, we conclude. 
\end{proof}

For every $\ell\in\RR$, we have shown there exists an event $\Om^{\ell}\subset \Om$ of full probability such that for all $\om\in \Om^{\ell}$, either $w^{\ell,*}(\cdot, \cdot, \om)\leq 0$ or $w^{\ell}_{*}(\cdot, \cdot,\om)\geq 0$. In order to find the value of $\ell=-\overline{F}(M)$ with the correct limiting behavior on a set of full probability, we define
\begin{equation*}
\overline{\ell}=\inf\left\{\ell\in\RR: \overline{p}(\ell)=0\right\}.
\end{equation*}

and
\begin{equation}\label{omt}
\tilde{\Om}=\bigcap _{M\in \mathbb{S}^{d}_{\mathbb{Q}}}\bigcap _{\ell\in \mathbb{Q}} \Om^{\ell},
\end{equation}
where $\mathbb{S}_{\mathbb{Q}}^{d}$ denotes the space of symmetric $d$-by-$d$ matrices with entries in $\mathbb{Q}$, and we point out that $\PP(\tilde{\Om})=1$. We claim that $\overline{\ell}$ is the correct choice of $-\overline{F}(M)$. 
 
 \begin{lemma}
For all $\om\in \tilde{\Om}$, $\displaystyle \lim_{\ve\rightarrow 0} w^{\overline{\ell}}_{\ve}= 0$ uniformly in $Q_{1}$. 
 \end{lemma}
 
 \begin{proof}
Fix $\om\in\tilde{\Om}$. Let $\left\{\ell_{n}\right\}\subset \mathbb{Q}$ be a monotonically decreasing sequence such that, $\lim_{n\rightarrow\infty}\ell_{n}=\overline{\ell}$, and $\overline{p}(\ell_{n})=0$. For notational simplicitiy, let $w^{n}_{\ve}=w^{\ell_{n}}_{\ve}$, the solution of \pref{correctell} with right hand side $\ell_{n}$. We note that since $\left\{\ell_{n}\right\}\subset \mathbb{Q}$, and $\om\in\tilde{\Om}$, the limiting behavior of $w^{n}_{\ve}$ is determined by Lemmas \ref{alpos} and \ref{al0}. Let $\left\{\al_{n}\right\}$ be a sequence of positive, real numbers, to be chosen later.  Letting $\tilde{w}^{n}_{\ve}=w^{n}_{\ve}-\al_{n}(s+1)$ solves
\begin{equation*}
\tilde{w}^{n}_{\ve, s}(y,s)-F(D^{2}\tilde{w}^{n}_{\ve}, y/\ve, s/\ve^{2}, \om)=w^{n}_{\ve, s}(y,s)-\al_{n}-F(D^{2}w^{n}_{\ve}, y/\ve, s/\ve^{2}, \om)=\ell_{n}-\al_{n}
\end{equation*}
and, for $\ell_{n}-\al_{n}\leq \overline{l}$, $\tilde{w}^{n}_{\ve}$ a subsolution with right hand side $\overline{\ell}$. Moreover, on $\partial_{p}Q_{1}$, $\tilde{w}^{n}_{\ve}(y,s)=-\al_{n}(s+1)\leq 0=w^{\overline{\ell}}_{\ve}(y,s)$. The comparison principle yields
\begin{equation*}
\tilde{w}^{n}_{\ve}(y,s, \om)\leq w^{\overline{\ell}}_{\ve}(y,s, \om).
\end{equation*}
Sending $\ve\rightarrow 0$, and then $n\rightarrow\infty$, we have that
\begin{equation*}
w^{\overline{\ell}}_{*}\geq 0.
\end{equation*}
 
 If we repeat this argument by approximating $\overline{\ell}$ from below, and use that the contact set has positive measure, we have $w^{\overline{\ell},*}\leq 0.$

Hence, 
 \begin{equation*}
0\leq w^{\overline{\ell}}_{*}\leq w^{\overline{\ell},*}\leq 0.
 \end{equation*}

This proves that $\lim_{\ve\rightarrow 0}w^{\bar{\ell}}_{\ve}=0$ uniformly, for all $\om\in \tilde{\Om}$. 
 \end{proof}
 
 We have shown that letting $-\overline{F}(M)=\overline{\ell}$ gives the desired limited behavior for $w^{-\overline{F}(M)}_{\ve}$. We now claim that this is the unique value of $\overline{F}(M)$. In order to do so, it is enough to verify the following uniqueness lemma: 
  \begin{lemma} \label{uniq}
 Suppose that there exists $w^{1}_{\ve}, w^{2}_{\ve}$ which solve $w^{i}_{\ve, s}-F(D^{2}w^{i}_{\ve}, y/\ve, s/\ve^{2}, \om)=\ell_{i}$ respectively, with $\ell_{2}=\ell_{1}+\eta$, for some $\eta>0$. Then there exists $\beta=\beta(\La, d)>0$, isuch that
 \begin{equation*}
 w^{1}_{\ve}(y,s)+\beta\eta (s+1)(1-|y|^{2})\leq w^{2}_\ve(y,s).
 \end{equation*}
 \end{lemma}
 
 \begin{proof}
Let $\tilde{w}_{\ve}(y,s)=w^{1}_{\ve}(y,s)+\beta\eta (s+1)(1-|y|^{2})$, with $\beta$ to be chosen later. Note that $\tilde{w}_{\ve}=0$ on $\partial_{p}Q_{1}$, and
 \begin{align*}
\tilde{w}_{\ve, s}-F(D^{2}\tilde{w}_{\ve}, y/\ve, s/\ve^{2}, \om) &=w^{1}_{\ve,s}+\beta\eta(1-|y|^{2})-F(D^{2}w^{1}-2\beta\eta (s+1)\ell, y/\ve, s/\ve^{2}, \om)\\
&\leq w_{\ve,s}^{1}+\beta\eta(1-|y|^{2})-F(D^{2}w^{1}, y/\ve, s/\ve^{2}, \om)+2\La d \eta \beta (s+1)\\
&\leq \ell_{1}+\beta\eta(1+2\La d)\\
&\leq \ell_{2}
 \end{align*}
 for an appropriate choice of $\beta$. By the comparison principle,
 \begin{equation*}
 \tilde{w}_{\ve}(y,s)=w_{\ve}^{1}(y,s)+\beta\eta (s+1)(1-|y|^{2})\leq w_{\ve}^{2}(y,s)
 \end{equation*}
 as asserted.  
 \end{proof}
 
 Lemma \ref{uniq} shows that it is impossible for two potential values of $\overline{F}(M)$ to yield the same limiting behavior for their corresponding solutions $w_{\ve}$, which implies that $\overline{F}(M)$ is unique for every $M\in \mathbb{S}^{d}$.  Finally, we show that $\overline{F}(\cdot)$  satisfies the necessary hypotheses, namely uniformly ellipticity, in order for \pref{limit} to be well-posed. 
 
We note that the uniform ellipticity of $\overline{F}(\cdot)$ will also imply that $\overline{F}(\cdot)$ is continuous. Therefore, for any $M\in \mathbb{S}^{d}$, consider $\left\{M_{k}\right\}\subset\mathbb{S}^{d}_{\mathbb{Q}}$ such that $M_{k}\rightarrow M$. By the continuity of $\overline{F}(\cdot)$, $\overline{F}(M_{k})\rightarrow \overline{F}(M)$. By \pref{omt}, for every $\om\in \tilde{\Om}$, for every $M\in \mathbb{S}^{d}$, for every $\ell\in \RR$, the limiting behavior of $w^{\ell}_{\ve}(y, s, \om)$ is understood. Hence, the construction of $\overline{F}(M)$ makes sense for all $\om\in \tilde{\Om}$. 
  
 It was observed in \cite{ashom} that the uniform ellipticity of $\overline{F}(\cdot)$ follows relatively easily from the properties of the obstacle problem. The same proof applies in this setting, so we choose to omit it. 
 \begin{proposition}[Armstrong, Smart \cite{ashom}]\label{ellip}
$ \overline{F}(\cdot)$ is uniformly elliptic. 
\end{proposition}

%\begin{proof}
%Fix $M, N\in \mathbb{S}^{d}$, with $M\geq N$. Consider for any $K\in \mathbb{S}^{d}$, 
%\begin{equation*}
%F(M+N, y, s, \om)+\la \tr (M-N)\leq F(M+, y, s, \om)\leq F(K+Y, y, s, \om)+\La\tr(M-K).
%\end{equation*}
%It follows by letting $F_{1}(Y, y, s, \om)=F(K+Y, y, s, \om)+\la \tr (M-K)$ and $F_{2}(Y, y, s, \om)=F(M+Y, y, s, \om)$, that $F_{1}\leq F_{2}$. By Theorem \ref{obsthm} \pref{obst5}, we have $\overline{F}_{1}(0)\leq \overline{F}_{2}(0)$. Since $\overline{F}_{1}(0)=\overline{F}(N)+\la \tr (M-N)$ and $\overline{F}_{2}(0)=\overline{F}(M)$, we conclude that
%\begin{equation*}
%\overline{F}(N)+\la \tr (M-N)\leq \overline{F}(M).
%\end{equation*}
%An analogous argument proves the other inequality. 
%\end{proof}
 
 \subsection{The Proof of Theorem \ref{stochhom}}\label{proofhom}
Having obtained $\overline{F}(M)$ with the appropriate limiting behavior of $w_{\ve}$, we employ the perturbed test function argument of Evans \cite{evanshom} to prove that this choice $\overline{F}(M)$ satisfies Theorem \ref{stochhom}.  

\begin{proof}[Proof of Thm \ref{stochhom}]
We define the half-relaxed limits 
\begin{equation*}
u^{*}(x,t,\om)=\limsup_{\ve\rightarrow 0 \atop y\rightarrow x, s\rightarrow t} u^{\ve}(y,s,\om) \quad\text{and}\quad u_{*}(x,t,\om)=\liminf_{\ve\rightarrow 0 \atop y\rightarrow x, s\rightarrow t} u^{\ve}(y,s,\om).
\end{equation*}

It is immediate that $u_{*}(\cdot, \cdot, \om)\leq u^{*}(\cdot, \cdot, \om)$. We claim that for all $\om\in \tilde{\Om}$, $u^{*}(\cdot,\cdot, \om)$ is a subsolution to \pref{limit}, and $u_{*}(\cdot,\cdot, \om)$ is a supersolution to \pref{limit}. 

Here we only show the proof for $u^{*}(x,t, \om)$. To this end, we fix $\om\in\tilde{\Om}$, and argue by contradiction assuming that there exists a smooth function $\vp(x,t)$ such that $u^{*}(\cdot, \cdot, \om)-\vp(\cdot, \cdot)$ has a strict local maximum at $(\hat{x},\hat{t})$, with $u^{*}(\hat{x},\hat{t}, \om)=\vp(\hat{x},\hat{t})$, but 
\begin{equation*}
\vp_{t}(\hat{x},\hat{t})-\overline{F}(D^{2}\vp(\hat{x},\hat{t}))=\sigma >0.
\end{equation*}

Let $M=D^{2}\vp(\hat{x}, \hat{t})$, and consider $w^{\ve}$ solving \pref{correct} with right hand side $-\overline{F}(M)$ in $Q_{1/\ve}(\hat{x}, \hat{t})$. We claim that
\begin{equation*}
\vp^{\ve}(x,t)=\vp(x,t)+\ve^{2}w^{\ve}(x/\ve,t/\ve^{2}, \om)
\end{equation*}
 is a subsolution to \pref{hom} in $Q_{r}(\hat{x}, \hat{t})$ for $r$ sufficiently small. Suppose that $\psi$ is a smooth function, and $\vp^{\ve}-\psi$ has a minimum at $(x_{\ve},t_{\ve})$ inside $Q_{r}(\hat{x}, \hat{t})$ ($r$ to be chosen later). It follows from \pref{correct} that
\begin{equation*}
\psi_{t}(x_{\ve}, t_{\ve})-\vp_{t}(x_{\ve},t_{\ve})-F(M+D^{2}\psi(x_{\ve}, t_{\ve})-D^{2}\vp(x_{\ve}, t_{\ve}), x_{\ve}/\ve, t_{\ve}/{\ve}^{2}, \om)\geq -\overline{F}(D^{2}\vp(\hat{x},\hat{t})) 
\end{equation*}
which implies that, for $r$ sufficiently small, 
\begin{equation*}
\psi_{t}(x_{\ve}, t_{\ve})-F(D^{2}\psi(x_{\ve}, t_{\ve}), x_{\ve}/\ve,  t_{\ve}/\ve^{2}, \om)\geq \vp_{t}(\hat{x}, \hat{t})-\overline{F}(D^{2}\vp(\hat{x},\hat{t}),\om)-\frac{\sigma}{2}=\frac{\sigma}{2}>0\quad\text{in}\quad Q_{r}(\hat{x}, \hat{t}).
\end{equation*}

The comparison principle for parabolic equations yields,
\begin{equation*}
\sup_{Q_r(\hat{x}, \hat{t})}(u^{\ve}-\vp^{\ve})^{+}\leq \sup_{\partial_{p}Q_{r}(\hat{x}, \hat{t})}(u^{\ve}-\vp^{\ve})^{+},
\end{equation*}
and after taking the appropriate $\limsup$ and using the decay condition of $w^{\ve}(\cdot, \cdot, \om)$, we have that 
\begin{equation*}
\sup_{Q}(u^{*}-\vp)^{+}\leq\sup_{\partial_{p}Q}(u^{*}-\vp)^{+}.
\end{equation*}
This contradicts having a strict maximum at $(\hat{x}, \hat{t})$. Thus, $u^{*}(\cdot, \cdot, \om)$ is a subsolution to \pref{limit}. 

A similar argument shows that $u_{*}(\cdot, \cdot, \om)$ is a supersolution to \pref{limit}. By the comparison principle, we have that $u^{*}(\cdot, \cdot, \om)=u_{*}(\cdot, \cdot, \om)=:u(\cdot, \cdot, \om)$, and $u(\cdot, \cdot, \om)$ solves \pref{limit}. By Proposition \ref{ellip}, $u(\cdot, \cdot, \om)$ is the unique solution to \pref{limit}. Since \pref{limit} is a deterministic equation, this means that $u(\cdot, \cdot, \om)=u(\cdot, \cdot)$ is independent of $\om$. Therefore, we may conclude that for all $\om\in \tilde{\Om}$, $u^{\ve}(\cdot, \cdot, \om)\rightarrow u(\cdot, \cdot)$ locally uniformly in $D_{T}$. A standard barrier argument \cite{parbar} and the uniform exterior cone condition allows us to conclude that the convergence is uniform on $\overline{D_{T}}$. 
\end{proof}

\section{A Rate of Decay of the Masses of the Obstacle Problem}\label{decaymass}
We prove a rate of decay for the product of the second moments of the masses of the obstacle problem from above and below, which holds for any choice of $\ell$. In light of \pref{dichot} Lemma \ref{alpos}, and Lemma \ref{al0},
\begin{equation*}
\begin{cases}
\ell\geq -\overline{F}(M)\Longrightarrow \overline{p}(\ell)=0,\\
\ell\leq -\overline{F}(M)\Longrightarrow \underline{p}(\ell)=0,
\end{cases}
\end{equation*}
which implies that $\overline{p}(\ell)\underline{p}(\ell)=0$ for every $\ell\in \RR$. Therefore, by \pref{pdef}, for each $\ell\in \RR$ and almost surely in $\om$, 
\begin{equation}\label{prodmeas}
\lim_{\ve\rightarrow 0} \overline{m}(Q_{1/\ve}, \ell, \om)\underline{m}(Q_{1/\ve}, \ell, \om)=0.
\end{equation}

The ultimate goal is to obtain a rate of decay on $\ve^{2}w^{\ve}$. As in the proof of Theorem \ref{stochhom}, this must come from studying the separation of $w^{\ve}$ from the obstacle problems from above and below. Recall that
% $\overline{v}^{\ve}-w^{\ve}$ and $w^{\ve}-\underline{v}^{\ve}$. Recall that
%\begin{align*}
%(w^{\ve}-\overline{v}^{\ve})_{s} -\MM^{-}(D^{2}w^{\ve}-D^{2}\overline{v}^{\ve}) &\geq(w^{\ve}-\overline{v}^{\ve})_{s}+ F(D^{2}\overline{v}^{\ve}, y, s, \om)-F(D^{2}w^{\ve}, y, s, \om)\\
%&\geq(\ell+F(0, y, s, \om))_{-}\chi_{\left\{\overline{v}^{\ve}=0\right\}}
%\end{align*}
%and $w^{\ve}-\overline{v}^{\ve}\rst{\partial_{p}Q_{1/\ve}}=0$. By the parabolic ABP-estimate, 
\begin{align*}
\sup_{Q_{1/\ve}}(-\ve^{2}w^{\ve}(\cdot, \cdot,\om))\leq \sup_{Q_{1/\ve}}(\ve^{2}\overline{v}^{\ve}(\cdot, \cdot, \om)-\ve^{2}w^{\ve}(\cdot, \cdot, \om))\leq C\left(\frac{1}{|Q_{1/\ve}|}\int_{\left\{\overline{v}^{\ve}=0\right\}} (\ell+F(0, y, s, \om))_{-}^{d+1}~dyds\right)^{1/d+1},
\end{align*}
and
\begin{align*}
\sup_{Q_{1/\ve}}(\ve^{2}w^{\ve}(\cdot, \cdot,\om))\leq \sup_{Q_{1/\ve}}(\ve^{2}w^{\ve}(\cdot, \cdot, \om)-\ve^{2}\underline{v}^{\ve}(\cdot, \cdot, \om))\leq C\left(\frac{1}{|Q_{1/\ve}|}\int_{\left\{\underline{v}^{\ve}=0 \right\}} (\ell+F(0, y, s, \om))_{+}^{d+1}~dyds\right)^{1/d+1}.
\end{align*}
This suggests that we should study the limiting behavior of the right hand sides of the former expressions. 

Moreover, instead of studying the above limit $\om$-by-$\om$, we obtain a rate of decay for the second moments of these quantities. In other words, we choose to study
\begin{equation}\label{somelim}
\lim_{\ve\rightarrow 0}\EE\left[\left(\frac{1}{|Q_{1/\ve}|}\int_{\left\{\overline{v}^{\ve}=0\right\}} (\ell+F(0, y, s, \om))_{-}^{d+1}~dyds\right)^{2}\right]\EE\left[\left(\frac{1}{|Q_{1/\ve}|}\int_{\left\{\underline{v}^{\ve}=0\right\}} (\ell+F(0, y, s, \om))_{+}^{d+1}~dyds\right)^{2}\right].
\end{equation}

\subsection{Notation, Set-Up, Statement of the Main Theorem}
We work with $G_{k}=C_{3^{k}}$, parabolic cubes with sides of length $3^{k}$ in space and $3^{2k}$ in time.  Unless otherwise stated, $G_{k}$ will always have top point $(0,0)$. We refer to $\overline{v}^{k}$ and $\underline{v}^{k}$ as the solutions to the obstacle problem from above and below in the cube $G_{k}$, respectively. We note that because we are working with domains which expand in $k$, the functions $F(\cdot, y, s, \om)$ are used. Also, $M$ and $\ell$ are fixed throughout this section, so we drop their dependencies to keep the notation as simple as possible. 

We define the total masses of the obstacle problems (from above and below) to be the random variables
\begin{equation*}
\overline{\pi}_{k}(\om)=\frac{1}{|G_{k}|}\int_{\left\{\overline{v}^{k}=0\right\}} (\ell+F(0, y, s, \om))_{-}^{d+1}~dyds \quad\text{and}\quad \underline{\pi}_{k}(\om)=\frac{1}{|G_{k}|}\int_{\left\{\underline{v}^{k}=0\right\}} (\ell+F(0, y, s, \om))_{+}^{d+1}~dyds.
\end{equation*}
We aim to study the decay of \pref{somelim} ``globally in $\om$." We consider the second moments of $\overline{\pi}_{k}(\om)$ and $\underline{\pi}_{k}(\om)$ defined by 
\begin{equation}\label{jdef}
\overline{J}_{k}=\EE [\overline{\pi}_{k}^{2}]\quad\text{\and}\quad \underline{J}_{k}=\EE[\underline{\pi}_{k}^{2}].
\end{equation}

We also define the quantities
\begin{equation}
\overline{E}_{k}=\EE[\overline{\pi}_{k}]\quad\text{and}\quad \underline{E}_{k}=\EE[\underline{\pi}_{k}],
\end{equation}
and 
\begin{equation}
\overline{V}_{k}=\overline{J}_{k}-(\overline{E}_{k})^{2}\quad\text{and}\quad \underline{V}_{k}=\underline{J}_{k}-(\underline{E}_{k})^{2}.
\end{equation}
We will frequently write equations with the terms $\overline{\underline{V}}_{k}, \overline{\underline{E}}_{k},$ and $\overline{\underline{J}}_{k}$, to signify that the equations hold for the quantities from above and below, respectively. 

Later in this section, we verify some precise statements regarding the decay of $\overline{J}_{k}$ and $\underline{J}_{k}$. For now, we point out that by the monotonicity of the obstacle problem, these are non-increasing quantities in $k$. We now state the main result of this section:

\begin{theorem}\label{thmrate1}
Assume \pref{f1}-\pref{f5}. There exists universal constants $\tau=\tau(\la, \La, d)\in(0,1)$, $C=C(\la, \La, d)>0$ and a positive integer $k_{0}=k_{0}(\la, \La, d)$ such that for $k\geq k_{0}$, 
\begin{equation}\label{ratek}
\mathbb{J}_{k^{3}}=\overline{J}_{k^{3}}\underline{J}_{k^{3}}\leq C(1+\norm{M}+\ell)^{4(d+1)} 3^{(k_{0}-k)\tau}
\end{equation}
\end{theorem}

\subsection{An outline of the proof of Theorem \ref{thmrate1}}
We provide a heuristic outline of the proof of Theorem \ref{thmrate1} in order to clarify the ideas in the rest of this section. We work inductively, by considering three different scales for parabolic cubes, which we denote by $G_{b}$, $G_{m}$ and $G_{s}$ to represent the big scale, middle scale, and small scale. We assume that \pref{ratek} holds for the small scale $G_{s}$. Given information about the small scale and the middle scale, there are three possible outcomes:\\
\underline{Case 1:}  If either $\mathbb{J}_{s}$ or $\mathbb{J}_{m}$ is less than some critical level, corresponding to the desired rate of $\mathbb{J}_{b}$, then by the non-increasing property of $\overline{\underline{J}}_{k}$ and $\mathbb{J}_{k}$, we may conclude. \\
\underline{Case 2:}
We suppose there exists some $\eta\in (0,1)$, to be chosen, such that $\overline{\underline{V}}_{m}\geq \eta \overline{\underline{J}}_{m}$, or $\overline{\underline{V}}_{s}\geq \eta \overline{\underline{J}}_{s}$. By applying the mixing hypothesis \pref{gmixing} to \pref{probfact}, we obtain a decay in the variances, which in this case yields a decay in the second moments. This allows us to conclude. \\
\underline{Case 3:} We address the case if both $\overline{\underline{J}}_{m}$ and $\overline{\underline{J}}_{s}$ are above the critical level in Case 1, and $\overline{\underline{V}}_{m}$ and $\overline{\underline{V}}_{s}$ are small compared to $\overline{\underline{J}}_{m}$ and $\overline{\underline{J}}_{s}$ respectively. Chebyshev's inequality yields that for most $\om\in \Om$, the random variables $\overline{\underline{\pi}}_{m}(\om)$ and $\overline{\underline{\pi}}_{s}(\om)$  are approximately constant and equal to $\overline{\underline{E}}_{m}$ and $\overline{\underline{E}}_{s}$ respectively. Therefore, a decay in the random variables $\overline{\underline{\pi}}_{m}(\cdot)$ to $\overline{\underline{\pi}}_{s}(\cdot)$ would be enough to conclude for these choices of $\om$. Since we are not in Case 1 or 2, $\overline{\pi}_{m}$ and $\underline{\pi}_{m}$ have explicit low bounds. Using these explicit lower bounds and the regularity estimates from \cite{linreg1}, we show that $\overline{v}^{m}$ and $\underline{v}_{m}$ are strictly separated in each subcube $G^{i}_{s}$ of $\frac{2}{3}G_{m}$. This implies there must be a decay in the measure of the contact sets from the small scale to the middle scale, and hence a decay from $\overline{\pi}_{s}$ to $\overline{\pi}_{m}$ or $\underline{\pi}_{s}$ to $\underline{\pi}_{m}$. This is enough to yield a rate on the second moments, and we may conclude. 

\subsection{Some Simplifying Observations and Notation}
We now assign specific values for the levels of $b, m, s$ in the outline above. The three scales we consider are $b=(k+1)^{3}$, $m=k^{3}+3k^{2}$ and $s=k^{3}$. Notice that $G_{(k+1)^{3}}$ is subdivided into $3^{(3k+1)(d+2)}$ subcubes $G^{i}_{k^{3}+3k^{2}}$, and then each of those subcubes is further divided into $3^{3k^{2}(d+2)}$ subcubes $G^{ij}_{k^{3}}$. Whenever any of the quantities such as $\overline{v}^{\al}$ and $\overline{\pi}^{\beta}_{\al}$ have indices on them, this will represent solving the obstacle problem in the domain $G^{\beta}_{\al}$, and working with the respective quantities there. 

Using these scales, we modify the statement of \pref{gmixing} for $r>3^{k^{2}}$. At the scales we have chosen, we may represent \pref{gmixing} as the following:
\begin{equation}\label{1mixing}
\sup_{A\in \mathcal{B}, B\in \mathcal{B}(r)} |\PP[A\cap B]-\PP[A]\PP[B]|\leq 3^{-k^{3/2}}\quad\text{if}\quad r>3^{k^{3}}.
\end{equation}

We also define the averaged quantities
\begin{equation}\label{avgdef}
\overline{\underline{A}}^{i}_{k^{3}+3k^{2}}=3^{-(3k^{2})(d+2)}\sum_{j=1}^{3^{k^{2}(d+2)}} \overline{\underline{\pi}}^{ij}_{k^{3}}(\omega)\quad\text{and}\quad\overline{\underline{A}}_{(k+1)^{3}}=3^{-(3k+1)(d+2)}\sum_{i=1}^{3^{(3k+1)(d+2)}} \overline{\underline{\pi}}^{i}_{k^{3}+3k^{2}}(\omega).
\end{equation}
These quantities naturally arise when applying \pref{gmixing} to \pref{probfact}. 

We note that by applying the stationarity property \pref{f1}, for $i=1, 2, \ldots, 3^{(3k+1)(d+2)}$, and $j=1, 2, \ldots, 3^{3k^{2}(d+2)}$, we have that
\begin{equation*}
\EE[\overline{\underline{\pi}}^{ij}_{k^{3}}]=\overline{\underline{E}}_{k^{3}}\quad\text{and}\quad \EE[\overline{\underline{\pi}}^{i}_{k^{3}+3k^{2}}]=\overline{\underline{E}}_{k^{3}+3k^{2}},
\end{equation*}
which implies that
\begin{equation*}
\EE[\overline{\underline{A}}_{k^{3}+3k^{2}}]=\overline{\underline{E}}_{k^{3}}\quad\text{and}\quad \EE[\overline{\underline{A}}_{(k+1)^{3}}]=\overline{\underline{E}}_{k^{3}+3k^{2}}.
\end{equation*}
Finally, recall that by \pref{f4},  $(\ell+F_{M}(0, y, s, \om))_{\pm}\leq C(1+\norm{M}+\ell)$. By a standard rescaling argument (see \cite{csw}), we may assume that for all $(y,s)\in\RR^{d+1}$, 
\begin{equation*}
(\ell+F(0, y, s, \om))_{\pm}\leq 1
\end{equation*}
which implies
\begin{equation}\label{simple}
0\leq \hat{\overline{\underline{\pi}}}_{k}\leq 1, ~~0\leq \overline{\underline{J}}_{k^{3}}\leq 1\quad\text{and}\quad 0\leq\overline{\underline{J}}_{k^{3}+3k^{2}}\leq 1.
\end{equation}
Using this rescaling, \pref{ratek} is equivalent to showing there exists $\tau=\tau(\la, \La, d)\in (0,1)$, $k_{0}=k_{0}(\la, \La, d)$ such that 
\begin{equation}\label{newrate}
\mathbb{J}_{k^{3}}\leq 3^{(k_{0}-k)\tau}
\end{equation}

\subsection{Technical Lemmas}
We next state the relevant lemmas needed to prove Theorem \ref{thmrate1}. We comment that some of the statements, and their proofs, are direct parabolic analogues of the corresponding results in \cite{cs}. We include them in the Appendix \ref{techlem} for completeness. However, the major difficulty in the parabolic setting is the argument in Case 3 to show that $\overline{v}^{m}$ and $\underline{v}^{m}$ are strictly separated. In the elliptic setting, it was shown in \cite{cs, csw} that the Fabes-Stroock estimate \cite{fs} yields a quantitative lower bound for $\overline{v}^{m}-\underline{v}^{m}$. The parabolic version of this quantitative regularity estimate was established in full generality in \cite{linreg1}. For this application, we use a special case of \cite{linreg1} to conclude that $\overline{v}^{m}-\underline{v}^{m}$ remains strictly positive in a subdomain of $G_{m}$. 

The first lemma we state is regarding the monotonicity of the obstacle problem, and the natural monotonicity inherited by $\overline{\underline{J}}_{k}$, and $\mathbb{J}_{k}$. 
\begin{lemma}\label{lemmonot}
For $k\geq 1$, 
\begin{enumerate}
\item$\overline{\underline{J}}_{k^{3}}\geq \overline{\underline{J}}_{k^{3}+3k^{2}}\geq\overline{\underline{J}}_{(k+1)^{3}},$\\
\item $ \mathbb{J}_{k^{3}} \geq \mathbb{J}_{k^{3}+3k^{2}}\geq \mathbb{J}_{(k+1)^{3}}$.
\end{enumerate}
\end{lemma}
This follows by a direct computation and employing the monotonicity of the obstacle problem. 

The next lemma addresses Case 2. In order to apply \pref{probfact}, we introduce the averaged quantities $\overline{\underline{A}}_{k}$ defined by \pref{avgdef}. 
\begin{lemma}\label{decayv}
For all $k\geq 1$, 
\begin{equation}\label{decays}
\mathbb{V}[\overline{\underline{A}}_{k^{3}+3k^{2}}]\leq 3^{-(3k^{2})(d+2)+(d+1)}\overline{\underline{V}}_{k^{3}}+3^{-k^{3/2}},
\end{equation}
and
\begin{equation}\label{decayb}
\mathbb{V}[\overline{\underline{A}}_{(k+1)^{3}}]\leq 3^{-(3k+1)(d+2)+(d+1)}\overline{\underline{V}}_{k^{3}+3k^{2}}+3^{-k^{3/2}}.
\end{equation}
\end{lemma}

For Case 3, we first justify the claim that $\overline{\underline{\pi}}_{k}(\om)$ is comparable to $\overline{\underline{E}}_{k}$, for most $\om$'s. Whenever we are in the case where,
\begin{equation}\label{badcase}
\overline{\underline{V}}_{k^{3}}\leq \eta(\overline{\underline{E}}_{k^{3}})^{2}\quad\text{and}\quad \overline{\underline{V}}_{k^{3}+3k^{2}}\leq \eta(\overline{\underline{E}}_{k^{3}+3k^{2}})^{2},
\end{equation}
then by Chebyshev's inequality, for $j=k^{3}$ or $k^{3}+3k^{2}$, 
\begin{equation*}
\PP\left[ |\overline{\underline{\pi}}_{j}-\overline{\underline{E}}_{j}|\geq \frac{\overline{\underline{E}}_{j}}{2}\right]\leq \PP\left[|\overline{\underline{\pi}}_{j}-\overline{\underline{E}}_{j}|\geq \sqrt{\frac{\overline{\underline{V}}_{j}}{4\eta}}\right]\leq 4\eta.
\end{equation*}

In Case 3, we would like to show that $\overline{\underline{\pi}}^{ij}_{k^{3}}$ is controlled by the expectation. By the above, this is true for each $G^{ij}_{k^{3}}$, up to certain exceptional sets. However, in order to understand the behavior at a global level, we define
%\[
%  \overline{\underline{\rho}}^{ij}_{k^{3}}(\om) = \left\{
%  \begin{array}{l l}
  %1 & \quad \text{if}\quad \overline{\underline{\pi}}^{ij}_{k^{3}}(\om)\leq \frac{1}{2}\overline{\underline{E}}_{k^{3}},\\
   %0 & \quad \text{otherwise}. \\
 % \end{array} \right.
%\]
%In average, we define
\begin{equation*}
\overline{\underline{a}}^{i}_{k^{3}+3k^{2}}(\om)=3^{-(3k^{2}(d+2))}\sum_{j=1}^{3^{3k^{2}(d+2)}} \mathbbm{1}_{\left\{ \overline{\underline{\pi}}^{ij}_{k^{3}}(\cdot)\leq \frac{1}{2}\overline{\underline{E}}_{k^{3}}\right\}}(\om),
\end{equation*}
and 
\begin{equation*}
\overline{\underline{N}}^{i}_{k^{3}+3k^{2}}(\om)=\sum_{j=1}^{3^{3k^{2}(d+2)}} \left(1-\mathbbm{1}_{\left\{ \overline{\underline{\pi}}^{ij}_{k^{3}}(\cdot)\leq \frac{1}{2}\overline{\underline{E}}_{k^{3}}\right\}}(\om)\right).
\end{equation*}
We point out that $\overline{N}^{i}_{k^{3}+3k^{2}}$ represents the number of subcubes where $\overline{\underline{\pi}}^{ij}_{k^{3}}(\om)>\frac{1}{2}\overline{\underline{E}}_{k^{3}}$.

%\begin{equation}
%\langle \overline{\underline{\rho}}^{i}_{k^{3}+3k^{2}}\rangle(\om)=3^{-(3k^{2}(d+2))}\sum_{j=1}^{3^{3k^{2}(d+2)}}\overline{\underline{\rho}}_{k^{3}}^{ij}(\om).
%\end{equation}

%We set $\overline{\underline{N}}^{i}_{k^{3}}(\om)=\left\{\#\quad \text{subcubes where}\quad \overline{\underline{\rho}}^{ij}_{k^{3}}(\om)=0\right\}=\left\{\#\quad\text{of subcubes where}\quad \overline{\underline{\pi}}^{ij}_{k^{3}}(\om)>\frac{1}{2}\overline{\underline{E}}_{k^{3}}\right\}$. 

Moreover, we consider
\begin{equation}\label{zeta}
\zeta_{d}=\frac{1}{4}\left(\frac{2}{3}\right)^{d+1}, 
\end{equation}
where the choice for $\zeta_{d}$ will become clear in the proof of Theorem \ref{thmrate1}. We consider the sets
\begin{equation}
\langle \overline{\underline{B}}_{k^{3}}\rangle =\left\{\om\in \Om: \langle\overline{\underline{a}}^{ij}_{k^{3}}\rangle(\om)>\zeta_{d}\right\}.
\end{equation}

We now justify the claim that $\overline{\underline{\pi}}_{k^{3}+3k^{2}}$ are ``almost constant" up to these exceptional sets:
\begin{lemma}\label{excsets}
Suppose \pref{badcase} holds. For every $k$, there exists subsets $\overline{\underline{B}}_{k^{3}+3k^{2}}\subset\Om$ so that
\begin{equation}\label{chebym}
\overline{\underline{\pi}}_{k^{3}+3k^{2}}\in \left[\frac{1}{2}\overline{\underline{E}}_{k^{3}+3k^{2}}, \frac{3}{2}\overline{\underline{E}}_{k^{3}+3k^{2}}\right]\quad\text{in}\quad \Om\setminus \overline{\underline{B}}_{k^{3}+3k^{2}}\quad\text{and}\quad\PP[\overline{\underline{B}}_{k^{3}+3k^{2}}]\leq 4\eta.
\end{equation}

Also, 
\begin{equation}
\PP[\langle \overline{\underline{B}}_{k^{3}}\rangle] < \zeta_{d}^{-1}\eta\quad\text{and}\quad\overline{\underline{N}}_{k^{3}}\geq (1-\zeta_{d})3^{3k^{2}(d+2)}\quad\text{in}\quad \Om\setminus \langle \overline{\underline{B}}_{k^{3}}\rangle.
\end{equation}
If 
\begin{equation*}
\textbf{B}_{k}=\overline{B}_{k^{3}+3k^{2}}\cup \underline{B}_{k^{3}+3k^{2}}\cup \langle\overline{B}_{k^{3}}\rangle\cup \langle \underline{B}_{k^{3}}\rangle,
\end{equation*} then we have that
\begin{equation}
\PP[\textbf{B}_{k}]\leq 8\eta (1+ \zeta_{d}).
\end{equation}
\end{lemma}

Up to now, these lemmas have been direct parabolic analogues of the results of \cite{cs}. We now begin to explain the most delicate part of the rate computation in the parabolic setting, the strict separation of $\overline{v}^{m}$ and $\underline{v}^{m}$. 

We consider $h^{m}(y,s, \om)=\overline{v}^{m}(y,s, \om)-\underline{v}^{m}(y,s,\om)$ is a nonnegative supersolution to 
\begin{equation}\label{diff}
h_{s}-\MM^{-}(D^{2}h)\geq f_{m}:=(\ell+F(0, y, s, \om))_{+}\chi_{\left\{\overline{v}^{m}=0\right\}}+(\ell+F(0, y, s, \om))_{-}\chi_{\left\{\underline{v}^{m}=0\right\}}. 
\end{equation}
A direct application of Theorem \ref{genfs} (with $\ka=2/3$) yields a constant $c=c(\la, \La, d)$ so that for all $|y|\leq \frac{2}{3} 3^{m}$, $0\geq s\geq -\frac{2}{3} \frac{\norm{f_{m}}^{d+1}_{L^{d+1}(G_{m})}}{2}3^{-md}$, for all $\om\in\Om$
\begin{equation}\label{genlb}
c3^{m(2-\left(d+2\right)\al)}\norm{f_{m}}^{1-(d+1)\al}_{L^{\infty}(G_{m})}\norm{f_{m}}_{L^{d+1}(G_{m})}^{(d+1)\al}\leq h^{m}(y,s, \om).
\end{equation}
We point out that this estimate is not good enough for the application we have in mind. The domain where \pref{genlb} holds depends on $\norm{f_{m}}_{L^{d+1}(G_{m})}$, which is a consequence of the causality property of parabolic equations. In particular, $f_{m}=f_{k^{3}+3k^{2}}$ decays as $k\rightarrow\infty$, and therefore we have an estimate which worsens in a domain which is also shrinking. If one tries to continue the argument of \cite{cs} with \pref{genlb}, the argument breaks down. We must find an estimate which holds in a fixed fraction of the cube $G_{m}$. 

In light of the remarks in Appendix \ref{quantreg}, it is enough to show that there exists a suitable choice of $t_{0}$ so that 
\begin{equation*}
\Ga(\om):=\left\{(y,s)\in G_{m}: f_{m}(y,s,\om)>\frac{1}{2|G_{m}|^{1/(d+1)}}\norm{f_{m}}_{L^{d+1}(G_{m})}\right\}\subset\left\{t\leq t_{0}\right\}. 
\end{equation*}
In order to do so, we will apply the subadditive ergodic theorem (Theorem \ref{genset}), in order to show that $\Ga(\om)$ spreads all over the cylinder $G^{m}$ almost surely. 

We define $G^{'}_{m}=(0, 3^{m})^{d}\times (-3^{2m}, -\frac{2}{3}\cdot3^{2m}]\subset G_{m}$. We note that by construction that $\partial_{p}G^{'}_{m}\subset \partial_{p}G_{m}$, which by Lemma \ref{monobst} implies that the corresponding obstacle problems in $G^{'}_{m}$ have the same contact sets as those for $G_{m}$. In other words, $\left\{\overline{v}^{'m}=0\right\}=\left\{\overline{v}^{m}=0\right\}\cap G^{'}_{m}$. In particular, $\overline{\underline{m}}(G^{'}_{m}, \ell, \om)=|\left\{\overline{\underline{v}}^{m}=0\right\}\cap G^{'}_{m}|.$ 

We note that since $\overline{\underline{m}}(G_{m}, \ell, \om)$ is a stationary, subadditive quantity, it follows by the same arguments of Proposition 3.1 and \pref{simple} that 
\begin{equation*}
R_{G_{m}}(\om):=\int f_{m}^{d+1}~dyds=\int (\ell+F(0, y, s, \om))^{d+1}_{+}\chi_{\left\{\overline{v}^{m}=0\right\}}+(\ell+F(0, y, s, \om))^{d+1}_{-}\chi_{\left\{\underline{v}^{m}=0\right\}}~dyds
\end{equation*}
is a stationary, subadditive process. The crucial step in the next lemma is to apply Theoreom \ref{genset} to $R_{G_{m}}(\om)$. 

\begin{lemma}\label{lemmaka}
There exists positive an event of full probability, $\Om_{1}$, such that for all $\om\in\Om_{1}$, 
\begin{equation}\label{frac}
\left| \left\{f_{j}(\cdot, \cdot, \om)> \frac{1}{(4|G_{j}|)^{1/(d+1)}}\norm{f_{j}}_{L^{d+1}(G_{j})}\right\}\cap G^{'}_{j} \right| \geq \frac{1}{12}\norm{f_{j}}^{d+1}_{L^{d+1}(G_{j})}.
\end{equation} 
for $j$ sufficiently large. 
\end{lemma}
 \begin{proof}
 We first note that by Theorem \ref{genset}, there exists an event $\Om_{1}$ of full probability, such that $\frac{R_{I}(\om)}{|I|}$ converges as $|I|\rightarrow\infty$. 

For the purposes of contradiction, we suppose that there exists a subsequence $\left\{j_{n}\right\}\rightarrow\infty$ (perhaps depending on $\om$), so that 
 \begin{equation}\label{wrong}
\left| \left\{f_{j}(\cdot, \cdot, \om)> \frac{1}{(4|G_{j}|)^{1/(d+1)}}\norm{f_{j}}_{L^{d+1}(G_{j})}\right\}\cap G^{'}_{j} \right| < \frac{1}{12}\norm{f_{j}}^{d+1}_{L^{d+1}(G_{j})}
\end{equation}
for all $j_{n}$ in this subsequence. To simplify the notation, let $H_{j}=\frac{1}{(4|G_{j}|)^{1/(d+1)}}\norm{f_{j}}_{L^{d+1}(G_{j})}$. We now compute
 \begin{align*}
 \int_{G^{'}_{j}}|f_{j}|^{d+1} dyds&=\int_{G^{'}_{j}\cap \left\{f_{j}> H_{j}\right\}}|f_{j}|^{d+1} dy ds+ \int_{G^{'}_{j}\cap \left\{f_{j}\leq H_{j}\right\}} |f_{j}|^{d+1}dy ds\\
 &\leq \left| \left\{f_{j}> H_{j}\right\}\cap G^{'}_{j} \right|+  \frac{|G^{'}_{j}|}{4|G_{j}|}\norm{f_{j}}^{d+1}_{L^{d+1}(G_{j})}\\
 &< \frac{1}{12}\norm{f_{j}}^{d+1}_{L^{d+1}(G_{j})}+\frac{|G^{'}_{j}|}{4|G_{j}|}\norm{f_{j}}^{d+1}_{L^{d+1}(G_{j})}\\
 &=\frac{1}{4}\frac{|G^{'}_{j}|}{|G_{j}|} \norm{f_{j}}^{d+1}_{L^{d+1}(G_{j})}+\frac{|G^{'}_{j}|}{4|G_{j}|}\norm{f_{j}}^{d+1}_{L^{d+1}(G_{j})}
  \end{align*}
which implies 
 \begin{equation}
 \frac{1}{|G^{'}_{j}|}  \int_{G^{'}_{j}}|f_{j}|^{d+1} dyds<\frac{1}{2}\frac{1}{|G_{j}|} \norm{f_{j}}^{d+1}_{L^{d+1}(G_{j})}.
 \end{equation}
 
Since $\left\{\overline{v}^{j}=0\right\}=\left(\left\{\overline{v}^{'j}=0\right\}\cap G^{'}_{j}\right)$, we have that that $f_{j}=f^{'}_{j}$ in $G^{'}_{j}$. This implies
 \begin{equation}\label{contr}
 \frac{R_{G^{'}_{j}}}{|G^{'}_{j}|}<\frac{1}{2} \frac{R_{G_{j}}}{|G_{j}|}.
 \end{equation}
 
 Finally, we note that since $\left\{G^{'}_{j}\right\}$ and $\left\{G_{j}\right\}$ form nested, increasing families of cubes, by Theorem \ref{genset}, both sides of \pref{contr} converge almost surely to the same number as $j_{n}\rightarrow \infty$, and we obtain a contradiction. 
 \end{proof}
 
Lemma \ref{lemmaka} shows that 
 \begin{equation}\label{inprob}
 \lim_{j\rightarrow\infty}\PP\left[\left| \left\{f_{j}(\cdot, \cdot, \om)> \frac{1}{(4|G_{j}|)^{1/(d+1)}}\norm{f_{j}}_{L^{d+1}(G_{j})}\right\}\cap G^{'}_{j} \right| \geq \frac{1}{12}\norm{f_{j}}^{d+1}_{L^{d+1}(G_{j})}\right]=0. 
 \end{equation}
 %In particular, there for any $\ve>0$, there exists $j_{0}$ such that for all $j\geq j_{0}$
 %\begin{equation}
 %\PP\left[\left| \left\{f_{j}(\cdot, \cdot, \om)> \frac{1}{(4|G_{j}|)^{1/(d+1)}}\norm{f_{j}}_{L^{d+1}(G_{j})}\right\}\cap G^{'}_{j} \right| \geq \frac{1}{12}\norm{f_{j}}^{d+1}_{L^{d+1}(G_{j})}\right]<\ve. 
 %\end{equation}
 
As a consequence of Lemma \ref{lemmaka}, and in light of the remarks made in Appendix \ref{quantreg}, we have have
 \begin{proposition}\label{FSL}
There exists positive constants $c_{fs}=c_{fs}(\la, \La, d)$, $\rho=\rho(\la, \La, d)$, and $\beta=\beta(\la, \La, d)$ such that for all $\om\in\Om, j\in\mathbb{N}$ satisfying \pref{frac}, for all $|y|\leq \frac{2}{3} 3^{j}$, for all $-\frac{2}{3}3^{-2j} \leq s\leq0$,
\begin{equation}
h_{j}(y,s, \om)\geq  c_{fs} \norm{f_{j}}^{d+1}_{L^{(d+1)\al}(G_{j})},
\end{equation}
with $\al=\rho+\beta \log\left(2|G_{j}|/\norm{f}_{L^{d+1}(G_{j})}^{d+1}\right)$.
 \end{proposition}
 
 %\begin{proof}
 %We claim that by scaling Theorem \ref{TSFS} and applying Lemma \ref{lemmaka}, and letting $\ka=\frac{1}{3}$, we have that for all $j\geq j_{0}$, 
 %\begin{equation*}
%g(y,s)\geq C_{FS}  \left| \left\{f_{j}> \frac{1}{3\cdot4^{1/(d+1)}\cdot 3^{j(d+2)/(d+1)}}\norm{f_{j}}_{L^{d+1}(G_{j})}\right\}\right|^{\al}\geq C_{FS}\left(\frac{1}{12}\norm{f_{j}}^{d+1}_{L^{d+1}(G_{j})}\right)^{\al}
 %\end{equation*}
 %for $\al=\beta \log\left(\frac{2|G_{j}|}{\norm{f}_{L^{d+1}(G_{j})}^{d+1}}\right)$, and for all $|y|\leq \frac{2}{3} 3^{j}$, $-\frac{2}{3}3^{-2m} \leq s \leq0$. This completes the proof. 
%\end{proof}

Finally, we state the result we will use in the proof of Theorem \ref{thmrate1}, which demonstrates the strict separation which will happen in Case 3. 
\begin{lemma}\label{fshomog}
Let $\om\in \Om_{1}, k\in \mathbb{N}$ such that \pref{frac} holds for $j=k^{3}+3k^{2}$, and assume that the total masses of the obstacle problems in $G_{k^{3}+3k^{2}}$ satisfy $\overline{\pi}_{k^{3}+3k^{2}}(\om)\underline{\pi}_{k^{3}+3k^{2}}(\om)\geq \theta$ for some $\theta>0$. Let $\ell, \theta$ are so that $2\frac{C_{h}}{c_{fs}}\leq 3^{3k^{2}(2+\sig-d/d+1)}3^{k^{3}(2-d/d+1)} \theta^{\al/2}$ where $\sigma$ is the Holder exponent which comes from the Holder continuity of the obstacle problem, $C_{h}$ is the Holder norm affiliated with the obstacle problem, and $\al, c_{fs}$ are as in Proposition \ref{FSL}. Then it is not possible for both $\overline{v}_{k^{3}+3k^{2}}(\cdot, \cdot, \om)$ and $\underline{v}_{k^{3}+3k^{2}}(\cdot, \cdot, \om)$ to vanish in any of the subcubes $G^{i}_{k^{3}}$ that subdivide $\frac{2}{3} G_{k^{3}+3k^{2}}=\left(-\frac{2}{3}3^{k^{3}+3k^{2}},\frac{2}{3}3^{k^{3}+3k^{2}}\right)^{d}\times\left(-\frac{2}{3}3^{2(k^{3}+3k^{2})}, 0\right]$. 
\end{lemma}

\subsection{Decay Estimate}
We now present the proof of Theorem \ref{thmrate1}. 
\begin{proof}[Proof of Theorem \ref{thmrate1}]
We seek $\tau\in(0,1)$, and $k_{0}\geq 1$ so that for all $k\geq k_{0}$, 
\begin{equation}\label{hypot}
\mathbb{J}_{k^{3}}\leq 3^{(k_{0}-k)\tau}.
\end{equation}
 
We proceed by induction on $k$, assuming \pref{hypot}. We would like to prove that for appropriate choices of $k_{0}, \tau$, 
\begin{equation}\label{decaygoal}
\mathbb{J}_{(k+1)^{3}}\leq 3^{(k_{0}-k-1)\tau}.
\end{equation}

If we are in the situation where either
\begin{equation*}
\overline{\underline{J}}_{k^{3}}\leq 3^{(k_{0}-k)\tau-1}, \quad \overline{\underline{J}}_{k^{3}+3k^{2}}\leq 3^{(k_{0}-k)\tau-1},\quad\text{or}\quad \mathbb{J}_{k^{3}+3k^{2}}\leq 3^{(k_{0}-k)\tau-1},
\end{equation*}
then by the non-increasing property of $\overline{\underline{J}}_{k}$, \pref{decaygoal} is immediate. 

We will assume then that
\begin{equation}\label{whoa}
\overline{\underline{J}}_{k^{3}}> 3^{(k_{0}-k)\tau-1}, \quad \overline{\underline{J}}_{k^{3}+3k^{2}}> 3^{(k_{0}-k)\tau-1},\quad\text{and}\quad \mathbb{J}_{k^{3}+3k^{2}}> 3^{(k_{0}-k)\tau-1}.
\end{equation}

Suppose we are in Case 2, where there exists $\eta\in(0,1)$, to be chosen, so that either
\begin{equation*}
\overline{\underline{V}}_{k^{3}}\geq \eta (\overline{\underline{E}}_{k^{3}})^{2}\quad\text{or}\quad \overline{\underline{V}}_{k^{3}+3k^{2}}\geq \eta (\overline{\underline{E}}_{k^{3}+3k^{2}})^{2}.
\end{equation*}
This implies that
\begin{equation}\label{baduse}
\overline{\underline{J}}_{k^{3}}\geq (1+\eta)(\overline{\underline{E}}_{k^{3}})^{2}\quad\text{or}\quad \overline{\underline{J}}_{k^{3}+3k^{2}}\geq (1+\eta)(\overline{\underline{E}}_{k^{3}+3k^{2}})^{2}.
\end{equation}
If $\overline{\underline{V}}_{k^{3}}\geq \eta (\overline{\underline{E}}_{k^{3}})^{2}$, then by Lemma \ref{decayv}, \pref{whoa}
\begin{align*}
\overline{\underline{J}}_{(k+1)^{3}}&\leq \overline{\underline{J}}_{k^{3}+3k^{2}}\leq \EE[(\overline{\underline{A}}_{k^{3}+3k^{2}})^{2}]=V[\overline{\underline{A}}_{k^{3}+3k^{2}}]+(\overline{\underline{E}}_{k^{3}})^{2}\\
&\leq 3^{-(3k^{2})(d+2)+(d+1)}\overline{\underline{V}}_{k^{3}}+3^{-k^{3/2}}+(\overline{\underline{E}}_{k^{3}})^{2}\\
&\leq (1+\eta)^{-1}[1+ 3^{-(3k^{2})(d+2)+(d+1)}(1+\eta)]\overline{\underline{J}}_{k^{3}}+3^{-k^{3/2}}3^{(k-k_{0})\tau+1}\overline{\underline{J}}_{k^{3}}.
\end{align*}

Therefore, if there exists choices of $k_{0}, \tau$ so that for all $k\geq k_{0}$, 
\begin{equation*}
(1+\eta)^{-1}[1+ 3^{-(3k^{2})(d+2)+(d+1)}(1+\eta)]+3^{-k^{3/2}+(k-k_{0})\tau+1}\leq 3^{-\tau}, 
\end{equation*}
then \pref{decaygoal} holds. We note that the former expression is the same as requiring
\begin{equation*}
(1+\eta)^{-1}[1+ 3^{-(3k^{2})(d+2)+(d+1)}(1+\eta)]+3^{-k(k^{1/2}-\tau)}3^{-k_{0}\tau+1}\leq 3^{-\tau}.
\end{equation*}
Since $\tau<1$, we see that this holds so long as $k_{0}, \eta, \tau$ satisfy
\begin{equation}\label{const1}
(1+\eta)^{-1}[1+ 3^{-(3k_{0}^{2})(d+2)+(d+1)}(1+\eta)]+3^{-k_{0}\tau+1}\leq 3^{-\tau}.
\end{equation}

Finally, we suppose we are in Case 3, where both 
\begin{equation}\label{case3}
\overline{\underline{V}}_{k^{3}}\leq \eta (\overline{\underline{E}}_{k^{3}})^{2}\quad\text{and}\quad \overline{\underline{V}}_{k^{3}+3k^{2}}\leq \eta (\overline{\underline{E}}_{k^{3}+3k^{2}})^{2}, 
\end{equation}
which in light of \pref{baduse}, yields
\begin{equation}\label{contabove}
\overline{\underline{J}}_{k^{3}}\leq (1+\eta)(\overline{\underline{E}}_{k^{3}})^{2}\quad\text{and}\quad \overline{\underline{J}}_{k^{3}+3k^{2}}\leq (1+\eta)(\overline{\underline{E}}_{k^{3}+3k^{2}})^{2}.
\end{equation}

We claim that in this case, either $\overline{J}_{k^{3}+3k^{2}}$ will have a strict decay from $\overline{J}_{k^{3}}$, or $\underline{J}_{k^{3}+3k^{2}}$ will have a strict decay from $\underline{J}_{k^{3}}$, and this will be enough to recover \pref{newrate}. 

We recall the set $\textbf{B}_{k}$ from Lemma \ref{excsets}. Let $\textbf{D}_{k}\subset \Om_{1}$ be the set of $\om\in\Om_{1}$ such that \pref{frac} does not hold for $j=k^{3}+3k^{2}$. In light of \pref{inprob}, there exists a $k_{0}$ such that for all $k\geq k_{0}$, 
\begin{equation}\label{setcond}
\PP[\textbf{D}_{k}]\leq(1+\zeta^{-1}_{d})\eta. 
\end{equation}
%\begin{equation*}
%\textbf{D}_{k}=\left\{\om\in\Om: \left|\left\{f_{k^{3}+3k^{2}}(\cdot, \cdot, \om)> \frac{1}{(4|G_{k^{3}+3k^{2}}|)^{1/(d+1)}}\norm{f_{k^{3}+3k^{2}}}_{L^{d+1}(G_{k^{3}+3k^{2}})}\right\}\cap G^{'}_{k^{3}+3k^{2}} \right| \geq \frac{1}{12}\norm{f_{k^{3}+3k^{2}}}^{d+1}_{L^{d+1}(G_{k^{3}+3k^{2}})}\right\}
%\end{equation*}

We apply Lemma \ref{fshomog} to show
\begin{claim}\label{fsclaim}
In $\Om_{1}\setminus(\textbf{B}_{k}\cup\textbf{D}_{k})$, either $\overline{v}_{k^{3}+3k^{2}}$ or $\underline{v}_{k^{3}+3k^{2}}$ cannot both have a contact point in any of the subcubes $G^{i}_{k^{3}}$ in $\frac{2}{3} G_{k^{3}+3k^{2}}$. 
\end{claim}
For $\om\in\Om_{1}\setminus(\textbf{B}_{k}\cup\textbf{D}_{k})$, by \pref{chebym} and \pref{contabove}, 
\begin{align*}
\overline{\pi}_{k^{3}+3k^{2}}\underline{\pi}_{k^{3}+3k^{2}}&\geq \frac{1}{4}\overline{E}_{k^{3}+3k^{2}}\underline{E}_{k^{3}+3k^{2}}\geq \frac{1}{4(1+\eta)}(\overline{J}_{k^{3}+3k^{2}}\underline{J}_{k^{3}+3k^{2}})^{1/2}\\
&=\frac{1}{4(1+\eta)}\mathbb{J}_{k^{3}+3k^{2}}^{1/2}\geq \frac{1}{4(1+\eta)}3^{\frac{(k_{0}-k)\tau-1}{2}}.
\end{align*}
In the language, of Lemma \ref{fshomog}, $\theta= \frac{1}{4(1+\eta)}3^{\frac{(k_{0}-k)\tau-1}{2}}$. We just need to check that
\begin{equation*}
2\frac{C_{h}}{c_{fs}}\leq 3^{k^{3}(2-d/d+1)}3^{3k^{2}(2+\sig-d/d+1)} \theta^{\al/2}
\end{equation*}
 where $\sigma$ is the Holder exponent that comes from the obstacle problem, we may check the condition for $\al=\rho+\beta \log\left(\frac{2}{\overline{\pi}_{k^{3}+3k^{2}}+\underline{\pi}_{k^{3}+3k^{2}}}\right)$. We note that 
 \begin{align*}
&3^{k^{3}(2-d/d+1)}3^{3k^{2}(2+\sig-d/d+1)} \theta^{(\al/2)}\\
%&=3^{k^{3}(2-d/d+1)}3^{3k^{2}(2+\sig-d/d+1)} \left(\frac{1}{4(1+\eta)}3^{\frac{(k_{0}-k)\tau-1}{2}}\right)^{(\al/2)}\\
 &=3^{k^{3}(2-d/d+1)}3^{3k^{2}(2+\sig-d/d+1)} \left( \frac{1}{4(1+\eta)}3^{\frac{(k_{0}-k)\tau-1}{2}}\right)^{\frac{\rho}{2}+\frac{\beta}{2} \log\left(\frac{2}{\overline{\pi}_{k^{3}+3k^{2}}+\underline{\pi}_{k^{3}+3k^{2}}}\right)}\\
 %&\geq 3^{k^{3}(2-d/d+1)}3^{3k^{2}(2+\sig-d/d+1)} \left( \frac{1}{4(1+\eta)}3^{\frac{(k_{0}-k)\tau-1}{2}}\right)^{\frac{\rho}{2}+\frac{\beta}{2} \log\left(\frac{1}{(\overline{\pi}_{k^{3}+3k^{2}}\underline{\pi}_{k^{3}+3k^{2}})^{1/2}}\right)}\\
 &\geq  3^{k^{3}(2-d/d+1)}3^{3k^{2}(2+\sig-d/d+1)}\left( \frac{1}{4(1+\eta)}3^{\frac{(k_{0}-k)\tau-1}{2}}\right)^{\frac{\rho}{2}+\frac{\beta}{4} \log\left(\frac{1}{\theta}\right)}\\
 &=3^{3k^{2}(2+\sig-d/d+1)}3^{k^{3}(2-d/d+1)} \left( \frac{1}{4(1+\eta)}3^{\frac{(k_{0}-k)\tau-1}{2}}\right)^{\frac{\rho}{2}-\frac{\beta}{4}\log(\frac{1}{4(1+\eta)})-\frac{\beta}{4}{\frac{(k_{0}-k)\tau-1}{2}}\log 3}\\
 &\geq 3^{k^{3}(2-d/d+1)}3^{3k^{2}(2+\sig-d/d+1)}\left( \frac{1}{8}\right)^{\frac{\rho}{2}-\frac{\beta}{4}\log(\frac{1}{4(1+\eta)})-\frac{\beta}{4}{\frac{(k_{0}-k)\tau-1}{2}}\log 3}3^{\frac{(k_{0}-k)\tau-1}{2}\cdot \left( \frac{\rho}{2}-\frac{\beta}{4}\log(\frac{1}{4(1+\eta)})-\frac{\beta}{4}{\frac{(k_{0}-k)\tau-1}{2}}\log 3 \right)}.\\
 %&\geq 3^{k^{3}(2-d/d+1)}3^{3k^{2}(2+\sig-d/d+1)}3^{-\frac{\beta \log 3}{4}[\frac{(k_{0}-k)\tau-1}{2}]^{2}}
 %&=3^{k^{3}(2-d/d+1)}3^{3k^{2}(2+\sig-d/d+1)}3^{\frac{(k_{0}-k)\tau-1}{2}\cdot \left( \frac{\rho}{2}-\frac{\beta}{4}\log(\frac{1}{4(1+\eta)})-\frac{\beta}{4}{\frac{(k_{0}-k)\tau-1}{2}}\log 3 \right)}C(\eta)
 \end{align*}
The above calculation demonstrates that $3^{k^{3}(2-d/d+1)}3^{3k^{2}(2+\sig-d/d+1)} \theta^{(\al/2)}\sim 3^{p(k)}$, where $p(k)$ is cubic polynomial. Therefore, there exists $k_{0}$ such that for all $k\geq k_{0}$, 
\begin{equation}\label{const2}
3^{k^{3}(2-d/d+1)}3^{3k^{2}(2+\sig-d/d+1)} \theta^{(\al/2)} \geq 2\frac{C_{h}}{c_{fs}},
\end{equation}
which allows us to apply Lemma \ref{fshomog}. 

Now we return to the proof of the final rate. Once Claim \ref{fsclaim} holds, for each $\om\in \Om_{1}\setminus(\textbf{B}_{k}\cup\textbf{D}_{k})$, at least one of $\overline{v}^{k^{3}+3k^{2}}(\cdot, \cdot, \om)$ or $\underline{v}^{k^{3}+3k^{2}}(\cdot, , \cdot, \om)$ cannot have a contact point in half of the $(\frac{2}{3})^{d+1} 3^{3k^{2}(d+2)}=4\zeta_{d}3^{3k^{2}(d+2)}$ subcubes $G^{ij}_{k^{3}}$ which make up $\frac{2}{3}G_{k^{3}+3k^{2}}$. Therefore, for at least half of the $\om$'s in $\Om_{1}\setminus (\textbf{B}_{k}\cup\textbf{D}_{k})$, either $\overline{v}^{k^{3}+3k^{2}}(\cdot, \cdot, \om)$ or $\underline{v}^{k^{3}+3k^{2}}(\cdot, , \cdot, \om)$ cannot have a contact point in at least half of the subcubes $G^{ij}_{k^{3}}$. Without loss of generality, we assume that $\overline{v}^{k^{3}+3k^{2}}$ does not touch half of the subcubes. Therefore, there exists $\textbf{W}_{k}\subset \Om_{1}\setminus( \textbf{B}_{k}\cup\textbf{D}_{k})$ so that
\begin{equation}\label{measC}
\PP[\textbf{W}_{k}]\geq \frac{1}{2} \PP[\Om_{1}\setminus(\textbf{B}_{k}\cup\textbf{D}_{k})]\geq \frac{1}{2}\left[1-9(1+\zeta^{-1}_{d})\eta\right]=\frac{1}{2}\left[1-9\left(1+\left(\frac{3}{2}\right)^{(d+1)}\right)\eta\right]
\end{equation}
%and also, in light of \pref{inmeasure}, there exists a subset $\hat{W}\subset \Om$, a choice of $k_{0}$ such that for all $\om\in\Om\setminus \hat{W}_{k}$, $k\geq k_{0}$, \pref{frac} holds, and 
%\begin{equation}
%\PP[\Om\setminus \hat{W}_{k}]\geq \frac{1}{2}\left[1-8(1+\zeta^{-1}_{d})\eta\right]
%\end{equation}
where for all $\om\in\textbf{W}_{k}$, 
%Therefore, we may conclude that for all $\om\in\textbf{W}_{k}$, 
\begin{equation}\label{Cdef}
\overline{v}^{k^{3}+3k^{2}}\quad\text{does not vanish in at least}\quad 2\zeta_{d}3^{3k^{2}(d+2)}\quad\text{subcubes}\quad G^{ij}_{k^{3}}\quad\text{inside}\quad \frac{2}{3}G^{i}_{k^{3}+3k^{2}}.
\end{equation}
In light of \pref{measC}, we restrict
\begin{equation}\label{const3}
\eta\in \left(0, \frac{1}{9\left(1+(3/2)^{d+1}\right)}\right).
\end{equation}

We now proceed to show how Case 3 yields a strict decay in masses. We will show that for $\om\in \textbf{W}_{k}$, $\overline{\pi}_{k^{3}+3k^{2}}(\cdot, \cdot, \om)$ is strictly smaller than $\overline{A}_{k^{3}+3k^{2}}$ by some multiple of $\overline{E}_{k^{3}}$. We fix $\om\in \textbf{W}_{k}$, and we quantify the number of subcubes $G^{ij}_{k^{3}}$ so that $\overline{\pi}^{ij}(\om)_{k^{3}}\geq \frac{1}{2}\overline{E}_{k^{3}}$. By Lemma \ref{excsets}, there exists at least $(1-\zeta_{d})3^{3k^{2}(d+2)}$ subcubes where $\overline{\pi}_{k^{3}}(\om)\geq \frac{1}{2}\overline{E}_{k^{3}}$. Thus, there has to be at least $3\zeta_{d}3^{k^{2}(d+2)}=(1-\zeta_{d})3^{3k^{2}(d+2)}-(1-4\zeta_{d})3^{3k^{2}(d+2)}$ subcubes inside $\frac{2}{3}G_{k^{3}+3k^{2}}$ with $\overline{\pi}^{ij}_{k^{3}}(\om)\geq \frac{1}{2}\overline{E}^{ij}_{k^{3}}$. Combining this with \pref{Cdef}, there exists at least $\zeta_{d}3^{3k^{2}(d+2)}$ subcubes $G^{ij}_{k^{3}}$ so that $\overline{v}^{k^{3}+3k^{2}}$ stays strictly positive, and $\overline{\pi}^{ij}_{k^{3}}(\om)\geq \frac{1}{2}\overline{E}^{ij}_{k^{3}}$. We denote these subcubes by $j^{'}$ and the rest of the subcubes by $j^{*}$. 

This yields
\begin{align*}
%&=3^{-(k^{3}+3k^{2})(d+2)}\int_{G^{i}_{k^{3}+3k^{2}}\cap \left\{\overline{v}^{k^{3}+3k^{2}}=0\right\}} (\ell+F(0, y, s, \om))_{+}^{d+1}dy ds\\
\overline{\pi}^{i}_{k^{3}+3k^{2}}(\om)&=3^{-(k^{3}+3k^{2})(d+2)}\sum_{j=1}^{3^{3k^{2}(d+2)}}\int_{G^{i}_{k^{3}+3k^{2}}\cap \left\{\overline{v}^{k^{3}+3k^{2}}=0\right\}}(\ell+F(0, y, s, \om)_{+}^{(d+1)}dy ds\\
&\leq 3^{-3k^{2}(d+2)}\sum_{j^{*}}\overline{\pi}^{ij^{*}}_{k^{3}}(\om)\\
&\leq 3^{-3k^{2}(d+2)}\left(\sum_{j^{*}}\overline{\pi}^{ij^{*}}_{k^{3}}(\om)+\sum_{j^{'}}\left(\overline{\pi}^{ij^{'}}_{k^{3}}(\om)-\frac{1}{2}\overline{E}_{k^{3}}\right)\right)\\
&\leq \overline{A}_{k^{3}+3k^{2}}(\om)-\frac{3^{-3k^{2}(d+2)}}{2}\left(\sum_{j^{'}}\overline{E}_{k^{3}}\right)\\
&\leq \overline{A}_{k^{3}+3k^{2}}(\om)-\frac{1}{2}\zeta_{d}\overline{E}_{k^{3}}.
\end{align*}

Finally, we compute that
\begin{align*}
\overline{E}_{k^{3}+3k^{2}}&=\EE[\overline{\pi}_{k^{3}+3k^{2}}]=\int_{\Om\setminus\textbf{W}_{k}} \overline{\pi}_{k^{3}+3k^{2}}(\om)d\PP+\int_{\textbf{W}_{k}} \overline{\pi}_{k^{3}+3k^{2}}(\om)d\PP\\
&\leq \int_{\Om\setminus\textbf{W}_{k}} \overline{A}_{k^{3}+3k^{2}}(\om)d\mu +\int_{\textbf{W}_{k}} \left( \overline{A}_{k^{3}+3k^{2}}(\om)-\frac{1}{2}\zeta_{d}\overline{E}_{k^{3}}\right)d\PP\\
&=\EE[\overline{A}_{k^{3}+3k^{2}}] -\frac{1}{2}\zeta_{d}\PP[\textbf{W}_{k}]\overline{E}_{k^{3}}\\
&\leq \overline{E}_{k^{3}}-\frac{1}{4}([1-9(1+\zeta_{d}^{-1})\eta]\zeta_{d})\overline{E}_{k^{3}}\\
&=\left(1-\frac{\zeta_{d}}{4}+\frac{9}{4}\eta\zeta_{d}+\frac{9}{4}\eta\right)\overline{E}_{k^{3}}
\end{align*}

Applying \pref{case3}, 
\begin{align*}
\overline{J}_{k^{3}+3k^{2}}&\leq (1+\eta)\left(1-\frac{\zeta_{d}}{4}+\frac{9}{4}\eta\zeta_{d}+\frac{9}{4}\eta\right)^{2}\overline{J}_{k^{3}}\\
&= (1+\eta)\left(1-\frac{1}{16}\left(\frac{2}{3}\right)^{d+1}+\frac{9}{16}\eta\left(\frac{2}{3}\right)^{d+1}+\frac{9}{4}\eta\right)^{2}\overline{J}_{k^{3}}.
\end{align*}
Therefore, if we choose $\eta$ so that
\begin{equation}\label{const4}
(1+\eta)\left(1-\frac{1}{16}\left(\frac{2}{3}\right)^{d+1}+\frac{9}{16}\eta\left(\frac{2}{3}\right)^{d+1}+\frac{9}{4}\eta\right)^{2}\leq 3^{-\tau},
\end{equation}
this produces the desired rate in Case 3. 

Thus, if we choose $k_{0}, \tau, \eta$ so that \pref{const1},  \pref{setcond}, \pref{const2}, \pref{const3}, and \pref{const4} all hold, then \pref{decaygoal} holds. 

\end{proof}

\section{A Rate of Decay on the Free Solutions}\label{decaycorrect}
In this section, we establish a rate of decay on $\norm{\ve^{2}w^{\ve}}_{L^{\infty}(Q_{1/\ve})}$ in measure. We note that we may reformulate Theorem \ref{thmrate1} by changing scales in terms of $\ve$. In other words, if we choose $3^{-k^{3}}=\ve$, then \pref{ratek} becomes
%\begin{align*}
%\ve&=3^{-k^{3}}\\
%\ln \ve &= -k^{3}\ln 3\\
%3&=\ve^{-1/k^{3}}\\
%3^{-\tau k}&=\ve^{-\tau/k^{2}}\\
%-k^{2} &=\left(\frac{\ln \ve}{\ln 3}\right)^{2/3}\\
%3^{-\tau k}&=\ve^{C(\ln\ve)^{-2/3}}
%\end{align*}
%which means
\begin{equation}\label{ratee}
\mathbb{J}_{\ve}\leq C(1+\norm{M}+|\ell|)^{4(d+1)}3^{c|\ln \ve|^{-2/3}}.
\end{equation}

Moreover, if we consider $Q_{1/\ve}$ instead of $C_{1/\ve}$, then we claim that this rate still holds, since the analysis stays the same. It was merely for convenience that we chose to present the previous section for $C_{1/\ve}$ instead of $Q_{1/\ve}$. By changing our perspective, and considering $w_{\ve}=\ve^{2}w^{\ve}(y/\ve, s/\ve^{2}, \om)$, we seek a rate of decay for $\norm{w_{\ve}}_{L^{\infty}(Q_{1})}$ in measure. We point out that $\chi_{\left\{\overline{\underline{v}}^{\ve}=0\right\}}(y,s, \om)=\chi_{\left\{\overline{\underline{v}}_{\ve}=0\right\}}(y/\ve, s/\ve^{2}, \om)$. This implies that $\overline{\pi}$ and $\underline{\pi}$ are invariant when we move between the interpretations of $w^{\ve}$ and $w_{\ve}$, which implies that \pref{ratee} holds for the problem with $w_{\ve}$. Moreover, throughout this section, we will constantly be relabeling constants $c, C, \hat{c}, \hat{C}$, etc. throughout the proofs, when these constants depend only on universal parameters such as $\la, \La, d$. 

In order to control the decay of $|w_{\ve}|$, it is enough to obtain a rate of decay for  $\overline{J}_{\ve}$ and $\underline{J}_{\ve}$. Although we have managed to obtain a rate on $\mathbb{J}_{\ve}$, this does not automatically yield a rate for $\overline{J}_{\ve}$ or $\underline{J}_{\ve}$. In particular, for $\ell$ very close to $-\overline{F}(M)$, it is possible that one of the quantities remains constant for some values of $\ve$, and then decays while the other stays constant. With this type of oscillatory behavior, it is not possible to show that $w_{\ve}\rightarrow\overline 0$ from both sides. In order to obtain the decay, we follow the strategy of \cite{cs} by studying the problem with $\ell=-\overline{F}(M)\pm\ga$ for some $\ga>0$. Using this choice for $\ell$ and \pref{dichot}, it will be immediate that either $\overline{J}_{\ve}$ or $\underline{J}_{\ve}$ must be strictly positive and bouned below. This means that the \pref{ratee} will be enough to yield a rate on the other quantity, which will be enough to obtain a rate on each side. 

We define $w_{\ve}^{\ga}$ to be the solution of 
\begin{equation}\label{freega}
\begin{cases}
\frac{\partial}{\partial s}w^{\ga}_{\ve}-F_{M}(D^{2}w^{\ga}_{\ve}, y/\ve, s/\ve^{2}, \om)=\overline{F}(M)+\ga\quad\text{in}\quad Q_{1}\\
w_{\ve}^{\ga}=\overline{P}\quad\text{on}\quad \partial_{p}Q_{1}.
\end{cases}
\end{equation}
and $\overline{v}^{\ga}_{\ve}$ and $\underline{v}^{\ga}_{\ve}$ to be the solutions of the obstacle problem from above and below respectively corresponding to \pref{freega}. Similarly, we let $\overline{\pi}_{\ve}^{\ga}(\om)$ and $\underline{\pi}_{\ve}^{\ga}(\om)$ denote the total masses of the obstacle problems corresponding to $\overline{v}_{\ve}^{\ga}$ and $\underline{v}_{\ve}^{\ga}$. We have the following ``perturbative estimate":
\begin{lemma}\label{perturb}
There exist uniform constants $C_{p}=C(\la, \La, d)$ and $\ve_{1}=\ve(\la, \La, d)$, such that for all $\ve<\ve_{1}$, 
\begin{equation}\label{perturbbnd}
\EE[(\overline{\pi}_{\ve}^{\ga})^{2}]\geq C_{p}\ga^{2(d+1)}.
\end{equation}
\end{lemma}

\begin{proof}
By Theorem \ref{stochhom}, $w_{\ve}^{\ga}\rightarrow w^{\ga}$ uniformly in $Q_{1}$, a.s. in $\om$, where $w^{\ga}$ solves
\begin{equation*}
\begin{cases}
w^{\ga}_{s}-\overline{F}(D^{2}w^{\ga})=\overline{F}(M)+\ga\quad\text{in}\quad Q_{1}\\
w^{\ga}=0\quad\text{on}\quad \partial_{p}Q_{1}.
\end{cases}
\end{equation*}

By the comparison principle and the uniform ellipticity of $\overline{F}$, there exists $\beta(\la, \La, d)$ so that
\begin{equation*}
w^{\ga}\leq -\beta \ga (s+1)(1-|y|^{2})
\end{equation*}
which implies, almost surely, 
\begin{equation}\label{compl}
\lim_{\ve\rightarrow 0}(\overline{v}_{\ve}^{\ga}(y,s, \om)-w^{\ga}(y,s, \om))\geq \beta \ga (s+1)(1-|y|^{2}).
\end{equation}
On the other hand, the parabolic ABP-estimate yields
\begin{equation}\label{abpest}
\sup_{Q_{1}}[ \overline{v}_{\ve}^{\ga}-w_{\ve}^{\ga}]^{d+1}(\cdot, \cdot, \om)\leq C_{abp}^{d+1} \overline{\pi}_{\ve}^{\ga}(\om).
\end{equation}
Combining \pref{compl} and \pref{abpest}, almost surely as $\ve\rightarrow0$, 
\begin{equation}
(\beta \ga)^{d+1}\leq[\overline{v}_{\ve}^{\ga}(0,0, \om)-w_{\ve}^{\ga}(0,0, \om)]^{d+1}\leq C_{abp}^{d+1} \overline{\pi}_{\ve}^{\ga}(\om).
\end{equation}
Squaring, integrating, and using the fact that the convergence is also in probability, there exists $\ve_{1}$ so that for all  $\ve<\ve_{1}$, \pref{perturbbnd} holds. 
\end{proof}
The same proof shows that if we consider $\ell=-\overline{F}(M)-\ga$, then we have for $\ve$ sufficiently small,
\begin{equation}\label{minusell}
\EE[(\underline{\pi}_{\ve}^{\ga})^{2}]\geq C_{p}\ga^{2(d+1)}.
\end{equation}

We now use this lemma to compute a rate of decay of $w_{\ve}$ in measure. 

\begin{proposition}\label{est1}
There exists positive constants $\tilde{C}, \tilde{c}, \hat{C}, \hat{c}$, and $\ve_{0}$ depending only on $\la, \La, d$ so that for every $M\in\mathbb{S}^{d}$, for all $\ve\leq \ve_{0}$, there exists $A^{M}_{\ve}\subset \Om$ such that 
\begin{equation*}
\PP[A^{M}_{\ve}]\leq \tilde{C}(1+\norm{M})^{2(d+1)}\ve^{\tilde{c}|\ln \ve|^{-2/3}}, 
\end{equation*}
and for all $\om\in \Om\setminus A^{M}_{\ve}$, 
\begin{equation*}
\norm{w_{\ve}(\cdot, \cdot, \om)}\leq \hat{C}\ve^{\hat{c}|\ln \ve|^{-2/3}}.
\end{equation*}
\end{proposition}

\begin{proof}
Fix $\ga>0$, to be chosen later. By the comparison principle, for each $\om$,
\begin{equation*}
\norm{w_{\ve}^{\ga}(\cdot, \cdot, \om)-w_{\ve}(\cdot, \cdot, \om)}_{L^{\infty}(Q_{1})}\leq C\ga.
\end{equation*}
Thus, 
\begin{equation*}
w_{\ve}(\cdot, \om)\leq w_{\ve}^{\ga}-\underline{v}_{\ve}^{\ga}+C\ga.
\end{equation*}
By combining Lemma \ref{perturb} and \pref{ratee}, 
\begin{equation*}
\EE[(\underline{\pi}_{\ve}^{\ga})^{2}]\leq C(1+\norm{M})^{4(d+1)}\ga^{-2(d+1)}\ve^{c|\ln\ve|^{-2/3}}.
\end{equation*}
Therefore, 
\begin{equation*}
\EE[\underline{\pi}_{\ve, \ga}]\leq C(1+\norm{M})^{2(d+1)}\ga^{-(d+1)}\ve^{c|\ln\ve|^{-2/3}}.
\end{equation*}
Finally, by applying \pref{abpest},
\begin{equation}\label{decaynorm}
\EE\left[\norm{w_{\ve}^{\ga}-\underline{v}_{\ve}^{\ga}}_{L^{\infty}(Q_{1})}^{d+1}\right]\leq C (1+\norm{M})^{2(d+1)}\ga^{-(d+1)}\ve^{c|\ln\ve|^{-2/3}}.
\end{equation}

For $\thet>0$, we define the exceptional sets $A^{\ga, M}_{\theta}\subset\Om$ by
\begin{equation}\label{defexcsets}
A^{\ga, M}_{\theta}=\left\{\om\in\Om: \norm{w_{\ve}^{\ga}(\cdot, \cdot, \om)-\underline{v}_{\ve, \ga}(\cdot, \cdot, \om)}_{L^{\infty}(Q_{1})}>\theta\right\},
\end{equation}

and by applying Chebyshev's inequality with \pref{decaynorm} , we have
\begin{equation*}
\PP\left[A^{\ga, M}_{\theta}\right]\leq C (1+\norm{M})^{2(d+1)}\ve^{c|\ln\ve|^{-2/3}}(\theta\ga)^{-(d+1)}.
\end{equation*}
Thus, for all $\om\in \Om\setminus A^{\ga, M}_{\theta}$, 
\begin{equation*}
\max_{\overline{Q}_{1}}(w_{\ve}(\cdot, \cdot, \om))\leq \theta+C\ga.
\end{equation*}
Finally, we just need to choose $\theta+C\ga=\hat{C}\ve^{\hat{c}(\ln \ve)^{-2/3}}$, for $\hat{c}$, $\hat{C}$ chosen appropriately, and this determines the choices of $\tilde{c}$ and $\tilde{C}$.  

A similar argument with $\ell=-\overline{F}(M)-\ga$ and applying \pref{minusell} yields the other side of the inequality. 
\end{proof}

We note that a priori, the set $A^{\ga, M}_{\theta}$ may depend upon the choice of $M$. We will show now that this is not the case, by obtaining an estimate which only depends on $\norm{M}$. 

%\begin{proposition}\label{alunifrate}
%There exists positive constants $\tilde{C}, \tilde{c}, \hat{C}, \hat{c}$ and $\ve_{0}$ depending only on $\la, \La, d$ such that for all $M_{0}\in \mathbb{S}^{d}$ and $\ve\leq \ve_{0}$, there exists $A^{M_{0}}_{\ve}\subset \Om$, which may depend on $M_{0}$, such that 
%\begin{equation*}
%\PP[A^{M_{0}}_{\ve}]\leq \tilde{C}(1+|\overline{M}|)^{2(d+1)}\ve^{\tilde{c}(\ln \ve)^{-2/3}}
%\end{equation*}
%and for all $\om \in \Om \setminus A_{\ve}$, 
%\end{proposition}

%\begin{proof}
%We note that by applying the comparison principle and using that $\overline{F}(\cdot)$ is Lipschitz, we have that
%\begin{equation*}
%\sup_{|M-M_{0}|\leq r}\norm{w_{\ve, M}(\cdot, \cdot, \om)}_{L^{\infty}(C_{1})}\leq Cr+\norm{w_{\ve, M_{0}}(\cdot, \cdot, \om)}_{L^{\infty}(C_{1})}
%\end{equation*}
%Moreover, by Proposition \ref{est1}, we have that there exists positive constants $\tilde{C}, \tilde{c}$ and $\hat{C}, \hat{c}$ and a set $A_{\ve}\subset \Om$, such that
%\begin{align*}
%&\PP[A^{M_{0}}_{\ve}]\leq \tilde{C}(1+|M_{0}|)^{2(d+1)}\ve^{\tilde{c}|\ln \ve|^{-2/3}}\quad\text{and}\\
%&\norm{w_{\ve}(\cdot, \cdot, \om)}_{L^{\infty}(C_{1})}\leq \hat{C}\ve^{\hat{c}|\ln \ve|^{-2/3}}\quad\text{for}\quad \om\in \Om\setminus A_{\ve}.
%\end{align*}
%Letting $r=\ve^{\hat{c}\ln (\ve)^{-2/3}}$, the claim follows. 
%\end{proof}
%Finally, we use a covering argument to construct an exceptional set $A^{\ve}$ which only depends upon the size of the quadratic $\overline{M}$. 

\begin{proposition}\label{unifrate}
Let $R>0$. There exist uniform positive constants $\tilde{C}, \hat{C}, \tilde{c}, \hat{c}$ and $\ve_{0}$ so that for all $M\in \mathbb{S}^{d}$, such that $\norm{M}\leq R$, for all $\ve\leq \ve_{0}$, there exists $A_{\ve}\subset \Om$ such that 
\begin{equation}
\PP[A_{\ve}]\leq \tilde{C}R^{d^{2}}(1+R)^{2(d+1)}\ve^{\tilde{c}|\ln \ve|^{-2/3}}, 
\end{equation}
and for all $\om\in \Om\setminus A_{\ve}$, 
\begin{equation}\label{rateat1}
\sup_{\norm{M}\leq R} \norm{w_{\ve, M}(\cdot, \cdot, \om)}_{L^{\infty}(Q_{1})}\leq \hat{C}\ve^{\hat{c}|\ln \ve|^{-2/3}}, 
\end{equation} 
where $w_{\ve, M}$ solves \pref{fbar}. 
\end{proposition}

\begin{proof}
We cover $\left\{M\in \mathbb{S}^{d}: \norm{M}\leq R\right\}\subset \cup_{i=1}^{K} \mathcal{B}_{r}(M_{i})$ where $\mathcal{B}_{r}(M_{i})=\left\{N\in \mathbb{S}^{d}: \norm{N-M_{i}}\leq r\right\}$, $\norm{M_{i}}\leq R$, and $K\leq 2(R/r)^{d^{2}}$.

By applying the comparison principle and using that $\overline{F}(\cdot)$ is uniformly elliptic, there exists a uniform constant $C$, such that for all $i$,  
\begin{equation}\label{nearby}
\sup_{\norm{M-M_{i}}\leq r}\norm{w_{\ve, M}(\cdot, \cdot, \om)}_{L^{\infty}(Q_{1})}\leq Cr+\norm{w_{\ve, M_{i}}(\cdot, \cdot, \om)}_{L^{\infty}(Q_{1})}. 
\end{equation}
 
By Proposition \ref{est1}, and taking the maxmimum over all $i$'s, there exist positive constants $\tilde{C}, \tilde{c}$, $\hat{C}, \hat{c}$ and a set $A_{\ve}\subset \Om$, such that
\begin{align*}
&\PP[A^{M_{i}}_{\ve}]\leq \tilde{C}(1+R)^{2(d+1)}\ve^{\tilde{c}|\ln \ve|^{-2/3}}\quad\text{and}\\
&\norm{w_{\ve, M_{i}}(\cdot, \cdot, \om)}_{L^{\infty}(Q_{1})}\leq \hat{C}\ve^{\hat{c}|\ln \ve|^{-2/3}}\quad\text{for}\quad \om\in \Om\setminus A_{\ve}.
\end{align*}
%Letting $r=\ve^{\hat{c}\ln (\ve)^{-2/3}}$, we have that for all $\om\in \Om\setminus A^{M_{0}}_{\ve}$, 
%\begin{equation*}
%\sup_{\norm{M-M_{0}}\leq \ve^{\hat{c}|\ln \ve|^{-2/3}}}\norm {w_{\ve, M}(\cdot, \om)}_{L^{\infty}(Q_{1})}\leq \hat{C}\ve^{\hat{c}|\ln \ve|^{-2/3}}.
%\end{equation*}
Let $A_{\ve}=\bigcup A^{M_{i}}_{\ve}$. Then 
\begin{equation*}
\PP[A_{\ve}]\leq \sum_{i=1}^{K}\PP[A^{M_{i}}_{\ve}]\leq \tilde{C}\left(\frac{R}{r}\right)^{d^{2}}(1+R)^{2(d+1)}\ve^{\tilde{c}|\ln\ve|^{-2/3}},
\end{equation*}
and for all $\om\in \Om\setminus A_{\ve}$,
\begin{equation}\label{anchors}
\sup_{\norm{M_{i}}\leq R} \norm{w_{\ve, M_{i}}(\cdot, \cdot, \om)}_{L^{\infty}(Q_{1})}\leq \hat{C}\ve^{\hat{c}|\ln\ve|^{-2/3}}\quad\forall i.
\end{equation}

Choose $r=\ve^{c|\ln\ve|^{-2/3}}$, where $cd^{2}\leq \tilde{c}$. By combining \pref{nearby} and \pref{anchors}, we have that for all $\om\in \Om\setminus A_{\ve}$, 
\begin{equation*}
\sup_{\norm{M}\leq R} \norm{w_{\ve, M}(\cdot, \cdot, \om)}_{L^{\infty}(Q_{1})}\leq \hat{C}\ve^{\hat{c}|\ln\ve|^{-2/3}}, 
\end{equation*}
and 
\begin{equation*}
\PP[A_{\ve}]\leq \tilde{C}(R^{d^{2}}(1+R)^{2(d+1)})\ve^{(\tilde{c}/2)(\ln\ve)^{-2/3}}.
\end{equation*}

\end{proof}

%\begin{remark}\label{alleasy}
%We note that we have been able to show that there exists a rate of convergence for $w_{\ve}\rightarrow \overline{M}$. Suppose we had boundary data which was quadratic in space, and linear in time. In other words, consider solving
%\begin{equation}\label{easydata}
%\begin{cases}
%g_{\ve, s}-F(D^{2}g_{\ve}, y/\ve, s/\ve^{2}, \om)=-\overline{F}(M)+\al\\
%g_{\ve}(y,s)=\overline{M}+\al s
%\end{cases}
%\end{equation}
%Then, by uniqueness arguments, $g_{\ve}=w_{\ve}+\al s$ is the unique solution to \pref{easydata}. More importantly, 
%\begin{equation*}
%\norm{g_{\ve}-(\overline{M}+\al s)}_{L^{\infty}(C_{1})}=\norm{w_{\ve}-\overline{M}}_{L^{\infty}(C_{1})}
%\end{equation*}
%and therefore there exists $A_{\ve}\subset \Om$ so that for any $\al\in \RR$, 
%\begin{equation}
%\sup_{|M|\leq R} \norm{g_{\ve}-(\overline{M}+\al s)}_{L^{\infty}(C_{1})}\leq \hat{C}\ve^{\hat{c}(\ln \ve)^{-2/3}}\quad\text{in}\quad \Om\setminus A_{\ve}
%\end{equation} 
%and
%\begin{equation}
%\PP[A_{\ve}]\leq \overline{C}R^{d^{2}}(1+R)^{2(d+1)}\ve^{\tilde{c}(\ln \ve)^{-2/3}}.
%\end{equation}
%This shows that we now have a rate for homogenization with quadratic in space, linear in time data. 
%\end{remark}

\begin{remark}\label{ratescale}
For future reference, we work out the scaling of \pref{rateat1}. Suppose that $w_{\ve}$ solves
%We comment on what happens when we rescale $w_{\ve}$. Suppose we are interested in understanding the rate of  convergence to 
\begin{equation*}
\begin{cases}
w_{\ve, s}-F_{M}(D^{2}w_{\ve}, y/\ve, s/\ve^{2}, \om)=-\overline{F}(M)\quad\text{in}\quad Q_{r}\\
w_{\ve}=0\quad\text{on}\quad \partial_{p}Q_{r}
\end{cases}
\end{equation*}
% Then if we define $g^{\ve, r}(y,s)=\frac{1}{r^{2}}g^{\ve}(ry, r^{2}s, \om)$, we have that
% \begin{equation*}
 %\begin{cases}
 %g^{\ve, r}_{s}-F(D^{2}g^{\ve, r}, ry/\ve, r^{2}s/\ve^{2}, \om)=\overline{F}(P)\quad\text{in}\quad C_{1}\\
 %g^{\ve, r}(y,s)=\overline{P}\quad\text{on}\quad \partial_{p}C_{1}
% \end{cases}
%  \end{equation*}
% Therefore, by applying Proposition \ref{unifrate},we have 
%\begin{equation*}
%\sup_{|P|\leq R} \norm{g^{\ve, r}(\cdot, \cdot, \om)-\overline{P}}_{L^{\infty}(C_{1})}\leq \hat{C}(\ve r^{-1})^{\hat{c}(\ln (\ve r^{-1}))^{-2/3}}\quad\text{in}\quad \Om\setminus A_{\ve}
%\end{equation*} 
%and
%\begin{equation*}
%\PP[A_{\ve}]\leq \overline{C}R^{d^{2}}(1+R)^{2(d+1)}(\ve r^{-1})^{\tilde{c}(\ln |\ve r^{-1}|)^{-2/3}}.
%\end{equation*} 
Rescaling Proposition \ref{unifrate}, there exists $A_{\ve}\subset\Om$ such that 
\begin{equation*}
\PP[A_{\ve}]\leq \overline{C}R^{d^{2}}(1+R)^{2(d+1)}(\ve r^{-1})^{\tilde{c}|\ln (\ve r^{-1})|^{-2/3}},
\end{equation*} 
and for all $\om\in \Om\setminus A_{\ve}$, 
\begin{equation*}
\sup_{\norm{M}\leq R} \norm{w_{\ve, M}(\cdot, \cdot, \om)}_{L^{\infty}(Q_{r})}\leq \hat{C}r^{2}(\ve r^{-1})^{\hat{c}|\ln (\ve r^{-1})|^{-2/3}}.
\end{equation*} 
\end{remark}

\section{The Proof of Theorem \ref{thmfinalrate}}\label{errorcomp}
We show that the rate of decay on the approximate correctors $w_{\ve}$ yields a rate of decay on $u^{\ve}-u$. This follows by performing a quantification of the perturbed test function method of Evans \cite{evanshom}. In the elliptic setting, this quantification was first presented in the language of $\delta$-solutions, introduced in \cite{cs}. $\delta$-solutions were recently generalized to the parabolic setting by Turanova \cite{olga1}. We heavily use the regularity estimates developed in \cite{olga1}, but for the problem at hand we choose to present the results directly. We will constantly relabel $C$ in this section, whenever $C$ depends only on universal constants such as $\la, \La, d, T$. 

Our goal is to control $\PP[\om: \norm{u^{\ve}(\cdot, \cdot, \om)-u(\cdot, \cdot)}_{L^{\infty}_(D_{T})} \geq C\delta^{\al}]$, where $C\delta^{\al}$ will be chosen in terms of $\ve$. For convenience, we denote this set $\PP\left[|u^{\ve}-u|\geq C\delta^{\al}\right]$.  

We first outline our method of approach. We compare $u$ to strict supersolutions and subsolutions of \pref{limit}, because these approximations allow us to absorb small errors which come from various perturbations in the argument. We define $u^{+\delta}$
\begin{equation}\label{delsub}
\begin{cases}
u^{+\delta}_{t}-\overline{F}(D^{2}u^{+\delta})=\delta^{\al}\quad\text{in}\quad D_{T},\\
u^{+\delta}=g\quad\text{on}\quad \partial_{p}D_{T},
\end{cases}
\end{equation}
and 
\begin{equation}\label{delsup}
\begin{cases}
u^{-\delta}_{t}-\overline{F}(D^{2}u^{-\delta})=-\delta^{\al} \quad\text{in}\quad D_{T},\\
u^{-\delta}=g\quad\text{on}\quad \partial_{p}D_{T}.
\end{cases}
\end{equation}

The comparison principle yields
\begin{equation*}
\sup_{D_{T}}(u^{+\delta}-u)\leq \delta^{\al} T\quad\text{and}\quad\sup_{D_{T}}(u-u^{-\delta})\leq \delta^{\al} T.
\end{equation*}

It follows that
\begin{align}
\PP\left[|u^{\ve}-u|\geq C\delta^{\al} \right]&\leq \PP\left[ (u^{\ve}-u)\geq C\delta^{\al} \right]+\PP\left[(u^{\ve}-u)\leq -C\delta^{\al} \right]\notag\\
%&=\PP[(u^{\ve}-u^{+\delta})\geq (u-u^{+\delta})+C\delta^{\al}]+\PP[(u^{\ve}-u^{-\delta})\leq (u-u^{-\delta})-C\delta^{\al}]\notag\\
&\leq \PP[(u^{\ve}-u^{+\delta})\geq C\delta^{\al}]+\PP[(u^{\ve}-u^{-\delta})\leq -C\delta^{\al}]\notag\\
&= \PP[(u^{+\delta}-u^{\ve})\leq -C\delta^{\al}]+\PP[(u^{-\delta}-u^{\ve})\geq C\delta^{\al}]\label{comp}
\end{align}
Therefore, the main focus of this section will be to control $\PP[(u^{+\delta}-u^{\ve})\leq -C\delta^{\al}]$ and $\PP[(u^{-\delta}-u^{\ve})\geq C\delta^{\al}]$.

We continue the argument for controlling $\PP[(u^{+\delta}-u^{\ve})\leq -C\delta^{\al}]$, since the argument for $\PP[(u^{-\delta}-u^{\ve})\geq C\delta^{\al}]$ follows similarly. After regularizing $u^{+\delta}$ and $u^{\ve}$ appropriately, it was pointed out in \cite{olga1} that if $\om\in \left\{(u^{+\delta}-u^{\ve})\leq -C\delta^{\al}\right\}$, then there exists a paraboloid $P$ which touches touches $u_{\ve}$ from above, and has opening $|D^{2}P|\leq C\delta^{-\sigma}$, where $C, \sigma$ are universal constants of the problem. Perturbing $P$ to be $\tilde{P}$, we show that $\tilde{P}+\ve^{2}w^{\ve}$ is a supersolution to \pref{delsub} with a given rate of decay for $\om\in \Om\setminus A_{\ve}$, where $A_{\ve}$ is the exceptional set originating from Proposition \ref{unifrate}. As in the proof of Theorem \ref{stochhom}, we obtain a contradiction, which implies that $\left\{(u^{+\delta}-u^{\ve})\leq -C\delta^{\al}\right\}\subset A_{\ve}$, whose measure is controlled by Proposition \ref{unifrate}. We note that the exceptional sets $A_{\ve}$ in Proposition \ref{unifrate} only hold for local cylinders or parabolic cubes. Therefore, to estimate the total probability corresponding to the domain $D_{T}$, we apply a covering argument to conclude. 

Without loss of generality, we assume that $u^{+\delta}$ are $u^{\ve}$ are semiconcave and semiconvex respectively. This may be justified by regularizing $u^{+\delta}$ and $u^{\ve}$ using inf and sup convolutions in space (see Appendix E and \cite{olga1} for details). Furthermore, we only need to analyze the behavior of $u^{+\delta}-u^{\ve}$ in the interior of the domain, since $u^{+\delta}$ and $u^{\ve}$ share the same boundary condition, and they are both Lipschitz continuous. 
We introduce some new notation which will be utilized in this section. In order to perform the covering, we consider a grid enumerated by $i$ of parabolic cubes $K^{i}_{r}=\left[x_{i}-\frac{r\sqrt{d}}{2}, x_{i}+\frac{r\sqrt{d}}{2}\right]^{d}\times (t_{i}+\frac{r^{2}}{162d}, t_{i}+\frac{r^{2}}{81d}]$, with $r$ to be chosen so that $K^{i}_{r}\subset Q^{+}_{r}(x_{i}, t_{i})\subset D_{T}$. We are allowed to do so, because we are only concerned with the interior of $D_{T}$. Moreover, we define $\mathcal{P}_{\sigma}$ to be the set of all paraboloids of opening $C\delta^{-\sigma}$. In other words, 
\begin{equation*}
\mathcal{P}_{\sigma}=\left\{P(x,t): P(x,t)=P_{t}\cdot t+\frac{1}{2}\langle D^{2}Px, x\rangle+C~\text{and}~|P_{t}|, |D^{2}P|\leq C\delta^{-\sigma}\right\}. 
\end{equation*}
Using the same arguments as in \cite{olga1}, we have
\begin{claim}\label{claimcube}
There exists $Q^{+}_{r}(x_{i}, t_{i})$ so that there exists $(x_{0}, t_{0})\in K^{i}_{r}$  and a paraboloid $P\in \mathcal{P}_{\sigma}$, with $P\geq u^{\ve}$ in $\overline{Q^{+}_{r}}(\overline{x}, \overline{t})\cap \left\{t\leq t_{0}\right\}$, $P(x_{0}, t_{0})=u^{\ve}(x_{0}, t_{0})$, $P_{t}(x_{0}, t_{0})-\overline{F}(D^{2}P(x_{0}, t_{0}))> \delta^{\al}$.
\end{claim}
We prove the claim in the appendix. Therefore, 
\begin{align*}
&\PP[(u^{+\delta}-u^{\ve})\leq -C\delta^{\al}]\\
&\leq \PP[~\exists i: \text{Claim \ref{claimcube} holds}].
\end{align*}
Moreover, in light of the grid,
\begin{equation*}
\PP[~\exists i: \text{Claim \ref{claimcube} holds}]\leq \sum_{i} \PP[\text{Claim \ref{claimcube} holds in}~Q^{+}_{r}(x_{i}, t_{i})].
 \end{equation*}
 Therefore, if we are able to control $\PP[\text{Claim \ref{claimcube} holds}]$ in each $Q^{+}_{r}(x_{i}, t_{i})$, then we are done. 
 
 The next lemma shows that 
 \begin{equation*}
 \left\{\om: \text{Claim \ref{claimcube} holds}\right\}\subset A_{\ve}^{i},
 \end{equation*}
where $A_{\ve}^{i}$ is the exceptional set corresponding to $Q^{+}_{r}(x_{i}, t_{i})$. 
 
 \begin{lemma}\label{deltaperturb}
 Let $Q^{+}_{r}(x_{i},t_{i})$, $K^{i}_{r}$ as above. Let $\om\in \Om$ so that there exists $(x_{0}, t_{0})\subset \overline{K}_{r}(x_{i},t_{i})$ and a paraboloid $P\in\mathcal{P}_{\sigma}$ such that $P(x_{0}, t_{0})=u^{\ve}(x_{0}, t_{0}, \om)$, $P\geq u^{\ve}(\cdot, \cdot, \om)$ in $\overline{Q}_{r}(x_{i},t_{i})\cap \left\{t\leq t_{0}\right\}$, and $P_{t}(x_{0}, t_{0})-\overline{F}(D^{2}P(x_{0}, t_{0}))> \delta^{\al}$. Then there exists a choice of $r(\delta, \ve)$ such that $\om\in A_{\ve}^{Q^{+}_{r}(x_{i}, t_{i})}$, the exceptional set associated with $Q^{+}_{r}(x_{i}, t_{i})$. 
 \end{lemma}

\begin{proof}
We first claim that without loss of generality, we may assume that $P(x,t)=u^{\ve}(x_{0}, t_{0})+P_{t}(t-t_{0})+\frac{1}{2}\langle D^{2}P(x-x_{0}), x-x_{0}\rangle$, where $P_{t}$ and $D^{2}P$ are constant. 
%This follows from the fact that the problem is independent of $x$, and therefore, we may subtract a linear function from it. We note that $|D^{2}P(x_0, t_{0})|\leq C\delta^{-\sigma}$.
 %We will eventually choose $r, \hat{r}$ in terms of $\delta$. 
We will perturb $P$ to a function $P^{*}$ where $P^{*}+w_{\ve, r}$ is a supersolution to \pref{hom}. 

%First, we note that for any $(x_{0}, t_{0})\in K_{r}(x_{i},t_{i})$, there exists $\hat{r}$ so that $Q_{\hat{r}}(x_{0}, t_{0})\subset Q^{+}_{r}(x_{i},t_{i})$. (in particular, note that $\hat{r}\leq \min\left(\frac{r\sqrt{d}}{2}, \frac{r}{\sqrt{162d}}\right)$. 

First, let $\tilde{P}(x,t)=P(x,t)-\eta\delta^{\al}(t-t_{i})(r^{2}-|x-x_{i}|^{2})$, with $\eta$ to be chosen. Note that by construction, $\tilde{P}=P\geq u^{\ve}$ on $\partial_{p}Q^{+}_{r}(x_{i}, t_{i})$. Also, we have
\begin{equation*}
\tilde{P}_{t}=P_{t}-\eta\delta^{\al}(r^{2}-|x-x_{i}|^{2})\quad\text{and}\quad D^{2}\tilde{P}=D^{2}P+2\eta\delta^{\al}(t-t_{i})Id
\end{equation*}
which implies 
\begin{equation*}
\tilde{P}_{t}(x_{0}, t_{0})=P_{t}-\eta\delta^{\al}(r^{2}-|x_{0}-x_{i}|^{2}) \quad\text{and}\quad D^{2}\tilde{P}(x_{0}, t_{0})=D^{2}P+2\eta\delta^{\al}(t_{0}-t_{i})Id.
\end{equation*}
%\begin{equation*}
%D^{2}\tilde{P}=D^{2}P+2\eta\delta^{\al}(t-t_{i})Id
%\end{equation*}
%so

%We note that by ellipticity, we may choose $\eta>0$ depending only on $d, \La, r$ so that
%\begin{align*}
%\tilde{P}_{t}(x_{0}, t_{0})-\overline{F}(D^{2}\tilde{P}(x_{0}, t_{0}))&=P_{t}-\eta\delta^{\al}|x_{i}+r-x_{0}|^{2}-\overline{F}(D^{2}P+2d\eta\delta^{\al}(t_{0}-t_{i})Id)\\
%&\geq P_{t}-\eta\delta^{\al}|x_{i}+r-x_{0}|^{2}-\overline{F}(D^{2}P)-\La 2d \eta \delta^{\al}|t_{0}-t_{i}|)\\
%&>\delta^{\al}-\eta\delta^{\al}|x_{i}+r-x_{0}|^{2}+\La 2d \eta \delta^{\al}|t_{0}-t_{i}|)\\
%&\geq \delta^{\al}(1-\eta(|x_{i}+r-x_{0}|^{2}+\La 2d |t_{0}-t_{i}|)\\
%&\geq \frac{3\delta^{\al}}{4}
%\end{align*}
%(This is true if 
%\begin{equation}
%\eta\left(r^{2}+\La 2d \frac{r^{2}}{81d}\right)<\frac{1}{4}
%\end{equation}

%Notice that 
%\begin{equation*}
%\tilde{P}_{t}(x,t)=\tilde{P}_{t}(x_{0}, t_{0})+\eta\delta^{\al}\left[|x-x_{i}|^{2}-|x_{0}-x_{i}|^{2}\right]\quad\text{and}\quad D^{2}\tilde{P}=D^{2}\tilde{P}(x_{0}, t_{0})+2\eta\delta^{\al}(t-t_{0})Id.
%\end{equation*}

This yields
\begin{equation*}
\tilde{P}(x_{0}, t_{0})-u^{\ve}(x_{0}, t_{0})=-\eta\delta^{\al}(t_{0}-t_{i})(r^{2}-|x_{0}-x_{i}|^{2})<0.
\end{equation*}
%P(x_{0}, t_{0})-u^{\ve}(x_{0}, t_{0})-\eta\delta^{\al}(t_{0}-t_{i})(r^{2}-|x_{0}-x_{i}|^{2})=-
%We point out that by construction, $P=\tilde{P}$ on $\partial{p}Q^{+}_{r}(x_{i}, t_{i})$. 
%We also point out that $\tilde{P}=P$ on $\partial_{p}Q_{\hat{r}}(x_{0}, t_{0})\subset \overline{Q^{+}_{r}}(x,t)$. Therefore, since $P$ is touching from above, and is an upward opening paraboloid, it follows that $\tilde{P}\geq P$ on $\partial_{p}\overline{Q^{+}_{r}}(x,t)$. 

%We note that if we considered the function $\rho(x,t)=(\ga+a\cdot x_{0}-3\eta\delta^{\al})t+\frac{1}{2}\langle (M+6\eta\delta^{\al})x, x\rangle+P_{\delta}(x_{0}, t_{0})$, then by Taylor's theorem, 
%\begin{equation*}
%|\rho(x,t)-P_{\delta}(x,t)|\leq |a\cdot (x-x_{0})(t-t_{0})|
%\end{equation*}
%but since we are on $Q_{\delta}$ and $ |a|\leq C \delta^{-\sigma}$, we have that the entire right hand side is bounded above by $C\delta^{3-\sigma}$. Moreover, we note that $\rho(x_{0}, t_{0})=P_{\delta}(x_{0}, t_{0})$. 
We consider the solution  $w_{\ve}$ which solves
\begin{equation}
\begin{cases}
w_{\ve, t}-F(D^{2}\tilde{P}(x_{0}, t_{0})+D^{2}w_{\ve}, x/\ve, t/\ve^{2}, \om)=-\overline{F}(D^{2}\tilde{P}(x_{0},t_{0}))\quad\text{in}\quad Q^{+}_{r}(x_{i},t_{i})\\
w_{\ve}=0\quad\text{on}\quad\partial_{p}Q^{+}_{r}(x_{i},t_{i})
\end{cases}
\end{equation}

%We let $P^{*}(x,t)=u^{\ve}(x_{0}. t_{0})+P_{t}(x_{0}, t_{0})(t-t_{0})+\frac{1}{2}\langle D^{2}P(x_{0}, t_{0})(x-x_{0}), x-x_{0}\rangle+2d\eta \delta^{\al}r^{4}$. By Taylor's theorem, we have that $P^{*} \geq \tilde{P}$. Then, $P^{*}+w^{\delta}_{\ve}$ is a supersolution to \pref{hom}. 

By Proposition \ref{unifrate}, there exists positive constants $\tilde{C}, \tilde{c}, \hat{C}, \hat{c}$ and $\ve_{0}$ such that for all $\ve\in (0, r\ve_{0})$, there exists a set of bad configurations $A^{r}_{\ve}=A_{\ve}^{\overline{Q}_{r}(x_{i}, t_{i})}\subset \Om$ such that for all $\om\in  \Om\setminus A^{r}_{\ve}$
\begin{equation}
\norm{w_{\ve}(\cdot, \cdot, \om)}\leq Cr^{2}(\ve r^{-1})^{\hat{c}|\ln (\ve r^{-1})|^{-2/3}}.
\end{equation} 
(Note that $P$ may vary between choice of $(x_{0}, t_{0})$ but $|D^{2}P|\leq C\delta^{-\sigma}$ is universal. Moreover, the size of the perturbation is controlled for all choices of $(x_{0}, t_{0})\in K_{r}(x_{i}, t_{i})$. We say that everything is controlled by a factor of $C\delta^{-\sigma}$ universally.)

On $\partial_{p}(Q^{+}_{r}(x_{i}, t_{i})\cap \left\{t\leq t_{0}\right\})$,  
\begin{equation*}
\tilde{P}+w_{\ve}=P\geq u_{\ve}.
\end{equation*}

Moreover, upon studying the solution properties, we have 
\begin{align*}
&\tilde{P}_{t}+w_{\ve, t}-F(D^{2}\tilde{P}+D^{2}w_{\ve}, x/\ve, t/\ve^{2}, \om)\\
 &=\tilde{P}_{t}(x_{0}, t_{0})+\eta\delta^{\al}\left[|x-x_{i}|^{2}-|x_{0}-x_{i}|^{2}\right]+w_{\ve,t}-F(D^{2}\tilde{P}(x_{0}, t_{0})-2\eta\delta^{\al}(t_{0}-t)Id+D^{2}w_{\ve}, x/\ve, t/\ve^{2}, \om)\\
&\geq \tilde{P}_{t}(x_{0}, t_{0})+w_{\ve,t}-F(D^{2}\tilde{P}(x_{0}, t_{0})+D^{2}w_{\ve}, x/\ve, t/\ve^{2}, \om)+\eta\delta^{\al}\left[|x-x_{i}|^{2}-|x_{0}-x_{i}|^{2}\right]-2\La d\eta\delta^{\al}|t_{0}-t|\\
&\geq P_{t}(x_{0}, t_{0})-\eta\delta^{\al}\left[r^{2}-|x_{0}-x_{i}|^{2}-(|x-x_{i}|^{2}-|x_{0}-x_{i}|^{2})\right]-F(D^{2}\tilde{P}(x_{0}, t_{0})+D^{2}w_{\ve}, x/\ve, t/\ve^{2}, \om)\\
&-2\La d\eta\delta^{\al}|t_{0}-t|\\
&\geq P_{t}(x_{0}, t_{0})-\eta\delta^{\al}\left[r^{2}-|x-x_{i}|^{2}\right]-\overline{F}(D^{2}\tilde{P}(x_{0}, t_{0}))-2\La d\eta\delta^{\al}|t_{0}-t|\\
&= P_{t}(x_{0}, t_{0})-\overline{F}(D^{2}P+2\eta\delta^{\al}(t_{0}-t_{i})Id)-\eta\delta^{\al}\left[r^{2}-|x-x_{i}|^{2}\right]-2\La d\eta\delta^{\al}|t_{0}-t|\\
&\geq \delta^{\al}-\eta\delta^{\al}\left[r^{2}-|x-x_{i}|^{2}\right]-2\La d\eta\delta^{\al}(|t_{0}-t_{i}|+|t_{0}-t|)\\
&=\delta^{\al}\left[1-\eta\left(r^{2}-|x-x_{i}|^{2}+2\La d(|t_{0}-t_{i}|+|t_{0}-t|)\right)\right]
\end{align*}
which is positive if 
\begin{equation}\label{etacond}
\eta\left[r^{2}+\frac{4\La d r^{2}}{(81d)^{2}}\right]<1.
\end{equation}

By the comparison principle, 
\begin{equation}
u_{\ve}\leq \tilde{P}+w_{\ve}\quad\text{in}\quad Q^{+}_{r}(x_{i}, t_{i})\cap\left\{t\leq t_{0}\right\}.
\end{equation}
In particular, this means
\begin{align*}
0&< w_{\ve}^{r}(x_{0}, t_{0})+\tilde{P}(x_{0}, t_{0})-u_{\ve}(x_{0}, t_{0})\leq  w_{\ve}^{\delta}(x_{0}, t_{0})-\eta\delta^{\al}(t_{0}-t_{i})(r^{2}-|x_{0}-x_{i}|^{2})\\
%&\leq \hat{C}r^{2}\left(\ve r^{-1}\right)^{-\hat{c}(\ln|\ve r^{-1}|)^{-2/3}}+2d\eta\delta^{\al}r^{4}-\eta\delta^{\al}\hat{r}^{4}
\end{align*}
which implies that
\begin{equation*}
\eta\delta^{\al}(t_{0}-t_{i})(r^{2}-|x_{0}-x_{i}|^{2})\leq  \hat{C} r^{2}\left(\ve r^{-1}\right)^{-\hat{c}|\ln(\ve r^{-1})|^{-2/3}}
\end{equation*}
For appropriate choices of $\delta, r, \al$, with $\delta^{\al}<C(\ve r^{-1})^{-\hat{c}|\ln(\ve r^{-1})|^{-2/3}}$, we obtain a contradiction. Therefore, $\om\in A^{r}_{\ve}$ for the appropriate choices of constants.  
\end{proof}
This shows that 
\begin{align*}
&\PP[\text{Claim \ref{claimcube} holds in}~Q^{+}_{r}(x_{i}, t_{i})]\\
&\leq \PP[A^{i}_{\ve}]\\
&\leq \tilde{C}\delta^{-\sigma d^{2}}(1+\delta^{-\sigma})^{2(d+1)}(\ve r^{-1})^{\tilde{c}(\ln |\ve r^{-1}|)^{-2/3}}.
\end{align*}

Therefore, 
\begin{align}
&\sum_{i} \PP[\text{Claim \ref{claimcube} holds in}~Q^{+}_{r}(x_{i}, t_{i})]\notag\\
&\leq \frac{|D_{T}|}{|K_{r}^{i}|}\tilde{C}\delta^{-\sigma d^{2}}(1+\delta^{-\sigma})^{2(d+1)}(\ve r^{-1})^{\tilde{c}|\ln (\ve r^{-1})|^{-2/3}}\notag\\
&\leq \frac{\tilde{C}(D_{T})}{r^{d+2}}\delta^{-\sigma d^{2}}(1+\delta^{-\sigma})^{2(d+1)}(\ve r^{-1})^{\tilde{c}|\ln (\ve r^{-1})|^{-2/3}}\label{sumcubes}
\end{align}
and for suitable choices of $r, \delta, \al$, this is controlled from above by $\tilde{C}\ve^{\tilde{c}|\ln \ve|^{-2/3}}$. 

We are now ready to prove Theorem \ref{thmfinalrate}. 

\begin{proof}[Proof of Theorem \ref{thmfinalrate}]
Combining \pref{comp}, Claim \ref{claimcube}, Lemma \ref{deltaperturb} and \pref{sumcubes}, we obtain the desired result for universal choices of $\tilde{c}, \tilde{C}, \hat{c},$ and $\hat{C}$. 
\end{proof}

\appendix 

\section{The Subadditive Ergodic Theorem}
We prove a generalized version of Akcoglu and Krengel's subadditive ergodic theorem \cite{akerg}. We present our work in a more general setting than what is needed in this paper. We show that the subadditive ergodic theorem holds for any sequence of nested cubes with sides of arbitrary lengths in $\RR^{d}$, where now $d$ is an arbitrary dimension. Let $\mathcal{I}$ denote the collection of subcubes in $\RR^{d}$, and for all $I\in \mathcal{I}$, we denote $|I|$ to be the Lebesgue measure of the cube. Recall the underlying probability space $(\Om, \mathcal{F}, \PP)$, and for all $s\in \RR^{d}$, transformation $\tau_{s}: \Om\rightarrow \Om$ is measurable. We consider nonnegative processes $R(I, \om): \mathcal{I} \times \Om\rightarrow\RR$ which satisfy the following:
\begin{list}{ (\thesscan)}
{
\usecounter{sscan}
\setlength{\topsep}{1.5ex plus 0.2ex minus 0.2ex}
\setlength{\labelwidth}{1.2cm}
\setlength{\leftmargin}{1.5cm}
\setlength{\labelsep}{0.3cm}
\setlength{\rightmargin}{0.5cm}
\setlength{\parsep}{0.5ex plus 0.2ex minus 0.1ex}
\setlength{\itemsep}{0ex plus 0.2ex}
}

\item \label{sethypot1}  $R$ is stationary, i.e. for any $s\in \RR^{d}$, 
\begin{equation*}
R(s+I, \om)=R(I, \tau_{s}\om).
\end{equation*}
\item\label{sethypot2} $R$ is subadditive, i.e. for all $I\in \mathcal{I}$, if $I=I_{1}\cup I_{2}$, with $I_{1}\cap I_{2}=\emptyset$, then 
\begin{equation*}
R(I,\om)\leq R(I_{1},\om)+R(I_{2}, \om).
\end{equation*}
\item \label{bnd} There exists a constant $C$ such that for all $\om\in \Om$, 
\begin{equation*}
0\leq R(I, \om)\leq C|I|.
\end{equation*} 
\end{list}
The purpose of this section is to establish
\begin{theorem}\label{genset}
Let $R(I, \om)$ satisfy \pref{sethypot1}, \pref{sethypot2}, and \pref{bnd}, and consider a sequence of cubes $\left\{I_{\textbf{n}_{j}}\right\}$ with $I_{\textbf{n}_{1}}\subset I_{\textbf{n}_{2}}\ldots$, and $I_{\textbf{n}_{j}}\rightarrow \RR^{d}_{+}=\left\{x\in \RR^{d}: x_{i}\geq 0\right\}$ as $j\rightarrow \infty$. Then $\displaystyle \lim_{j\rightarrow\infty} \frac{R(I_{\textbf{n}_{j}}, \om)}{|I_{\textbf{n}_{j}|}}$ exists, and converges almost surely to a function $R\in L^{1}(\Om)$. Moreover, we have that
\begin{equation*}
\EE[R]=\lim_{j\rightarrow\infty} \frac{1}{|I_{\textbf{n}_{j}}|} \int R(I_{\textbf{n}_{j}}, \om)~d\PP.
\end{equation*}
In particular, if $\tau_{s}$ is ergodic, then $R(\om)=\ga(R)=:\inf_{K\in \mathcal{I}} \frac{1}{|K|} \int R(K, \om)~d\PP$ a.s.
\end{theorem}

We point out that in contrast to the result of Akcoglu and Krengel, we do not require that there exists a constant $C$ so that $|I_{\textbf{n}_{j}}|\leq C |I_{n}|$, where $I_{n}=[0,n)^{d}$. In particular, Theorem \ref{genset} allows us to apply the subadditive ergodic theorem to parabolic cubes $K_{n}=[0, n^{2})\times [0, n)^{d}$, which are needed in the proof of Theorem \ref{stochhom} and Theorem \ref{thmrate1}. 

We show that the proof of Akcoglu and Krengel can be extended to the setting of Theorem \ref{genset}. We point out that it is enough to prove a multiparameter ergodic theorem for additive processes under the hypotheses above. One can then extend to subadditive processes using the same argument of \cite{akerg}, and thus we omit the details of this step. We first establish a general maximal inequality which will be used at various points in the proof. We then prove a multiparameter ergodic theorem for additive processes using an inductive argument. This will be enough to conclude the proof of Theorem \ref{genset}. 

%In the case when $d=1$, this result is known as Kingman's subadditive ergodic theorem, and there are various proofs of this result.  The authors of \cite{avilabochi} presented a simplified, direct proof of Kingman's subadditive ergodic theorem based upon the following observation. If $F^{*}(\om)=\limsup_{j\rightarrow\infty} \frac{F(I_{j}, \om)}{j}$, and $F_{*}(\om)=\liminf_{j\rightarrow\infty} \frac{F(I_{j},\om)}{j}$, and $\EE[F^{*}]=\EE[F_{*}]$, then $F^{*}(\om)=F_{*}(\om)$ almost surely and we have almost sure convergence of $\frac{F(I_{j}, \om)}{j}$. We borrow from this idea in the proof of Theorem \ref{genset}. 

We prove a discrete version of Theorem \ref{genset}. The proof for continuous processes is done by a standard approximation argument (see \cite{akerg}), using the fact that $R(I, \om)$ is continuous in $I$ by \pref{bnd}.

%For convenience, we denote $I^{1}_{\textbf{n}_{j}=}=[0, 1]^{d-1}\times [0, n^{1}_{j})$, $I^{2}_{\textbf{n}_{j}}=[0, 1]^{d-2}\times [0, n^{1}_{j})\times [0, n^{2}_{j}) \ldots$ where now $\textbf{n}_{j}\in \ZZ^{d}$ for all $j$. 

%We present several maximal inequalities which we will use throughout the proof of Theorem \ref{genset}. We recall the maximal ergodic theorem . 

%\begin{proposition}
%Let $f\in L^{1}(\Om), \al>0$. Let 
%\begin{equation*}
%E_{N}=\left\{\om: \sup_{1\leq j\leq N}\left|\sum_{j=0}^{N} f(\tau_{(0, 0, \ldots j)}\om)\right|>\al\right\}.
%\end{equation*}
%Then 
%\begin{equation*}
%\PP[E_{N}]\leq \frac{1}{\al}\int _{E_{N}} |f|~d\PP
%\end{equation*}
%and in particular, 
%\begin{equation*}
%\PP\left[\om: \sup_{j}\left|\sum_{j=0}^{N} f(\tau_{(0, 0, \ldots j)}\om)\right|\right]
%\end{equation*}
%\end{proposition}

The first lemma we present is similar to Vitali's covering lemma, for cubes of arbitrary dimension.
\begin{lemma}\label{covering}
Consider a finite subset $A\subset \ZZ^{d}$ such that for all $u\in A$, there exists a corresponding vector $\textbf{n}(u)\in \RR^{d}$. Then there exists $u_{1}, u_{2}, \ldots u_{\ell}\in A$ so that $\left\{ u_{i}+I_{\textbf{n}(u_{i})}\right\}_{i=1}^{\ell}$ are disjoint, and
\begin{equation*}
3^{d} \sum_{i=1}^{\ell} |I_{\textbf{n}(u_{i})}|\geq |A|.
\end{equation*}
\end{lemma}

\begin{proof}
Let $u_{1}$ be so that $\textbf{n}(u_{1})=\max_{u\in A}\left\{|I_{\textbf{n}(u)}|\right\}$, and set $K_{\textbf{n}(u_{1})}=u_{1}+I_{\textbf{n}(u_{1})}$. Define 
\begin{equation*}
\tilde{K}_{\textbf{n}(u_{1})}=\bigcup_{ u\in \partial K_{\textbf{n}(u_{1})}} u\pm I_{\textbf{n}(u_{1})}.
\end{equation*}
Then $|\tilde{K}_{\textbf{n}(u_{1})}|=3^{d} |I_{\textbf{n}(u_{1})}|$. Let $A_{1}=A\setminus ( \tilde{K}_{\textbf{n}(u_{1})}\cap A)$. If $A_{1}=\varnothing$, then we are done. Otherwise, we continue this process by choosing $u_{2}$ such that $\textbf{n}(u_{2})=\max_{u\in A_{1}}\left\{|I_{\textbf{n}(u)}|\right\}$. In this way, we select a sequence of vectors $\left\{u_{i}\right\}$ such that $u_{i}+I_{\textbf{n}(u_{i})}$ are disjoint, cover all of $A$, and stay within $\bigcup_{i} \tilde{K}_{\textbf{n}(u_{i})}$. By the construction of $\tilde{K}_{\textbf{n}(u_{i})}$, we have 
\begin{equation*}
3^{d} \sum_{i=1}^{\ell} |I_{\textbf{n}(u_{i})}|\geq |A|
\end{equation*}
as asserted. 
\end{proof}

Using this covering argument, we now prove a general maximal inequality, which holds in every dimension $d$. 
\begin{proposition}\label{max}
Let $H$ be a nonnegative, discrete, superadditive process. Let $\al>0$, and let $E=\left\{\om: \sup_{j\geq 1} \frac{1}{|I_{\textbf{n}_{j}}|} H(I_{\textbf{n}_{j}}, \om)>\al\right\}$. Then 
\begin{equation*}
\PP[E]\leq \frac{3^{d}}{\al}\left(\lim_{j\rightarrow\infty} \frac{1}{|I_{\textbf{n}_{j}}|}\int H(I_{\textbf{n}_{j}}, \om)~d\PP\right).
\end{equation*}
\end{proposition}
\begin{proof}
Fix $J>0$, and let $E_{J}=\left\{\om: \sup_{1\leq j\leq J} \frac{1}{|I_{\textbf{n}_{j}}|} H(I_{\textbf{n}_{j}}, \om)>\al\right\}$. 
Choose $K>J$, fix $\om\in \Om$, and define 
\begin{equation*}
A(\om)=\left\{u\in \ZZ^{d}: u+I_{\textbf{n}_{J}}\subset I_{\textbf{n}_{K}},  \tau_{u}\om\in E_{J}\right\}.
\end{equation*}
We note that $A(\om)$ is finite for all $\om$ since we are in the discrete case. By definition, for each $u\in A(\om)$, there exists $\textbf{n}_{j}=\textbf{n}_{j}(u)$, with $1\leq j\leq J$ so that 
\begin{equation*}
H(u+I_{\textbf{n}_{j}}, \om)=H(I_{\textbf{n}_{j}}, \tau_{u}\om)> \al |I_{\textbf{n}_{j}}|.
\end{equation*}
Lemma \ref{covering} yields a subset $\left\{u_{i}\right\}_{i=1}^{\ell}$ such that $\left\{u_{i}+I_{\textbf{n}_{j}(u_{i})}\right\}_{i=1}^{\ell}$ are disjoint and 
 \begin{equation*}
 3^{d}\sum_{i=1}^{\ell} |I_{\textbf{n}(u_{i})}|\geq |A(\om)|.
 \end{equation*}
 
The nonnegativity and superadditivity of $H$ yield
 \begin{equation*}
 3^{d} H(I_{\textbf{n}_{K}}, \om)\geq 3^{d}\sum_{i=1}^{\ell} H(u_{i}+I_{\textbf{n}(u_{i})}, \om)\geq 3^{d}\al\sum_{i=1}^{\ell} |I_{\textbf{n}(u_{i})}|\geq \al |A(\om)|.
 \end{equation*}
 
 Therefore, 
 \begin{align*}
 3^{d}\int H(I_{\textbf{n}_{K}}, \om)~d\PP&\geq \al \int |A(\om)|~d\PP\\
 &=\al \int \sum_{u+I_{\textbf{n}_{J}}\subset I_{\textbf{n}_{K}}} \mathds{1}_{E_{J}}(\tau_{u}\om)~d\PP\\
 &=\al \prod_{p=1}^{d} (\textbf{n}^{p}_{K}-\textbf{n}^{p}_{J} )\PP(E_{J})
 \end{align*} 
 where $\textbf{n}^{p}_{K}$ is the $p$-th coordinate of $\textbf{n}_{K}$. Hence, 
 \begin{equation*}
  3^{d}\frac{1}{|I_{\textbf{n}_{K}|}} \int H(I_{\textbf{n}_{K}}, \om)~d\PP\geq \al \frac{\prod_{p} (\textbf{n}^{p}_{K}-\textbf{n}^{p}_{J} )}{|I_{\textbf{n}_{K}}|} \PP(E_{J}).
  \end{equation*}
  Sending $K\rightarrow\infty$, then $J\rightarrow\infty$, we have
  \begin{equation*}
  \PP(E)\leq \frac{3^{d}}{\al} \lim_{K\rightarrow\infty}\frac{1}{|I_{\textbf{n}_{K}|}} \int H(I_{\textbf{n}_{K}}, \om)~d\PP.
  \end{equation*} 
\end{proof}

We now prove a multiparameter ergodic theorem for additive processes. In order to do so, we proceed by induction on the dimension of the space. 

\begin{theorem}\label{multiet}
 Let $G(I_{\textbf{n}_{j}}, \om)$ be a nonnegative, discrete, additive process, with the same assumptions as Theorem \ref{genset} on $\left\{I_{\textbf{n}_{j}}\right\}$. Then  $\displaystyle \lim_{j\rightarrow\infty} \frac{G(I_{\textbf{n}_{j}}, \om)}{|I_{\textbf{n}_{j}|}}$ converges a.s. to a function $G\in L^{1}(\Om)$. Moreover,
\begin{equation*}
\EE[G]=\lim_{j\rightarrow\infty} \frac{1}{|I_{\textbf{n}_{j}}|} \int G(I_{\textbf{n}_{j}}, \om)~d\PP.
\end{equation*}
\end{theorem}

\begin{proof}
We denote $I^{1}_{\textbf{n}_{j}}=[0, 1]^{d-1}\times [0, n^{1}_{j})$, $I^{2}_{\textbf{n}_{j}}=[0, 1]^{d-2}\times [0, n^{1}_{j})\times [0, n^{2}_{j}) \ldots$ where $\textbf{n}_{j}\in \ZZ^{d}$ for all $j$, and recall that $I_{n}=[0, n)^{d}$. 

In the case when $d=1$, Birkhoff's ergodic theorem ensures that there exists $G^{1}(\cdot)\in L^{1}(\Om)$ such that $\displaystyle\lim_{n\rightarrow\infty}\frac{G(I^{1}_{n}, \om)}{n}\rightarrow G^{1}(\om)$ a.s., and 
\begin{equation*}
\EE[G^{1}]=\lim_{n\rightarrow\infty}\frac{1}{n}\int G(I^{1}_{n}, \om)~d\PP.
\end{equation*}
%In particular, this implies that $\lim_{j\rightarrow\infty}\frac{G(I^{1}_{\textbf{n}_{j}}, \om)}{n^{1}_{j}}= G^{1}(\om)$ almost surely and $\EE[G^{1}]=\lim_{j\rightarrow\infty}\frac{1}{\textbf{n}^{1}_{j}}\int G(I^{1}_{\textbf{n}_{j}}, \om)~d\PP$.

We suppose now that the result holds true in dimension $d-1$. We modify the notation at this step in order to present the ideas clearly. We let $I_{j}=I^{d}_{\textbf{n}_{j}}$ and $I'_{j}=I^{d-1}_{\textbf{n}_{j}}$. By the inductive hypothesis, there exists $G^{d-1}(\om)$ such that $\displaystyle \lim_{j\rightarrow\infty}\frac{G(I'_{j}, \om)}{|I'_{j}|}=G^{d-1}(\om)$ a.s., and 
\begin{equation*}
\EE[G^{d-1}]=\lim_{j\rightarrow\infty}\frac{1}{|I'_{j}|} \int G(I'_{j}, \om)~d\PP.
\end{equation*}

We define the linear operator $A_{j}: L^{1}(\Om)\rightarrow L^{1}(\Om)$ by
\begin{equation*}
A_{j}[f](\om)=\frac{1}{n^{d}_{j}}\sum_{k=0}^{n^{d}_{j}-1} f(\tau_{(0, 0, \ldots, k)}\om).
\end{equation*}

We note that for every $j$, 
\begin{align*}
\int\left| A_{j}[f](\om)\right|~d\PP&=\int \left|\frac{1}{n^{d}_{j}}\sum_{k=0}^{n^{d}_{j}-1} f(\tau_{(0, 0, \ldots, k)}\om)\right|~d\PP\\
&\leq \frac{1}{n^{d}_{j}}\sum_{k=0}^{n^{d}_{j}-1}  \int |f(\tau_{(0, 0, \ldots, k)}\om)|~d\PP\\
&= \frac{1}{n^{d}_{j}}\sum_{k=0}^{n^{d}_{j}-1} \int |f(\om)|~d\PP
\end{align*}
and hence, 
\begin{equation}\label{opbound}
\sup_{j} \norm{A_{j}}_{1}=\sup_{j} \sup_{\norm{f}_{L^1}\leq 1} \norm{A_{j}[f]}_{L^{1}}\leq 1. 
\end{equation}

By additivity, 
\begin{equation*}
\frac{1}{|I_{j}|}G(I_{j}, \om)= \frac{1}{n^{d}_{j}}\sum_{k=0}^{n^{d}_{j}-1} \frac{G(I'_{j}, \tau_{(0, 0, \ldots, k)}\om)}{|I'_{j}|}=A_{j}\left[\frac{G(I'_{j}, \cdot)}{|I'_{j}|}\right](\om).
\end{equation*}

For simplicity, we define $G'_{j}(\cdot)=\frac{G(I'_{j},\cdot)}{|I'_{j}|}$. Recall by the inductive hypothesis, $G'_{j}(\om)\rightarrow G^{d-1}(\om)$ a.s. We now claim that $A_{j}[G'_{j}](\om)$ converge in $L^{1}(\Om)$. Indeed, 

\begin{align}
\int \left|A_{j}[G'_{j}](\om)-A_{\ell}[G'_{\ell}](\om)\right|~d\PP &\leq \int \left|A_{j}[G'_{j}-G'_{\ell}](\om)\right|~d\PP+\int \left|A_{j}[G'_{\ell}](\om)-A_{\ell}[G'_{\ell}](\om)\right|~d\PP\\
&\leq \norm{A_{j}}_{1}\norm{G'_{j}-G'_{\ell}}_{L^{1}}+\int \left|A_{j}[G'_{\ell}](\om)-A_{\ell}[G'_{\ell}](\om)\right|~d\PP. \label{est}
\end{align}

Using \pref{opbound} and the inductive hypothesis, we claim that for $j, \ell$ sufficiently large, 
\begin{equation}\label{bnd1}
\norm{A_{j}}_{1}\norm{G'_{j}-G'_{\ell}}_{L^{1}}\leq \frac{\ve}{3}.
\end{equation}

To control the second term of \pref{est}, we observe that 
\begin{align*}
\int \left|\left(A_{j}-A_{\ell}\right)[G'_{\ell}](\om)\right|~d\PP &\leq \int \left|\left(A_{j}-A_{\ell}\right)[G^{d-1}](\om)\right|~d\PP+\\
&\int \left|\left(A_{j}-A_{\ell}\right)[G'_{\ell}-G^{d-1}](\om)\right|~d\PP\\
&\leq  \int \left|\left(A_{j}-A_{\ell}\right)[G^{d-1}](\om)\right|~d\PP+2\sup_{j} \norm{A_{j}}_{1}\norm{G'_{\ell}-G^{d-1}}_{L^{1}}.
\end{align*}

We note that $A_{j}$ is the standard 1-dimensional Birkhoff ergodic summation operator, and since $G^{d-1}\in L^{1}(\Om)$, we may apply the Birkhoff Ergodic Theorem to conclude that for $j, \ell$ sufficiently large,  
\begin{equation}\label{bnd2}
 \int \left|\left(A_{j}-A_{\ell}\right)[G^{d-1}](\om)\right|~d\PP\leq \frac{\ve}{3}. 
 \end{equation}

Furthermore, \pref{opbound} and the fact that $G'_{j}(\om)$ converge almost surely to $G^{d-1}$ yields that for $\ell$ sufficiently large, 
\begin{equation}\label{bnd3}
2\sup_{j} \norm{A_{j}}_{1}\norm{G'_{\ell}-G^{d-1}}_{L^{1}}\leq \frac{\ve}{3}.
\end{equation}

 Combining \pref{bnd1}, \pref{bnd2}, and \pref{bnd3}, we have that $\ell, j$ sufficiently large, 
\begin{equation*}
\int \left|A_{j}[G'_{j}](\om)-A_{\ell}[G'_{\ell}](\om)\right|~d\PP\leq \ve.
\end{equation*}
Since $\left\{A_{j}[G'_{j}]\right\}$ are Cauchy in $L^{1}(\Om)$, there exists $A(\om)\in L^{1}(\Om)$ such that  $\frac{1}{|I_{j}|}G(I_{j}, \om)=A_{j}[G'_{j}](\om)\rightarrow A(\om)$ in $L^{1}(\Om)$. We now show that $\frac{1}{|I_{j}|}G(I_{j}, \om)$ converge almost uniformly, which is enough to conclude almost sure convergence. Consider that
 \begin{align*}
 \limsup_{j\rightarrow\infty}\frac{1}{|I_{j}|}G(I_{j}, \om)-\liminf_{j\rightarrow\infty} \frac{1}{|I_{j}|}G(I_{j}, \om) &=\limsup_{j\rightarrow\infty}\frac{1}{|I_{j}|}G(I_{j}, \om)-A(\om)+A(\om)-\liminf_{j\rightarrow\infty} \frac{1}{|I_{j}|}G(I_{j}, \om)\\
 &\leq 2\sup_{j} \left|\frac{1}{|I_{j}|}|G(I_{j}, \om)-A(\om)|I_{j}|\right|.
 \end{align*}
 Moreover, since  $\frac{1}{|I_{j}|}G(I_{j}, \om)\rightarrow A(\om)$ in $L^{1}(\Om)$, there exists $\rho, \eta$ such that for all $j$ sufficiently large, $G(I_{j}, \om)-A(\om)|I_{j}|+\rho\geq0$ and $A(\om)|I_{j}|-G(I_{j}, \om)+\eta\geq0$ almost surely. Therefore, for any $\al>0$, 
\begin{align*}
\PP\left[\sup_{j} \left|\frac{1}{|I_{j}|}|G(I_{j}, \om)-A(\om)|I_{j}|\right|>\frac{\al}{2}\right]&\leq \PP\left[\sup_{j} \frac{1}{|I_{j}|}\left(G(I_{j}, \om)-A(\om)|I_{j}|\right)>\frac{\al}{2}\right]\\
&+\PP\left[\sup_{j} \frac{1}{|I_{j}|}\left(A(\om)|I_{j}|-G(I_{j}, \om)\right)>\frac{\al}{2}\right]\\
&\leq  \PP\left[\sup_{j} \frac{1}{|I_{j}|}\left(G(I_{j}, \om)-A(\om)|I_{j}|+\rho\right)>\frac{\al}{2}\right]\\
&+\PP\left[\sup_{j} \frac{1}{|I_{j}|}\left(A(\om)|I_{j}|-G(I_{j}, \om)+\eta\right)>\frac{\al}{2}\right].
\end{align*}
By Proposition \ref{max}, 
\begin{align*}
\PP\left[\limsup_{j\rightarrow\infty}\frac{1}{|I_{j}|}G(I_{j}, \om)-\liminf_{j\rightarrow\infty} \frac{1}{|I_{j}|}G(I_{j}, \om)>\al\right] &\leq \PP\left[\sup_{j} \frac{1}{|I_{j}|}\left(G(I_{j}, \om)-A(\om)|I_{j}|+\rho\right)>\frac{\al}{2}\right]\\
&+\PP\left[\sup_{j} \frac{1}{|I_{j}|}\left(A(\om)|I_{j}|-G(I_{j}, \om)+\eta\right)>\frac{\al}{2}\right]\\
&\leq 2 \frac{3^{d}}{\al}\lim_{j\rightarrow\infty}\frac{1}{|I_{j}|}\int |A(\om)|I_{j}|-G(I_{j}, \om)|~d\PP\\
&+\frac{3^{d}}{\al}\lim_{j\rightarrow\infty}\frac{\rho+\eta}{|I_{j}|}=0.
\end{align*}

Thus, $\lim_{j\rightarrow\infty} \frac{1}{|I_{j}|}G(I_{j}, \om)$ converge almost surely to a limit $G(\om)$, and by uniqueness, they must converge almost surely to the $L^{1}$-limit, $A(\om)$. We conclude that
\begin{equation*}
\int G(\om)~d\PP=\lim_{j\rightarrow\infty} \frac{1}{|I_{j}|} \int G(I_{j}, \om)~d\PP.
\end{equation*} 
 By induction, we have the conclusion of Theorem \ref{multiet} for all dimensions.
\end{proof}

\section{Proof of Technical Lemmas}\label{techlem}

\begin{proof}[Proof of Lemma \ref{lemmonot}]
We only show the argument for the decay of $\overline{J}_{(k+1)^{3}}$ to $\overline{J}_{k^{3}+3k^{2}}$, since the other cases follow similarly. We also drop the dependence on $\om$ since it plays no role in our analysis. By minimality of the obstacle problem, we have for all $i=1, 2,\ldots 3^{(3k+1)(d+2)}$, 
\begin{equation*}
\left\{\overline{v}^{(k+1)^{3}}=0\right\} \cap G^{i}_{k^{3}+3k^{2}}\subset \left\{\overline{v}^{k^{3}+3k^{2}}=0\right\}.
\end{equation*}
Therefore, we have that 
\begin{align*}
|G_{(k+1)^{3}}|\overline{\pi}_{(k+1)^{3}}(\om) &=\int_{G_{(k+1)^{3}}\cap \left\{\overline{v}^{(k+1)^{3}}=0\right\}\cap G_{(k+1)^{3}}} (\ell+F(0, y, s, \om))^{d+1}_{+} dy ds\\
&= \sum_{i=1}^{3^{(3k+1)(d+2)}} \int_{G^{i}_{k^{3}+3k^{2}}\cap \left\{\overline{v}^{(k+1)^{3}}=0\right\} }(\ell+F(0, y, s, \om))^{d+1}_{+} dy ds\\
&\leq \sum_{i=1}^{3^{(3k+1)(d+2)}} \int_{\left\{\overline{v}^{k^{3}+3k^{2}}=0\right\} }(\ell+F(0, y, s, \om))^{d+1}_{+} dy ds\\
&=|G_{k^{3}+3k^{2}}|  \sum_{i=1}^{3^{(3k+1)(d+2)}} \overline{\pi}^{i}_{k^{3}+3k^{2}}(\om).
\end{align*}
This yields
\begin{equation*}
\overline{\pi}_{(k+1)^{3}}(\om)\leq \frac{|G_{k^{3}+3k^{2}}|}{|G_{(k+1)^{3}}|} \sum_{i=1}^{3^{(3k+1)(d+2)}} \overline{\pi}^{i}_{k^{3}+3k^{2}}(\om)= \overline{A}_{(k+1)^{3}}(\om), 
\end{equation*}
which implies $\overline{J}_{(k+1)^{3}}\leq \EE\left[\left(\overline{A}_{k^{3}+3k^{2}}\right)^{2}\right]$. 
%hence,
%\begin{equation}\label{maleav}
%\overline{J}_{(k+1)^{3}}\leq \EE[\overline{A}_{k^{3}+3k^{2}}]^{2}.
%\end{equation}

Finally, we compare
\begin{align*}
\EE[(\overline{A}_{(k+1)^{3}})^{2}]&=3^{-2(3k+1)(d+2)}\EE\left[\left(\sum_{i=1}^{3^{(3k+1)(d+2)}} \overline{\pi}^{i}_{k^{3}+3k^{2}}\right)^{2}\right]\\
&=3^{-2(3k+1)(d+2)}\sum_{i,j=1}^{3^{(3k+1)(d+2)}} \EE\left[ \overline{\pi}^{i}_{k^{3}+3k^{2}} \cdot \overline{\pi}^{j}_{k^{3}+3k^{2}}\right]\\
&\leq  3^{-2(3k+1)(d+2)}\sum_{i,j=1}^{3^{(3k+1)(d+2)}} \left(\EE\left[ (\overline{\pi}^{i}_{k^{3}+3k^{2}})^{2}\right]\right)^{1/2}\left( \EE\left[(\overline{\pi}^{j}_{k^{3}+3k^{2}})^{2}\right]\right)^{1/2}\\
&=\EE\left[\left(\overline{\pi}_{k^{3}+3k^{2}}\right)^{2}\right]=\overline{J}_{k^{3}+3k^{2}}
\end{align*}
where we applied the Cauchy-Schwarz inequality and stationarity. Combining the above, we have $\overline{J}_{(k+1)^{3}}\leq \overline{J}_{k^{3}+3k^{2}}$.  
\end{proof}

\begin{proof}[Proof of Lemma \ref{decayv}]
We only show the argument for \pref{decays}, since the argument for \pref{decayb} follows similarly. We consider each of the subcubes $G^{ij}_{k^{3}}=G^{j}_{k^{3}}$, with centers $z_{j}$. (The centers of the cubes are centered in space and time.) We fix one of the subcubes $G^{\hat{j}}_{k^{3}}$ and consider all of the other subcubes which make any contact with $G^{\hat{j}}_{k^{3}}$.  These can be characterized by considering all of the subcubes so that $d(z_{i},z_{\hat{j}})< 2\cdot3^{k^{3}}$. There are at most $3^{d+1}$ such distinct subcubes, including $G^{\hat{j}}_{k^{3}}$.
Note that by \pref{1mixing}, 
\begin{equation*}
\EE[\overline{\pi}^{i}_{k^{3}}\overline{\pi}^{j}_{k^{3}}]\leq \overline{E}_{k^{3}}+3^{-k^{3/2}}\quad\text{if}\quad d(z_{i},z_{j})>3^{{k^{3}}}.
\end{equation*}
Moreover, by the Cauchy-Schwartz inequality and stationarity, 
\begin{equation*}
\EE[\overline{\pi}^{i}_{k^{3}}\overline{\pi}^{j}_{k^{3}}]-(\overline{E}_{k^{3}})^{2}\leq \overline{J}_{k^{3}}-(\overline{E}_{k^{3}})^{2}=\overline{V}_{k^{3}}
\end{equation*}
for all $i,j$. 
Therefore, 
\begin{align*}
V(\overline{A}_{k^{3}+3k^{2}})&=3^{-2(3k^{2})(d+2)}\EE\left[\sum_{i=1}^{3k^{2}(d+2)}(\overline{\pi}^{i}_{k^{3}}-\overline{E}^{i}_{k^{3}})\right]^{2}\\
&=3^{-2(3k^{2})(d+2)}\sum_{i,j=1}^{3k^{2}(d+2)}\EE[(\overline{\pi}^{i}_{k^{3}}-\overline{E}^{i}_{k^{3}})(\overline{\pi}^{j}_{k^{3}}-\overline{E}^{j}_{k^{3}})]\\
&=3^{-2(3k^{2})(d+2)}\sum_{i,j=1}^{3k^{2}(d+2)}\left(\EE[\overline{\pi}^{i}_{k^{3}}\overline{\pi}^{j}_{k^{3}}]-(\overline{E}_{k^{3}})^{2}\right)\\
&=3^{-2(3k^{2})(d+2)}\sum_{\substack{i,j\\d(z_{i}, z_{j})<2\cdot 3^{k^{3}}}}^{3k^{2}(d+2)}\left(\EE[\overline{\pi}^{i}_{k^{3}}\overline{\pi}^{j}_{k^{3}}]-(\overline{E}_{k^{3}})^{2}\right)\\
&+3^{-2(3k^{2})(d+2)}\sum_{\substack{i,j\\d(z_{i}, z_{j})\geq2\cdot 3^{k^{3}}}}^{3k^{2}(d+2)}\left(\EE[\overline{\pi}^{i}_{k^{3}}\overline{\pi}^{j}_{k^{3}}]-(\overline{E}_{k^{3}})^{2}\right)\\
&\leq 3^{-2(3k^{2})(d+2)}[3^{(3k^{2})(d+2)}3^{d+1}\overline{V}_{k^{3}}+(3^{(3k^{2})(d+2)})(3^{(3k^{2})(d+2)}-3^{d+1})3^{-k^{3/2}}]\\
%&= 3^{-(3k^{2})(d+2)}[3^{d+1}\overline{V}_{k^{3}}+(3^{(3k^{2})(d+2)}-3^{d+1})3^{-k^{3/2}}]\\
&\leq 3^{-(3k^{2})(d+2)+(d+1)}\overline{V}_{k^{3}}+3^{-k^{3/2}}.
\end{align*}
The argument for \pref{decayb} follows similarly. 
%We also write out the proof in the case of \pref{decayb} so we can see why there is a difference in the decay factors. In this case, we have
%\begin{align*}
%V(\overline{A}_{k^{3}+3k^{2}})&=3^{-2(3k+1)(N+2)}\EE\left[\sum_{i=1}^{(3k+1)(N+2)}(\overline{h}^{i}_{k^{3}+3k^{2}}-\overline{\EE}^{i}_{k^{3}+3k^{2}})\right]^{2}\\
%&=3^{-2(3k+1)(N+2)}\sum_{i,j=1}^{(3k+1)(N+2)}\EE[(\overline{h}^{i}_{k^{3}+3k^{2}}-\overline{\EE}^{i}_{k^{3}+3k^{2}})(\overline{h}^{j}_{k^{3}}-\overline{\EE}^{j}_{k^{3}+3k^{2}})]\\
%&=3^{-2(3k+1)(N+2)}\sum_{i,j=1}^{(3k+1)(N+2)}[\EE(\overline{h}^{i}_{k^{3}+3k^{2}}\overline{h}^{j}_{k^{3}+3k^{2}}-(\overline{\EE}_{k^{3}+3k^{2}})^{2}]\\
%&=3^{-2(3k+1)(N+2)}\sum_{\substack{i,j\\d(z_{i}, z_{j})<2\cdot 3^{k^{3}}}}^{(3k+1)(N+2)}[\EE(\overline{h}^{i}_{k^{3}+3k^{2}}\overline{h}^{j}_{k^{3}+3k^{2}}-(\overline{\EE}_{k^{3}+3k^{2}})^{2}]\\
%&+3^{-2(3k+1)(N+2)}\sum_{\substack{i,j\\d(z_{i}, z_{j})\geq2\cdot 3^{k^{3}}}}^{(3k+1)(N+2)}[\EE(\overline{h}^{i}_{k^{3}+3k^{2}}\overline{h}^{j}_{k^{3}+3k^{2}}-(\overline{\EE}_{k^{3}+3k^{2}})^{2}]\\
%&\leq 3^{-2(3k+1)(N+2)}[3^{(3k+1)(N+2)}3^{N+1}\overline{V}_{k^{3}+3k^{2}}+(3^{(3k+1)(N+2)})(3^{(3k+1)(N+2)}-3^{N+1})3^{-k^{3/2}}]\\
%&= 3^{-(3k+1)(N+2)}[3^{N+1}\overline{V}_{k^{3}+3k^{2}}+(3^{(3k+1)(N+2)}-3^{N+1})3^{-k^{3/2}}]\\
%&\leq 3^{-(3k+1)(N+2)+(N+1)}\overline{V}_{k^{3}+3k^{2}}+3^{-k^{3/2}}
%\end{align*}
\end{proof}

\begin{proof}[Proof of Lemma \ref{excsets}]
We only show the proofs for $\overline{\pi}_{k^{3}+3k^{2}}$ since those for $\underline{\pi}_{k^{3}+3k^{2}}$ follow similarly. By Chebyshev's inequality, if we define $\overline{\underline{B}}_{k^{3}+3k^{2}}$ by
\begin{equation}\label{chebyB}
\overline{B}_{k^{3}+3k^{2}}=\left\{\om\in\Om: |\overline{\pi}_{k^{3}+3k^{2}}(\om)-\overline{E}_{k^{3}+3k^{2}}|\geq \frac{1}{2}\overline{E}_{k^{3}+3k^{2}}\right\},
\end{equation}
then \pref{chebym} is immediate.

%Moreover, we have that
%\begin{equation*}
%\langle \overline{\rho}_{k^{3}}\rangle=3^{-(3k^{2})(d+2)}(3^{3k^{2}(d+2)}-\overline{N}_{k^{3}})
%\end{equation*}

Moreover, by \pref{badcase}, 
%\begin{equation*}
%\EE[\overline{\rho}^{ij}_{k^{3}}]=\PP\left[\om\in\Om: \overline{\pi}^{ij}_{k^{3}}(\om)\leq \frac{1}{2}\overline{E}_{k^{3}}\right]
%\end{equation*}
%and
\begin{align*}
\overline{V}_{k^{3}}=\EE\left[\left(\overline{\pi}^{ij}_{k^{3}}(\cdot)-\overline{E}_{k^{3}}\right)^{2}\right] & \geq \int [\overline{\pi}^{ij}_{k^{3}}(\cdot)-\overline{E}_{k^{3}}]^{2}\mathbbm{1}_{  \left\{\overline{\pi}^{ij}_{k^{3}}(\cdot)\leq \frac{1}{2}\overline{E}_{k^{3}}\right\} } (\om)~d\PP\\
&\geq \frac{1}{4}(\overline{E}_{k^{3}})^{2}\PP\left[\overline{\pi}^{ij}_{k^{3}}\leq \frac{1}{2}\overline{E}_{k^{3}}\right],
\end{align*}
which implies that  
\begin{equation}
\PP\left[\overline{\pi}^{ij}_{k^{3}}\leq \frac{1}{2}\overline{E}_{k^{3}}\right]\leq \frac{4\overline{V}_{k^{3}}}{(\overline{E}_{k^{3}})^{2}}\leq 4\eta.
\end{equation}

By stationarity, we have that
\begin{equation}
\EE[\overline{a}^{ij}_{k^{3}+3k^{2}}]\leq 4\eta, 
\end{equation}
so by Chebyshev's inequality and \pref{zeta}, 
\begin{equation}\label{angleB}
\PP\left[\langle \overline{B}_{k^{3}}\rangle\right]=\PP\left[\overline{a}^{ij}_{k^{3}+3k^{2}}>\zeta_{d}\right] \leq 4\eta \zeta_{d}^{-1}.
\end{equation}

We note that for $\om\in\Om\setminus \langle\overline{B}_{k^{3}}\rangle=\left\{\om\in\Om: \overline{a}^{ij}_{k^{3}+3k^{2}}<\zeta_{d}\right\}$, 
\begin{equation*}
3^{-3k^{2}(d+2)}\left(3^{3k^{2}(d+2)}-\overline{N}_{k^{3}}\right)<\zeta_{d},
\end{equation*}
which implies
\begin{equation*}
\overline{N}_{k^{3}}\geq (1-\zeta_{d})3^{3k^{2}(d+2)}.
\end{equation*}
Finally, by combining \pref{chebym}, \pref{angleB}, and the analogous statements for $\underline{\pi}_{k^{3}+3k^{2}}$, 
\begin{equation*}
\PP[\textbf{B}_{k}]\leq \PP[\overline{B}_{k^3+3k^{2}}]+\PP[\underline{B}_{k^{3}+3k^{2}}]+\PP[\langle\overline{B}_{k^{3}}\rangle]+\PP[\langle\underline{B}_{k^{3}}\rangle]\leq 8\eta \left[1+\zeta^{-1}_{d}\right].
\end{equation*}
\end{proof}

\begin{proof}[Proof of Lemma \ref{fshomog}]
%Recall that by \pref{obsteqa} and \pref{obsteqb}, $h^{k^{3}+3k^{2}}(y,s, \om)=\overline{v}^{k^{3}+3k^{2}}(y,s, \om)-\underline{v}^{k^{3}+3k^{2}}(y,s,\om)$ is a nonnegative supersolution to 
%\begin{equation}\label{diff}
%h_{s}-\MM^{-}(D^{2}h, y,s, \om)\geq f_{k^{3}+3k^{2}}:=(\ell+F(0, y, s, \om))_{+}\chi_{\left\{\overline{v}^{k^{3}+3k^{2}}=0\right\}}+(\ell+F(0, y, s, \om))_{-}\chi_{\left\{\underline{v}^{k^{3}+3k^{2}}=0\right\}}. 
%\end{equation}

Recall that $h^{k^{3}+3k^{2}}$ is a supersolution to
\begin{equation*}
h_{s}-\MM^{-}(D^{2}h)\geq f_{k^{3}+3k^{2}}=(\ell+F(0, y, s, \om))_{+}\chi_{\left\{\overline{v}^{k^{3}+3k^{2}}=0\right\}}+(\ell+F(0, y, s, \om))_{-}\chi_{\left\{\underline{v}^{k^{3}+3k^{2}}=0\right\}}
\end{equation*}

Therefore, since $\om\in\Om_{1}, k^{3}+3k^{2}$ satisfy \pref{frac}, by Proposition \ref{FSL}, there exists $c_{fs}$ so that for all $|y|\leq \frac{2}{3} 3^{k^{3}+3k^{2}}$, $\frac{2}{2}3^{-2(k^{3}+3k^{2})}\geq s\geq 0$,
\begin{equation}\label{lbg}
h^{k^{3}+3k^{2}}(y,s, \om)\geq c_{fs}3^{(k^{3}+3k^{2})(2-(d+2)\al)}\norm{f_{k^{3}+3k^{2}}}_{L^{d+1}(G_{k^{3}+3k^{2}}}^{\al}
\end{equation}
with $\al=\rho+\beta \log \left(\frac{2\cdot 3^{(d+2)(k^{3}+3k^{2})}}{\norm{f}_{L^{d+1}(G_{k^{3}+3k^{2}}}^{d+1}}\right)$. 

Using the fact that $(\ell+F(0, y, s, \om))_{-}(\ell+F(0, y, s, \om))_{+}=0$,
\begin{align}
\norm{f_{k^{3}+3k^{2}}}_{L^{d+1}(G_{k^{3}+3k^{2}})}^{d+1} &=\int_{G_{k^{3}+3k^{2}}}[(\ell+F(0, y, s, \om))_{+}\chi_{\left\{\overline{v}^{k}=0\right\}}+(\ell+F(0, y, s, \om))_{-}\chi_{\left\{\underline{v}^{k}=0\right\}}]^{d+1} dy ds\notag\\
&= \int_{G_{k^{3}+3k^{2}}}(\ell+F(0, y, s, \om))^{d+1}_{+}\chi_{\left\{\overline{v}^{k}=0\right\}}+(\ell+F(0, y, s, \om))^{d+1}_{-}\chi_{\left\{\underline{v}^{k}=0\right\}}~ dy ds\notag\\
&= |G_{k^{3}+3k^{2}}|(\overline{\pi}_{k^{3}+3k^{2}}+\underline{\pi}_{k^{3}+3k^{2}})\geq 2|G_{k^{3}+3k^{2}}|\theta^{1/2}.\label{flbnd}
\end{align}

Combining this with \pref{lbg}, we have for all $(y,s)\in \frac{2}{3} G^{k^{3}+3k^{2}}$, 
\begin{align*}
h^{k^{3}+3k^{2}}(y,s, \om)&\geq c_{fs}3^{(k^{3}+3k^{2})(2-(d+2)\al)}|G_{k^{3}+3k^{2}}|^{\al}\theta^{\al/2}\\
&=c_{fs}3^{2(k^{3}+3k^{2})}\theta^{\al/2}.
\end{align*}
%Furthermore,
This implies that for all $(y,s)\in \frac{2}{3} G^{k^{3}+3k^{2}}$, 
\begin{equation}
\left(\overline{v}^{k^{3}+3k^{2}}(y,s, \om)-0\right)+\left(0-\underline{v}^{k^{3}+3k^{2}}(y,s,\om)\right)\geq C_{FS}3^{2(k^{3}+3k^{2})}\theta^{(\al/2)}.
\end{equation}

%Since both terms in the above sum are positive, this means that for every $(y,s)$ in the range we are discussing, we must have that either
%\begin{equation}
%\overline{v}_{m+\ell}-M\geq (1/2)C_{FS}3^{2(m+\ell)}\theta^{(\al/2)}\quad\text{or}\quad M-\underline{v}_{m+\ell}\geq (1/2)C_{FS}3^{2(m+\ell)}\theta^{(\al/2)}
%\end{equation}

If there exists $G^{i}_{k^{3}}\subset \frac{2}{3} G_{k^{3}+3k^{2}}$ such that both $\overline{v}^{k^{3}+3k^{2}}$ and $\underline{v}^{k^{3}+3k^{2}}$ vanish, then by the Holder continuity of $\overline{v}^{k^{3}+3k^{2}}$ and $\underline{v}^{k^{3}+3k^{2}}$, properly rescaled,
\begin{align*}
\max_{G^{i}_{k^{3}}}(\overline{v}^{k^{3}+3k^{2}}(\cdot, , \cdot, \om)-\underline{v}^{k^{3}+3k^{2}}(\cdot, , \cdot, \om))&\leq \max_{G^{i}_{k^{3}}} (\overline{v}^{k^{3}+3k^{2}}(\cdot, \cdot, \om)-0)+\max_{D^{i}_{k^{3}}} (0-\underline{v}^{k^{3}+3k^{2}})\\
&\leq 2C_{h}3^{(k^{3}+3k^{2})[(d/d+1)-\sigma]}3^{\sigma k^{3}}=2C_{h}3^{(k^{3}+3k^{2})(d/d+1)-\sig3k^{2}}.
\end{align*}

Combining the above, we have
\begin{equation*}
2C_{h}3^{(k^{3}+3k^{2})(d/d+1)-\sig3k^{2}}\geq  c_{fs}3^{2(k^{3}+3k^{2})}\theta^{(\al/2)}, 
\end{equation*}
which implies
\begin{equation*}
2\frac{C_{h}}{c_{fs}}\geq 3^{3k^{2}(2+\sig-d/d+1)}3^{k^{3}(2-d/d+1)} \theta^{(\al/2)}
\end{equation*}
which is impossible in light of the hypotheses.
\end{proof}

\section{A Quantitative Regularity Estimate}\label{quantreg}
We review the regularity estimates established in \cite{linreg1} which play a crucial role in the proof of Theorem \ref{thmrate1}. We explain everything in terms of a general $Q_{R}$ instead of $G_{k^{3}+3k^{2}}$. We recall the following general result from \cite{linreg1}:

\begin{theorem}[Lin, \cite{linreg1}]\label{genfs}
Let $0\leq f\leq \norm{f}_{L^{\infty}(Q_{R})}$, and let $u$ be a nonnegative function in $\overline{S}(f, Q_{R})$. There exists $c, C, \rho, \beta> 0$, depending on $\la, \La, d$, such that  for all $|x|\leq \frac{2}{3}R$, $0\geq t\geq -\frac{2}{3} \frac{\norm{f}^{d+1}_{L^{d+1}(Q_{R})}}{2}R^{-d}$, 
\begin{equation}\label{techest}
cR^{2-\left(d+2\right)\al}\norm{f}^{1-(d+1)\al}_{L^{\infty}(Q_{R})}\norm{f}_{L^{d+1}(Q_{R})}^{(d+1)\al}\leq u(x,t)\leq C R^{d/(d+1)} \norm{f}_{L^{d+1}(Q_{R})}
\end{equation}
with $\al=\rho+2\beta R^{d+2}\norm{f}_{L^{d+1}(Q_{R})}^{-(d+1)}$. 
\end{theorem}

We point out that the domain where \pref{techest} holds depends on $\norm{f}_{L^{d+1}(Q_{R})}$. This is problematic in adapting the proof of \cite{cs} to this setting, because the domain is shrinking too fast to gain information in a subset of $Q_{R}$ as $R$ increases.

It was shown in \cite{linreg1} that \pref{techest} holds in the domain for all $t\geq t_{0}$ for  
\begin{equation*}
t_{0}=\sup_{t}\left\{(x,t): f> 2^{-1}R^{d+2}\norm{f}^{d+1}_{L^{d+1}(Q_{R})}\right\}.
\end{equation*}
Therefore, to obtain \pref{techest} in a domain such as $\frac{2}{3}Q_{R}$, it is enough to show that 
\begin{equation*}
\left\{(x,t); f> 2^{-1}R^{d+2}\norm{f}^{d+1}_{L^{d+1}(Q_{R})}\right\}\subset\left\{t\leq -\frac{2}{3}R^{2}\right\}.
\end{equation*}
This is where Lemma \ref{lemmaka} plays a role. 

Furthermore, we comment that if we are in the situation of Lemma \ref{lemmaka}, we may improve the estimate of \cite{linreg1}. The key contribution of \cite{linreg1} was to obtain a quantitative lower bound for solutions to 
\begin{equation*}
\begin{cases}
w_{t}-\MM^{-}(D^{2}w)=\chi_{Q_{r}(x_{0}, t_{0})}\quad\text{in}\quad Q_{R},\\
w=0\quad\text{on}\quad \partial_{p}Q_{R},
\end{cases}
\end{equation*}
where $Q_{r}(x_{0}, t_{0})\subset Q_{R}$. If $t_{0}\leq -\ka R^{2}$, then it is possible to rework the argument of  Proposition 2.5 of \cite{linreg1}, to show that the lower bound decreases by $r^{\beta}$ in each step of the iteration, where $\beta=\beta(\la, \La, d)$. Therefore, once may improve the estimate to show that 
\begin{equation} 
w(x,t)\geq C r^{\al}
\end{equation}
where $\al=\rho+\beta \log (R^{d+2}/\norm{f}^{d+1}_{L^{d+1}(Q_{R}})$. We omit the details here. This yields the following improvement to Theorem \ref{genfs}
\begin{theorem}\label{thisfs}
Let $0\leq f\leq \norm{f}_{L^{\infty}(Q_{R})}$, and let $u$ be a nonnegative function in $\overline{S}(f, Q_{R})$. Suppose that $\left\{f\geq 2^{-1}R^{d+2}\norm{f}^{d+1}_{L^{d+1}(Q_{R})}\right\}\subset \left\{t\leq -\frac{2}{3} R^{2}\right\}$. There exists $c_{fs}, C, \rho, \beta> 0$, depending on $\la, \La, d$, such that  for all $|x|\leq \frac{2}{3}R$, $0\geq t\geq -\frac{2}{3} R^{2}$,
\begin{equation}
c_{fs}R^{2-\left(d+2\right)\al}\norm{f}^{1-(d+1)\al}_{L^{\infty}(Q_{R})}\norm{f}_{L^{d+1}(Q_{R})}^{(d+1)\al}\leq u(x,t)\leq C R^{d/(d+1)} \norm{f}_{L^{d+1}(Q_{R})}
\end{equation}
with $\al=\rho+\beta \log (R^{d+2}/\norm{f}_{L^{d+1}(Q_{R})})$. 
\end{theorem}

\section{Regularized Solutions for Parabolic Equations}
\subsection{Regularizing Solutions of Parabolic Equations}
We state some results regarding regularizations of solutions to parabolic equations. We state these results without proof, as the details can be found in \cite{olga1}. We define the aforementioned sup/inf convolutions in space:

\begin{definition}
For $ u\in C^{0,1}(D_{T})$, we define the $x$-sup convolution $ \overline{u}_{\theta}$ and the $x$-inf convolution $\underline{u}_{\theta}$ to be 
\begin{equation*}
\overline{u}_{\theta}(x,t)=\sup_{y}\left\{u(y,t)-\frac{|x-y|^{2}}{2\theta}\right\}\quad\text{and}\quad \underline{u}_{\theta}(x,t)=\inf_{y}\left\{u(y,t)+\frac{|x-y|^{2}}{2\theta}\right\}.
\end{equation*}
\end{definition}
If the supremum/infimum is achieved at some point, we denote that point $x^{*}$. We also define
\begin{equation*}
D_{T}^{\theta}=\left\{(x,t)\in D_{T}: \inf_{(y,s)\in \partial_{p}D_{T}} |x-y|\geq 2\theta |Du|\right\}.
\end{equation*}
We now list some standard properties pertaining to the $x$-sup/inf convolutions. 
\begin{proposition}[Turanova, \cite{olga1}]
Let $u\in C^{0,1}(D_{T})$. 
\begin{enumerate}
\item $|x-x^{*}|\leq 2\theta|Du|$
\item For $(x,t)\in D^{\theta}_{T}$, $\overline{u}_{\theta}(x,t)\leq u(x,t)+2\theta(|Du|)^{2}$ and $\underline{u}_{\theta}(x,t)\geq u(x,t)-2\theta(|Du|)^{2}.$
\item $\overline{\underline{u}}_{\theta}$ is twice differential in $x$, a.e. in $D_{T}^{\theta}$; and $D^{2}_{x}\overline{u}_{\theta}\geq -\theta^{-1}Id$ and $D^{2}_{x}\underline{u}_{\theta}\leq \theta^{-1}Id$ in the sense of distributions.
\item If $u$ is a subsolution, (resp. super solution) to $u_{t}-F(D^{2}u)=c$ in $D_{T}$, then $\overline{u}_{\theta}$ (respectively $\underline{u}_{\theta}$ is a sub (respectively super) solution in $D^{\theta}_{T}$. 
\end{enumerate}
\end{proposition}

Moreover, we need the following proposition which says that upon regularizing function a function $u(x,t)$ in $x$, $u$ may be touched by paraboloids in certain sets. This is the main ingredient of Claim \ref{claimcube}. We define the sets
\begin{equation*}
Q^{+, \theta}_{r}(x,t)=B(x, r-\theta\norm{Du}_{\infty})\times (t, t+r^{2}]
\end{equation*}
and
\begin{equation*}
K^{+,\theta}_{r}(x,t)=\left[x-\left(\frac{r}{9\sqrt{d}}-\theta\norm{Du}_{\infty}\right), x+\left(\frac{r}{9\sqrt{d}}-\theta\norm{Du}_{\infty}\right)\right]^{d}\times  \left(t, t+\frac{r^{2}}{81d}\right]
\end{equation*}

\begin{proposition}[Turanova, \cite{olga1}]
Assume $F(0)=0$. Suppose $u\in C^{0,1}(D_{T})$ is a solution to $u_{t}+\overline{F}(D^{2}u)=0$. Assume $\overline{Q}^{\theta}_{r}(\overline{x},\overline{t})\subset D_{T}$, and $dist(Q^{+, \theta}_{r}(\overline{x},\overline{t}), \partial_{p}D_{T})\geq \theta$. 
\begin{enumerate}
\item If $(x_{1}, t_{1}), (x_{2}, t_{2})\in K^{+,\theta}_{r}(x,t)$, then $(x_{1}^{*}, t), (x_{2}^{*}, t)\in K^{+,\theta}_{r}(x,t)$ and 
\begin{equation*}
|x_{1}-x_{2}|\leq C\theta^{-1}(1+\norm{u}_{L^{\infty}(D_{T})})|x_{1}^{*}-x_{2}^{*}|
\end{equation*}
\item There are universal constants $M_{0}, \sig,$ and $C$ such that for all $M\geq M_{0}$, there exists a set 
\begin{equation*}
\overline{\underline{\Theta}}_{M}^{\theta}(\overline{\underline{u}}_{\theta}, Q^{+,\theta}_{r}(\overline{x},\overline{t}))\subset K^{+,\theta}_{r}(\overline{x},\overline{t})
\end{equation*}
 such that for any $(x_{0},t_{0})\in\overline{\underline{\Theta}}_{M}^{\theta}(\overline{\underline{u}}_{\theta}, Q^{+,\theta}_{r}(\overline{x},\overline{t}))$, there exists a polynomial $P=c+\ell\cdot x+m\cdot t+xQx^{T}$ where $c, \ell, m, a\in \RR$, and $Q\in \mathbb{S}^{d}$ such that $P_{t}(x_{0},t_{0})-\overline{F}(D^{2}P(x_{0},t_{0}))=0.$ For all $(y,s)\in \overline{Q}_{r}^{\theta}(\overline{x},\overline{t})\cap \left\{s\leq t_{0}\right\}$, 
\begin{align*}
\overline{\underline{u}}_{\theta}(y,s)-\overline{\underline{u}}_{\theta}(x_{0}, t_{0})&\geq (\leq) P(y-x_{0},s-t_{0})+\frac{C(d)}{r^{2}}M(|y-x_{0}|^{3}+|y-x_{0}|^{2}|t_{0}-s|+|t_{0}-s|^{2})\\
&-\frac{C(1+\norm{u}_{L^{\infty}(D_{T})})}{\theta^{2}}|y-x_{0}|(t_{0}-s).
\end{align*}
Moreover, 
\begin{equation*}
\left|C^{+, \theta}_{r}(\overline{x}, \overline{t})\setminus \overline{\underline{\Theta}}^{\theta}_{M}(\overline{\underline{u}}_{\theta}, Q^{+,\theta}_{r}(\overline{x},\overline{t}))\right|\leq \frac{Cr^{d+1}}{\theta^{2}M^{\sigma}}(\norm{u}_{L^{\infty}(D_{T})}+1)^{1+\sigma}\left(1+\frac{r}{\theta}\right).
\end{equation*}
\end{enumerate}
\end{proposition}

\subsection{A brief summary of the proof of Claim \ref{claimcube}}
We next explain the proof of the Claim \ref{claimcube}. As the argument is technical, and follows the same reasoning as Lemma 4.2 of \cite{olga1}, we omit the details. We first fix $\om\in\Om$ such that $|\overline{u}^{\delta, \theta}-\underline{u}^{\ve}_{\theta}(\cdot, \cdot, \om)|\geq C \delta^{\al}$. Using the parabolic ABP-estimate for $\overline{u}^{\delta, \theta}-\underline{u}^{\ve}_{\theta}$ and a covering argument, one may show that $\overline{\Theta}^{\theta}_{M}$ and the convex envelope of $u^{+\delta}-u^{\ve}$ must intersect in one of the cubes of the grid. Therefore, there exists a point where $u^{+\delta}$ has a second order expansion from above, and $u^{+\delta}-u^{\ve}$ is convex there. By perturbing the polynomial approximation, we may lower it so that it touches $u^{\ve}$ from above, and stays above $u^{\ve}$ in $Q^{+}_{r}(x_{i}, t_{i})$.

\bibliographystyle{amsplain}
\bibliography{parrate}

\end{document}